\documentclass[journal,onecolumn]{IEEEtran}
\usepackage{cancel}
\usepackage{ulem}
\usepackage[T1]{fontenc}
\usepackage{tgpagella}
\usepackage{graphicx,amsmath,amsfonts,amscd,amssymb,bm,url,color,latexsym}
\usepackage[numbers]{natbib}
\usepackage{fullpage}
\usepackage[small,bf]{caption}
\usepackage[outdir=./]{epstopdf}
\usepackage{subcaption}
\usepackage{bbm}
\usepackage{microtype}
\usepackage[ruled]{algorithm2e}
\usepackage{algpseudocode}
\usepackage{xr-hyper}
\usepackage{hyperref,cleveref}
\usepackage{verbatim}
\usepackage{framed}
\usepackage{comment}
\usepackage{tikz}
\usepackage{fge}
\usepackage{mathdots}
\usepackage{mathabx}
\usepackage{breqn}
\usepackage[margin=0.95in]{geometry}
\usepackage{enumitem}
\usepackage{amssymb}
\usepackage{amsthm}

\makeatletter
\newcommand*{\addFileDependency}[1]{% argument=file name and extension
  \typeout{(#1)}
  \@addtofilelist{#1}
  \IfFileExists{#1}{}{\typeout{No file #1.}}
}
\makeatother

\newcommand*{\myexternaldocument}[1]{%
    \externaldocument{#1}%
    \addFileDependency{#1.tex}%
    \addFileDependency{#1.aux}%
}

\myexternaldocument{writeup/test.tex}

\allowdisplaybreaks
\numberwithin{equation}{section}
\hypersetup{
    colorlinks=true,%
    citecolor=blue,%
    filecolor=blue,%
    linkcolor=blue,%
    urlcolor=blue
}
\newtheorem{theorem}{Theorem}[section]
\newtheorem{lemma}[theorem]{Lemma}

\newtheorem{proposition}[theorem]{Proposition}

\newtheorem{assumption}[theorem]{Assumption}

\newtheorem{remark}[theorem]{Remark}

\newcommand{\mb}{\boldsymbol}
\newcommand{\mc}{\mathcal}

\newcommand{\bb}{\mathbb}

\newcommand{\E}{\bb E}

\newcommand{\norm}[2]{\left\| #1 \right\|_{#2}}
\newcommand{\abs}[1]{\left| #1 \right|}
\newcommand{\innerprod}[2]{\left\langle #1,  #2 \right\rangle}

\newcommand{\R}{\bb R}

\newcommand{\Sp}{\bb S}

\renewcommand{\P}{\mathbb{P}}

\newcommand{\Brac}[1]{\left\{#1 \right\} }
\newcommand{\brac}[1]{\left[ #1 \right] }
\newcommand{\paren}[1]{ \left( #1 \right) }

\DeclareMathOperator{\diag}{diag}

\DeclareMathOperator{\rank}{rank}

\DeclareMathOperator{\grad}{grad}
\DeclareMathOperator{\Hess}{Hess}

\DeclareMathOperator{\st}{s.t.}

\newcommand{\wt}{\widetilde}

\newcommand{\bY}{\mb Y}
\newcommand{\bX}{\mb X}
\newcommand{\bA}{\mb A}
\newcommand{\bq}{\mb q}
\newcommand{\ba}{\mb a}
\newcommand{\bI}{\mb I}
\newcommand{\zinf}{\|\mb \zeta\|_{\infty}}

\newcommand{\bP}{\mb P}
\newcommand{\bu}{\mb u}
\newcommand{\bv}{\mb v}
\newcommand{\bD}{\mb D}
\newcommand{\bU}{\mb U}
\newcommand{\bV}{\mb V}
\newcommand{\bE}{\mb E}
\newcommand{\bLbd}{\mb \Lambda}
\newcommand{\op}{\textrm{op}}

\newcommand{\bx}{\mb x}
\newcommand{\cE}{\mathcal E}

\newcommand{\0}{\mb 0}
\newcommand{\oA}{\bar{\bA}}
\newcommand{\oX}{\bar{\bX}}

\title{Unique sparse decomposition of low rank matrices}

\begin{document}

\title{Unique sparse decomposition of low rank matrices}
\author{Dian Jin,\thanks{Dian Jin is with the Department of Electrical and Computer Engineering, Rutgers University, New Brunswick, Piscataway, NJ, 08854 USA(email: dj370@scarletmail.rutgers.edu)}
        Xin Bing \thanks{Xin Bing is with the Department of Statistical Sciences at the University of Toronto, Toronto, ON M5G 1Z5, Canada(email: xin.bing@utoronto.ca)}
        and Yuqian Zhang\thanks{Yuqian Zhang is with the Department of Electrical and Computer Engineering, Rutgers University, New Brunswick, Piscataway, NJ, 08854 USA(email: yqz.zhang@rutgers.edu)}% <-this % stops a space
% \thanks{M. Shell is with the Department
% of Electrical and Computer Engineering, Georgia Institute of Technology, Atlanta,
% GA, 30332 USA e-mail: (see http://www.michaelshell.org/contact.html).}% <-this % stops a space
% \thanks{J. Doe and J. Doe are with Anonymous University.}% <-this % stops a space
%\thanks{Manuscript received April 19, 2005; revised September 17, 2014.}
}
\maketitle
\begin{abstract}
 The problem of finding a unique low-dimensional decomposition of a given matrix has been a fundamental and recurrent problem in many areas. In this paper, we study the problem of seeking a unique decomposition of a low-rank matrix $\mb Y\in \R^{p\times n}$ that admits a sparse representation. Specifically, we consider $\bY = \bA \bX$ where the matrix $\bA\in \R^{p\times r}$ has full column rank, with $r < \min\{n,p\}$, and the matrix $\bX\in \R^{r\times n}$ is element-wise sparse. We prove that this low-rank, sparse decomposition of $\bY$ can be uniquely identified, up to some intrinsic signed permutation. Our approach relies on solving a non-convex optimization problem constrained over the unit sphere. Our geometric analysis for its non-convex optimization landscape shows that any {\em strict} local solution is close to the ground truth, and can be recovered by a simple data-driven initialization followed with any second-order descent algorithm. Our theoretical findings are corroborated by numerical experiments.\footnote{Codes to reproduce all experiment results are available in \url{https://github.com/Jindiande/Unique_Fac_of_Low_Rank}}% <-this % stops a space.
  %consider the case with $\mb A_0\in\R^{d\time r}$ to be under-complete,  namely $d\ge k$. We cast the problem as a problem of recovering non-zero sparse vector over the sphere. 
\end{abstract}


\section{Introduction}
The problem of matrix decomposition has been a popular and fundamental topic under extensive investigations across several disciplines, including signal processing, machine learning, natural language processing \cite{blei2003latent, candes2011robust, lin2016review, zhang2013toward, mu2016scalable, bing2020topicmodels}, just to name a few. From the decomposition, one can construct useful representations of the original data matrix. 
However, for any matrix $\bY \in \R^{p\times n}$ that can be factorized as a product of two matrices $\bA\in \R^{p\times r}$ and $\bX \in \R^{r\times n}$, there exist infinitely many decompositions, simply because one can use any $r\times r$ invertible matrix $\mb Q$ to construct $\bA' = \bA\mb Q$ and $\bX' = \mb Q^{-1}\bX$ such that $\bY = \bA \bX = \bA'\bX'$, while $\bA' \ne \bA$ and $\bX'\ne \bX$. Thus, 
% Given a low rank matrix $\bY\in\R^{p\times n}$ with $\rank(\bY) = r < \min\{n,p\}$, there exist infinitely many decompositions of $\bY$ as a product of two matrices $\bA\in \R^{p\times r}$ and $\bX \in \R^{r\times n}$, simply because one can use any $r\times r$ invertible matrix $\mb Q$ to construct $\bA' = \bA\mb Q$ and $\bX' = \mb Q^{-1}\bX$ such that $\bY = \bA \bX = \bA'\bX'$, while $\bA' \ne \bA$ and $\bX'\ne \bX$. %This leads to various low-dimensional representations of $\bY$ \cite{???, fan2020polynomial, ding2020low}. For instance, when columns of $\bY$ correspond to $n$ samples of a $p$-dimensional feature, one has many ways of constructing a $r$-dimensional representation of the $n$ samples.
in various applications, additional structures and priors are imposed for a unique representation \cite{haeffele2014structured,fan2020polynomial}. For example, principal component analysis (PCA) aims to find orthogonal representations which retain as many variations in $\bY$ as possible \cite{Karl1901,Hotelling1936}, whereas independent component analysis (ICA) targets the representations of statistically independent non-Gaussian signals \cite{HyvarinenO00}. 

In this paper, we are interested in finding a unique, {\em sparse}, low-dimensional representation of $\bY$. To this end, we study the decomposition of a low-rank matrix $\mb Y\in\R^{p\times n}$ that satisfies
\begin{align}
    \label{eqn:model}
    \mb Y = \mb A\mb X,
    \end{align}
where $\mb A\in\R^{p\times r}$ is an unknown deterministic matrix, with $r < \min\{n,p\}$, and $\mb X\in\R^{r\times n}$ is an unknown sparse %coefficient 
matrix. 

Finding a unique sparse decomposition of (\ref{eqn:model}) turns out to be important in many applications. For example, if the matrix $\bX$ is viewed as a low-dimensional sparse representation of $\bY$, finding such representation is ubiquitous in signal recovery, image processing and compressed sensing, just to name a few. If columns of $\bY$ are viewed as linear combinations of columns of $\bA$ with $\bX$ being the sparse coefficient, then (\ref{eqn:model}) can be used to form overlapping clusters of the $n$ columns of $\bY$ via the support of $\bX$ with columns of $\bA$ being viewed as $r$ cluster centers \cite{Chen12253,bing2020}.
Furthermore, we could form a $p\times r$ low-dimensional representation of $\bY$ via  linear combinations with sparse coefficients. Such sparse coefficients greatly enhance the interpretability of the resulting representations \cite{Jolliffe,Gravuer6344,Baden16}, in the same spirit as the sparse PCA, but generalizes to the factorization of non-orthogonal matrices. 

To motivate our approach, consider the simple case that $\bA$ has orthonormal columns, namely, $\bA^T \bA = \bI_r$\footnote{$\bI_r$ is the identity matrix of size $r\times r$.}. % without loss of generality. Under this column orthonormality assumption, 
Then it is easy to see that the sparse coefficient matrix $\mb X$ is recovered by multiplying $\mb Y$ on the left by $\mb A^T$,
\begin{align}
    \mb A^T\mb Y=\mb A^T\mb A\mb X=\mb X.
\end{align}
The problem of finding such orthonormal %sparsifying 
matrix $\mb A$ boils down to successively finding a unit-norm direction $\bq$ that renders $\bq^T\mb Y$ as sparse as possible % hence recovers one column of $\mb A$ 
\cite{qu2014finding,sun2016complete,qu2020geometric}, 
    \begin{align}\label{obj_min}
        \min_{\bq}\quad\norm{\bq^T\mb Y}{\rm sparsity} \qquad
        \st \quad \norm{\bq}2=1.
    \end{align}
However, %solving the above problem when $\bY$ has low rank (or, more accurately, when $\bA$ has full column rank)
the natural choice of the sparsity penalty, such as the $\ell_p$ norm for $p\in \left[0,1\right]$, 
% either $\ell_0$ or $\ell_1$, 
 leads to trivial and meaningless solutions, as there always exists $\bq$ in the null space of $\mb A^T$ such that $\bq^T\mb Y=\mb 0$. 

To avoid the null space of $\bA^T$,
%In order to distinguish a sparsifying direction and a null-space direction, 
we instead choose to find the unit direction $\bq$ that maximizes the $\ell_4$ norm of $\bq^T\bY$ as 
\begin{equation}\label{intro_obj}
    \max_{\bq} \quad \norm{\bq^T\bY}4 \qquad \st \quad \norm{\bq}2=1.
\end{equation}
The above formulation is
based on the key observation that the objective value is maximized when $\bq$ coincides with one column of $\bA$ (see, Section \ref{sec_method}, for details) while
%sparsifying direction renders large $\ell_4$ norm and 
the objective value is zero when $\bq$ lies in the null space of $\bA^T$. The $\ell_4$ norm objective function and its variants have been adopted as a sparsity regularizer in a line of recent works \cite{li2018global, zhang2018structured, zhai2020complete, qu2020geometric, zhai2020Understanding},  arguably because solving (\ref{intro_obj}) requires a milder restriction on  
sparsity level of $\bX$ to recover $\bA$  comparing to solving (\ref{obj_min}).
However, all previous works study the setting where $\bA$ has full row rank and, to the best of our knowledge, the setting where $\bA$ has full column rank has not been studied elsewhere. As we will elaborate below, when $\bA$ has full column rank, analysis of the optimization landscape of (\ref{intro_obj}) becomes more difficult since the null space of $\bA^T$ persists as a challenge for solving the optimization problem: they form a flat region of locally optimal solutions.

This paper characterizes the nonconvex optimization landscape of (\ref{intro_obj}) and proposes a guaranteed procedure that avoids the flat null region and provably recovers the global solution to (\ref{intro_obj}), which corresponds to one column of $\bA$. More specifically, we demonstrate that, despite the non-convexity, (\ref{intro_obj}) still possesses benign geometric property in the sense that any {\em strict} local solution with {\em large} objective value is globally optimal and recovers one column of $\bA$, up to its sign. See, Theorem \ref{thm:obj_pop} in Section \ref{sec_theory_orth} for the population level result and Theorem \ref{thm:obj_sample} for the finite sample result. 

We further extend these results to the general case when $\bA$ only has full column rank in Theorem \ref{thm:obj_gene_sample} of Section \ref{sec:theo_obj_gene}. To recover a general $\bA$ with full column rank, our procedure first resorts to a preconditioning procedure of $\bY$ proposed in Section \ref{sec_gene} and then solves an optimization problem similar to (\ref{intro_obj}). 
From our analysis of the optimization landscape, the intriguing problem boils down to developing algorithms to recover nontrivial local solutions by avoiding regions with small objective values. We thus propose a simple initialization scheme in Section \ref{sec:init} and prove in Theorem \ref{thm:complete} that such initialization, proceeded with any second-order descent algorithm \cite{goldfarb1980curvilinear, jin2017escape}, suffices to find the global solution, up to some statistical error. Our theoretical analysis provides explicit rates of convergence of the statistical error and characterizes the dependence on various dimensions, such as $p$, $r$ and $n$, as well as the sparsity of $\bX$.

Numerical simulation results are provided in Section \ref{sec:exp}. Finally, in Section \ref{sec:concl}, we conclude our results and discuss several future directions of our work. All the proofs and supplementary simulations are deferred to the Appendix.

\paragraph{Notations} Throughout this paper, we use bold lowercase letters, like $\mb a$, to represent vectors and bold uppercase letters, like $\mb A$, to represent matrices. For matrix $\mb X$, $\mb X_{ij}$ denotes the entry at the $i$-th row and $j$-th column of $\mb X$, with $\mb X_{i\cdot}$ and $\mb X_{\cdot j}$ denoting the $i$-th row and $j$-th column of $\mb X$, respectively. Oftentimes, we write $\mb X_{\cdot j} = \mb X_j$ for simplicity. We use %$\nabla$ and $\nabla^2$ to denote the Euclidean gradient and Hessian, and use 
$\grad$ and $\Hess$ to represent the Riemannian gradient and Hessian. For any vector $\bv\in \R^d$, we use $\|\bv\|_q$ to denote its $\ell_q$ norm, for $1\le q\le \infty$. The notation $\bv^{\circ q}$ stands for $\{\bv_i^q\}_i$. For matrices, we use $\norm{\cdot}{F}$ and $\|\cdot\|_{\op}$ to denote the Frobenius norm and the operator norm, respectively. For any positive integer $d$, we write $[d] = \{1,2,\ldots, d\}.$ The unit sphere in $d$-dimensional real space $\R^d$ is written as $\Sp^{d-1}$. 
For two sequences $a_n$ and $b_n$, we write $a_n \lesssim b_n$ if there exists some constant $C>0$ such that $a_n \le C b_n$ for all $n$. 
Both uppercase $C$ and lowercase $c$ are reserved to represent numerical constants, whose values may vary line by line.

\subsection{Related work}
Finding the unique factorization of a matrix is an ill-posed problem in general due to infinitely many solutions. 
    % Additional structures on $\bA$ and $\bX$ are needed to ensure there exists only one type of such factorization. 
    %This renders the natural formulation of the decomposition problem ineffective and identifiability of the ground truth is possible up to above inherent ambiguity \cite{zhang2020symmetry}. 
  %One common type of structure assumes sparsity on either $\bA$ or $\bX$. 
  There exist several strands of studies from different contexts on finding the unique decomposition of $\bY$ 
  by imposing additional structures on $\bA$ and $\bX$. We start by reviewing the literature which targets the sparse decomposition of $\bY$.

\paragraph{Dictionary learning}The problems of dictionary learning (DL) \cite{agarwal2013exact, spielman2013exact, geng2014local,sun2016complete} and sparse blind deconvolution or convolutional dictionary learning \cite{cheung2020dictionary, kuo2019geometry} study the unique decomposition of $\bY = \bA\bX$ where $\bX$ is sparse and $\bA$ has full row rank. In this case, the row space of $\mb Y$ lies in the row space of $\mb X$, suggesting to recover the sparse rows of $\bX$ via solving the following problem, % hence the recovery problem can be cast as
    \begin{align}
        \min_{\bq}\quad\norm{\mb q^T\bY}1 \qquad
        \st \quad \mb q \ne 0.
    \end{align}
%There, in view of the sparsity of $\bX$, the problem of decomposing $\bY$ is cast as recovering the sparse rows of $\mb X$ over the column space of $\mb A$. Specifically, one seeks the direction $\mb v$ to minimize the sparsity of $\bY^T\bv$ via solving the following problem,
Under certain scaling and incoherence conditions on $\bA$, the objective achieves the minimum value when $\mb q$ is equal to one column of $\mb A$, at the same time $\mb q^T\bY$ recovers one sparse row of $\bX$. 
This idea has been considered and generalized in a strand of papers when $\bA$ has full row rank \cite{spielman2013exact, sun2016complete,zhang2017global,li2018global, zhang2018structured,qu2020geometric,zhai2020Understanding, shi2020manifold,zhang2020symmetry}. In our context, the major difference rises in the matrix $\bA$, which has full {\em column} rank rather than  {\em row} rank, therefore  minimizing $\norm{\mb q^T\bY}1$ as before only leads to some vector in the null space of $\bA^T$, yielding the trivial zero objective value.  %In another word, we are interested in finding the unique decomposition of a {\em low-rank} matrix $\bY$. Since $\bA$ has full column rank,

 On the other hand,  \cite{qu2020geometric, zhai2020complete} consider the same objective function in (\ref{intro_obj}) to study the problem of overcomplete/complete dictionary learning (where $\bA$ has full row rank). However the optimization landscape when $\bA$ has full column rank is significantly different from that in the (over)complete setting. The more complicated optimization landscape in our setting brings additional difficulty of the analysis and requires a proper initialization in our proposed algorithm. We refer to Appendix \ref{app_related} for detailed technical comparison with \cite{qu2020geometric, zhai2020complete}.

\paragraph{Sparse PCA} Sparse principal component analysis (SPCA) is a popular method that recovers a unique decomposition of a low-rank matrix $\bY$ by utilizing the sparsity of its singular vectors. However, as being said, under $\bY =\bA\bX$, SPCA is only applicable when $\bX$ coincides with the right singular vectors of $\bY$. Indeed, one formulation of SPCA is to solve
\begin{align}
    &\max_{\bU \in \R^{n\times r}}\quad  \text{tr}\left(\bU^T\bY^T\bY\bU\right) - \lambda \|\bU\|_1,\\
    &\st \quad \bU^T\bU = \bI_r,\nonumber
\end{align}
which is promising only if $\bX$ corresponds to the right singular vectors of $\bY$.
It is worth mentioning that among the various approaches of SPCA, the following one might be used to recover one sparse row of $\bX$, 
    \begin{align}
        \min_{\bu, \bv} ~\norm{\bY - \bu \bv^T}2^2 ~  + ~ \lambda \|\bv\|_1 \quad
        \st \quad \|\bu\|_2=1.
    \end{align}
This procedure was originally proposed by \cite{Zou06spca} and \cite{ShenHuangspca} together with an efficient  algorithm by alternating the minimization between $\bu$ and $\bv$. %is also proposed to solve the optimization. 
%for recovering one sparse singular vector of $\bY$. \cite{ShenHuangspca} also provides an efficient algorithm to solve  the above problem 
However, there is no guarantee that the resulting solution recovers the ground truth.

% \paragraph{ICA}One type of structure is to assume that the columns of $\bX$ are realizations of independent and non-Gaussian random vectors. Independent component analysis (ICA) \cite{HyvarinenO00} is proposed to capture such independent non-Gaussian signal and recovers the mixing matrix $\bA$. However, as ICA heavily relies on the non-Gaussianity, it is not applicable when the columns of $\bX$ follow Gaussian distributions, a common situation of interests in practice. 

\paragraph{Factor analysis} Factor analysis is a popular statistical tool for constructing low-rank representations of $\bY$ by postulating 
\begin{align}
   \bY = \bA\bX + \bE 
\end{align}
where $\bA \in \R^{p\times r}$ is the so-called loading matrix with $r = \rank(\bA) < \min\{n,p\}$, $\bX \in \R^{r\times n}$ contains $n$ realizations of a $r$-dimensional factor and $\bE$ is some additive noise. Factor analysis is mainly used to recover the low-dimension column space of $\bA$ or the row space of $\bX$, rather than to identify and recover the unique decomposition. Recently, \cite{bing2020} studied the unique decomposition of $\bY$ when the columns of $\bX$ are i.i.d. realizations of a $r$-dimensional latent random factor. The unique decomposition is further used for (overlapping) clustering the rows of $\bY$ via the assignment matrix $\bA$. To uniquely identify $\bA$, \cite{bing2020} assumes that $\bA$ contains at least one $r\times r$ identity matrix, coupled with other scaling conditions on $\bA$ (we refer to \cite{bing2020} for detailed discussions of other existing conditions in the literature of factor models that ensure the unique decomposition of $\bY$ but require strong prior information on either $\bA$ or $\bX$). By contrast, we rely on the sparsity of $\bX$, which is more general than requiring the existence of a $r\times r$ identity matrix in $\bA$.

% Although there are existing identifiability conditions under which there exists a unique pair of $\bA$ and $\bX$ satsifying $\bY = \bA\bX$, the required conditions are unnecessarily strong (we refer to \cite{bing2020} for detailed discussions of other existing conditions in the literature of factor models that ensure the unique decomposition of $\bY$ but require strong prior information on either $\bA$ or $\bX$). 
%\paragraph{Clustering} 
\paragraph{NMF and topic models} 
Such existence condition of identity matrix in either $\bA$ or $\bX$ has a variant in non-negative matrix factorization (NMF) \cite{donoho04} and topic models \cite{arora13,bing2020topicmodels,bing2020sparsetopicmodels}, also see the references therein, where $\bY$, $\bA$ and $\bX$ have non-negative entries. Since all $\bY$, $\bA$ and $\bX$ from model (\ref{eqn:model}) are allowed to have arbitrary signs in our context, the approaches designed for NMF and topic models are inapplicable.

%\paragraph{Nongaussianity in Independent Component Analysis} Note that maximising \textit{kurtosis} has been widely used in encouraging nongaussianity in \textit{Independent Component Analysis}(ICA)\cite{hyvarinen1997fast,decarlo1997meaning}. More specifically, given the observed random vector $\mb y\in \mathbb{R}^{p}$, ICA want to find a weight unit vector $\mb w\in \mathbb{S}^{p-1}$ to project $\mb w^{T}\mb y$ with maximal nongaussianity which is typically measured by \textit{kurtosis} that defined as
%\begin{align}
%    \textrm{kurt}\left(\mb w^{T}\mb y\right)=\mathbb{E}\left[\left(\mb w^{T}\mb y\right)^4\right]-3\left\{\mathbb{E}\left[\left(\mb w^{T}\mb y\right)^2\right]\right\}^2.
%\end{align}
%One could find the \textit{kurtosis} is similar to our object function. However, the biggest difference is that ICA requires at most one gaussian component in independent source $\mb x$ \cite{comon1994independent,hyvarinen1999survey} which is not applicable under our assumptions that $\mb X$ follows i.i.d Bernoulli-Gaussian distribution.  

\section{Formulation and Assumptions}\label{sec_method}

The decomposition of $\mb Y = \mb A\mb X$ is not unique  without further assumptions. To ensure the uniqueness of such decomposition, we rely on two assumptions on the matrices $\bA$ and $\bX$, stated in Section \ref{sec_assump}. 

Our goal is to uniquely recover $\bA$ from $\bY$, up to some signed permutation. More precisely, we aim to recover columns of $\bA \bP$ for some signed permutation matrix $\bP\in \R^{r\times r}$. To facilitate understanding, in Section \ref{sec_orth} we first state our procedure for uniquely recovering $\bA$ when $\bA$ has orthonormal columns. Its theoretical analysis is presented in Section \ref{sec_theory}. Later in Section \ref{sec_gene}, we discuss how 
to extend our results to the case where $\bA$ is a general full column rank matrix under Assumption \ref{ass_A_rank}.  

For now, we only focus on the recovery of one column of $\bA$ as the remaining columns can be recovered via the same procedure after projecting $\bY$ onto the complement space spanned by the recovered columns of $\bA$ (see Section \ref{sec:deflation} for detailed discussion). 

\subsection{Assumptions}\label{sec_assump}
We first resort to the matrix $\mb X\in\R^{r\times n}$ being element-wise sparse. The sparsity of $\mb X$ is modeled via the Bernoulli-Gaussian distribution, stated in the following assumption.
\begin{assumption}\label{ass_X}
  Assume $\mb X_{ij} = \mb B_{ij}\mb Z_{ij}$ for $i\in [r]$ and $j\in[n]$, where 
  \begin{align}
        &\mb B_{ij} \overset{i.i.d.}{\sim} \textrm{Ber}(\theta),\\
        &\mb Z_{ij}\overset{i.i.d.}{\sim} \mc N(0, \sigma^2).
  \end{align}
  
\end{assumption}
The Bernoulli-Gaussian distribution is popular for modeling sparse random matrices \cite{spielman2013exact,agarwal2013exact,sun2016complete}. The overall sparsity level of $\mb X$ is controlled by $\theta$, the parameter of the Bernoulli distribution. We remark that the Gaussianity is assumed only to simplify the proof and to obtain more transparent deviation inequalities between quantities related with $\mb X$ and their population counterparts. Both our approach and analysis can be generalized to cases where $\mb Z_{ij}$ are centered i.i.d. sub-Gaussian random variables.

We also need another condition on the matrix $\bA$. To see this, note that even when $\bA$ were known, recovering $\bX$ from $\bY = \bA\bX$ requires $\bA$ to have full column rank. We state this in the following assumption.
\begin{assumption}\label{ass_A_rank}
  Assume the matrix $\bA \in \R^{p\times r}$ has $\rank(\bA) = r$ with $\|\bA\|_{\op} = 1$.
\end{assumption}

\noindent The unit operator norm of $\bA$ is assumed without loss of generality as one can always re-scale $\sigma^2$, the variance of $\bX$, by $\|\bA\|_{\op}$.

\subsection{Recovery of the orthonormal columns of $A$}\label{sec_orth}

In this section, we consider the recovery of one column of $\bA$ when $\bA$ is a semi-orthogonal matrix satisfying the following assumption. 

\begin{assumption}\label{ass_A_orth}
  Assume 
  $
    \bA^T \bA = \bI_r
  $.
\end{assumption}

Our approach recovers columns of $\bA$ one at a time by adopting the $\ell_4$ maximization in (\ref{obj}) to utilize the sparsity of of $\bX$ and orthogonality of $\bA$. Its rationale is based on the following lemma. 

%simultaneously avoids the trivial all zero solution and encourages sparsity to recover one ground truth column.

\begin{lemma}\label{lem_key}
  Under Assumption \ref{ass_A_orth}, solving the following problem
  \begin{align}\label{crit_ell_4}
      \max_{\bq} \quad \norm{\bA^T \bq}4^4 \qquad
      \st\quad\norm{\mb q}2=1
  \end{align}
  recovers one column of $\bA$, up to its sign. 
\end{lemma}
%\begin{proof}
%The proof is deferred to Section \ref{app_proof_lem_key} of the Appendix.
%\end{proof}

Intuitively, under Assumption \ref{ass_A_orth}, we have $\|\bA^T \bq\|_2 \le 1$ for any unit vector $\bq$. Therefore, criterion (\ref{crit_ell_4}) seeks a vector $\bA^T \bq$ within the unit ball to maximize its $\ell_4$ norm.  When $\bq$ corresponds to one column of $\bA$, that is, $\bq = \ba_i$ for any $i\in [r]$, we have the largest objective 
$\|\bA^T \ba_i\|_4^4 = 1$. This $\ell_4$ norm maximization approach has been used in several related literatures, for instance, sparse blind deconvolution \cite{zhang2018structured, li2018global}, complete and over-complete dictionary learning \cite{zhai2020complete,zhai2020Understanding,qu2020geometric}, independent component analysis \cite{hyvarinen1997fast,hyvarinen1997family} and tensor decomposition \cite{ge2020optimization}. 

The appealing property of maximizing the $\ell_4$ norm is its benign geometry landscape under the unit sphere constraint. Indeed, despite of the non-convexity of (\ref{crit_ell_4}), our result in Theorem \ref{thm:obj_pop} implies that any strict location solution to (\ref{crit_ell_4}) is globally optimal. This enables us to use any second-order gradient accent method to solve (\ref{crit_ell_4}).

Motivated by Lemma \ref{lem_key}, since we only have access to $\bY \in \R^{p\times n}$, we propose to solve the following problem to recover one column of $\mb A$,
\begin{align}\label{obj}
    \min_{\bq} ~ F(\bq)  \doteq   -\frac1{12 \theta \sigma^4 n} \norm{\mb Y^T \mb q}4^4\quad %= -\frac1{12 \theta \sigma^4 n}\norm{\mb q^{T}\mb A\mb X}4^4\\
    \st~\norm{\mb q}2=1.
\end{align}
The scalar $(12 \theta \sigma^4 n)^{-1}$ is a normalization constant. The following lemma justifies the usage of (\ref{obj}) and also highlights the role of the sparsity of $\bX$. 

 \begin{lemma}\label{lem:obj}
     Under model (\ref{eqn:model}) and Assumption \ref{ass_X}, we have 
     \begin{equation*}
          \E\brac{ F(\mb q)}  =  -{1\over 4}\brac{(1-\theta)\norm{\mb A^T \mb q}4^4 + \theta \norm{\mb A^T \mb q}2^4}
     \end{equation*}
     where the expectation is taken over the randomness of $\bX$.
 \end{lemma}
%\begin{proof}
%    The proof is deferred to Section \ref{app_proof_lem_obj} of the Appendix.
%\end{proof}

\begin{remark}[Role of the sparsity parameter $\theta$]\label{rem_theta}
    Lemma \ref{lem:obj} implies that, for large $n$, solving (\ref{obj}) approximately finds the solution to 
    \begin{align}\label{obj_pop}\nonumber
    &\min_{\bq} \quad f\paren{\mb q} \doteq -{1\over 4}\brac{(1-\theta)\norm{\mb A^T \mb q}4^4 +\theta \norm{\mb A^T \mb q}2^4}\\
    &\st\quad\norm{\bq}2=1
    \end{align}
    The objective function is a convex combination of $\|\mb A^T \mb q\|_4^4$ and
    % \sout{$\|\mb A^T \mb q\|_2^2$}
    $\|\mb A^T \mb q\|_2^4$ with coefficients depending on the magnitude of $\theta$. In view of Lemma \ref{lem_key}, it is easy to see that solving (\ref{obj_pop}) recovers one column of $\bA$, up to the sign, as long as $\theta < 1$. However, the magnitude of $\theta$ controls the benignity of the geometry landscape of (\ref{obj_pop}). When $\theta$ is small, equivalently, $\mb X$ is sufficiently sparse, we essentially solve (\ref{crit_ell_4}) which has the most benign landscape. On the other hand, when $\theta \to 1$, the landscape of (\ref{obj_pop}) is mostly determined by the eigenvalue problem\footnote{When $\mb A$ is orthonormal, this eigenvalue problem possesses the worst landscape as there are infinitely many solutions.} which maximizes $\|\bA^T\bq\|_2$ subject to $\|\bq\|_2=1$. We will demonstrate that when $\mb X$ is sufficiently sparse, second order descent algorithm with a simple initialization finds the globally optimal solution to (\ref{obj}) in  Section \ref{sec_theory}. %We refer to Section \ref{sec:init} for detailed discussion on choosing the initialization of (\ref{obj}).
\end{remark}

\subsection{Recovery of the non-orthogonal columns of $A$}\label{sec_gene}
 In this section, we discuss how to extend our procedure to recover $\bA$ from $\bY = \bA\bX$ when $\bA$ is a general full column rank matrix satisfying Assumption \ref{ass_A_rank}. The main idea is to first resort to a preconditioning procedure of $\bY$ such that the preconditioned $\bY$ has the decomposition $\oA \oX$, up to some small perturbation, where $\oA$ satisfies Assumption \ref{ass_A_orth} and $\oX$ satisfies Assumption \ref{ass_X} with $\sigma^2 = 1$. Then we apply our procedure in Section \ref{sec_orth} to recover $\oA$. The recovered $\oA$ is further used to recover the original $\bA$.

    To precondition $\bY$, we propose to left multiply $\bY$ by the following matrix
    \begin{equation}\label{def_D}
        \bD \doteq \brac{ 
            \paren{\bY \bY^T}^+
        }^{1/2} \in \R^{p\times p}
    \end{equation}
    where $\mb M^+$ denotes the Moore-Penrose inverse of any matrix $\mb M$. The resulting preconditioned $\bY$ satisfies
    \[
        \bar \bY \doteq \bD \bY  =  \oA\oX  + \bE 
    \]
    with $\bar\bA$ satisfying Assumption \ref{ass_A_orth}, $\oX = \bX / \sqrt{\theta n\sigma^2}$ and $\bE$ being a perturbation matrix with small entries. We refer to Proposition \ref{prop_precond} below for its precise statement. 
    
    Analogous to (\ref{obj}), we propose to recover one column of $\bar\bA$ by solving the following problem 
    \begin{align}\label{obj_gene}
        \min_{\bq}~ F_{\textrm{g}}(\bq) \doteq  -\frac{\theta n}{12} \norm{\bar \bY^T \mb q}4^4,\quad
        \st~ \norm{\mb q}2=1.
    \end{align}
    Theoretical guarantees of this procedure are provided in Section \ref{sec:theo_obj_gene}. 
    After recovering one column of $\bar\bA$, the remaining columns of $\bar \bA$ can be successively recovered via the procedure in Section \ref{sec:deflation}. In the end, $\bA$ can be recovered by first inverting the preconditioning matrix $\bD$ as 
    $
        \bD^{-1}\oA
    $
    and then re-scaling its largest singular value to 1. 
    % which are used to recover the rest  columns of $\bA$.\footnote{When $\bA$ only satisfies Assumption \ref{ass_A_rank}, $\bA$ is only identifiable up to a scalar. To fix this scalar, one could assume the largest eigenvalue of $\bA$ equal to some known constant, for instance, one.}

\section{Theoretical Guarantees}\label{sec_theory}

We provide theoretical guarantees for our procedure (\ref{obj})  in Section \ref{sec_theory_orth} when $\bA$ has orthonormal columns. The theoretical guarantees of (\ref{obj_gene}) for recovering a general full column rank $\bA$ are stated in Section \ref{sec:theo_obj_gene}.

\subsection{Theoretical guarantees for semi-orthonormal $A$}\label{sec_theory_orth}

In this section, we provide guarantees for our procedure by characterizing the solution to (\ref{obj}) when $\bA$ satisfies Assumption \ref{ass_A_orth}. 
%is semi-orthogonal in Section \ref{sec_theory_orth}. The theory for more general $\bA$ under Assumption \ref{ass_A_rank} is stated in Section \ref{sec:precondition}.
%\subsection{Guarantees for semi-orthogonal $A$}\label{sec_theory_orth}

As the objective function $F(\bq)$ in (\ref{obj}) concentrates around $f(\bq)$ in (\ref{obj_pop}), it is informative to first analyze the solution to (\ref{obj_pop}). 
Although (\ref{obj_pop}) is a nonconvex problem and has multiple local solutions,
Theorem \ref{thm:obj_pop} below guarantees that any strict local solution to (\ref{obj_pop}) is globally optimal, in the sense that, it recovers one column of $\mb A$, up to its sign. We introduce the null region $R_0$ of our objective in (\ref{obj_pop}),
\begin{equation}\label{def_R0}
    R_0 = \Brac{
        \mb q \in \bb S^{p-1}:\|\mb A^{T}\mb q\|_\infty = 0
    }.
\end{equation}

\begin{theorem}[Population case]
\label{thm:obj_pop}
  Under Assumption \ref{ass_A_orth}, assume $\theta \le 1/6$. Any local solution $\bar{\mb q}$ to (\ref{obj_pop}), that is not in $R_0$, satisfies
  \begin{align}
      \bar{\mb q} =  \mb A \bP_{\cdot 1}
  \end{align}
  for some signed permutation matrix $\bP\in \R^{r\times r}$. 
  \end{theorem} 
  
 The detailed proof of Theorem \ref{thm:obj_pop} is deferred to %Section \ref{app_proof_thm_obj_pop} of 
 Appendix \ref{app_proof_thm_obj_pop}.
 We only offer an outline of our analysis below. 
 
 The proof of Theorem \ref{thm:obj_pop} relies on analyzing the optimization landscape of (\ref{obj_pop}) on disjoint partitions of $\bb S^{p-1}= \{\mb q \in \R^p: \|\mb q\|_2 = 1\}$, defined as
     \begin{align}\label{def:obj}
    R_1 &\doteq R_1(C_{\star})=\Brac{\mb q \in \bb S^{p-1}:\|\mb A^{T}\mb q\|_\infty^{2}\geq  C_{\star} },\\\nonumber
    R_2 & = \bb S^{p-1} \setminus \paren{R_0 \cup R_1}.
    %\doteq\Brac{\mb q:\norm{\mb A^{T}\mb q}\infty^{2}<  C_{\star} }.    
    \end{align}
  Here $C_{\star}$ is any fixed constant between $0$ and $1$. 
%   The upper bound follows from the inequality that $\|\mb A^{T}\mb q\|_\infty = \max_k |\mb a_k^T \mb q| \le$
% %   \sout{$\|\mb a_k\|_2 \|\mb q\|_2$} 
%   $\max_k\|\mb a_k\|_2 \|\mb q\|_2 = 1$ for any $\bq\in \bb S^{p-1}$. 
  The region $R_0$ can be easily avoided by choosing an initialization such that the objective function $f(\mb q)$ is not equal to zero. For $R_1$ and $R_2$, we are able to show the following results. Let $\Hess f(\mb q)$ be the Riemannian Hessian matrix of (\ref{obj_pop}) at any point $\bq\in \Sp^{p-1}$.
  \begin{enumerate}
      \item[(1)] Optimization landscape for $ R_1$:
      \begin{lemma}\label{lem_R1}
        Assume $\theta < 1$. Any local solution $\bar{\mb q} \in R_1(C_\star)$ to (\ref{obj_pop}) with 
        $
            C_\star > {1\over 2}\sqrt{\theta \over 1-\theta}
        $ 
        recovers one column of $\mb A$, that is, for some signed permutation matrix $\bP$
          $$
              \bar{\mb q} =\mb A \bP_{\cdot 1}.
          $$
          
     \end{lemma}
     To prove 
     Lemma \ref{lem_R1}, we characterize all critical points in $R_1$. Specifically, we show that any critical point $\mb q\in R_1$ is either a strict saddle point in the sense that there exists a direction along which the Hessian is negative, or  satisfies the second order optimality condition and is equal to one column of $\bA$, up to its sign. 
     
      \item[(2)] Optimization landscape for $R_2$:
      \begin{lemma}\label{lem:NegativeCurR2}
        Assume $\theta < 1 / 3$.
        For any point $\mb q\in R_2(C_\star)$ with $C_\star \le {1-3\theta \over 2}$, there exists $\mb v$ such that
        \begin{align}
             \mb v^{T}\Hess f\paren{\mb q}\mb v < 0.
        \end{align}
     \end{lemma}
     Lemma \ref{lem:NegativeCurR2} %is proved in Section \ref{app_proof_R2} of the Appendix and it 
     implies that 
%      \sout{any critical point in $R_2(C_{\star})$ is a saddle point that can be escaped by negative
% curvature} 
any point in $R_2(C_{\star})$ has at least one direction with negative curvature hence can be escaped by any second-order algorithm. Consequently, there is no local solution to (\ref{obj_pop}) in the region $R_2(C_{\star})$. 
  \end{enumerate}
   
  Theorem \ref{thm:obj_pop} thus follows from Lemma \ref{lem_R1} and Lemma \ref{lem:NegativeCurR2}, provided that
  \begin{equation}\label{cond_theta}
        \sqrt{\theta \over 1 - \theta} < 1-3\theta.
    \end{equation}
    Condition (\ref{cond_theta}) puts restrictions on the upper bound of $\theta$. It is easy to see that (\ref{cond_theta}) holds for any $\theta \le 1/6$. As discussed in Remark \ref{rem_theta}, a smaller $\theta$ leads to a more benign optimization landscape. Figure \ref{fig:sup_land} illustrates our results by depicting the landscape of (\ref{obj_pop}) in a $3$-dimensional case.  As shown there, the region $R_1$ contains saddle points as well as all local solutions which recover the ground truth up to the sign. \\

     \begin{figure}[ht]
  \centering 
    \includegraphics[width=0.45\textwidth]{{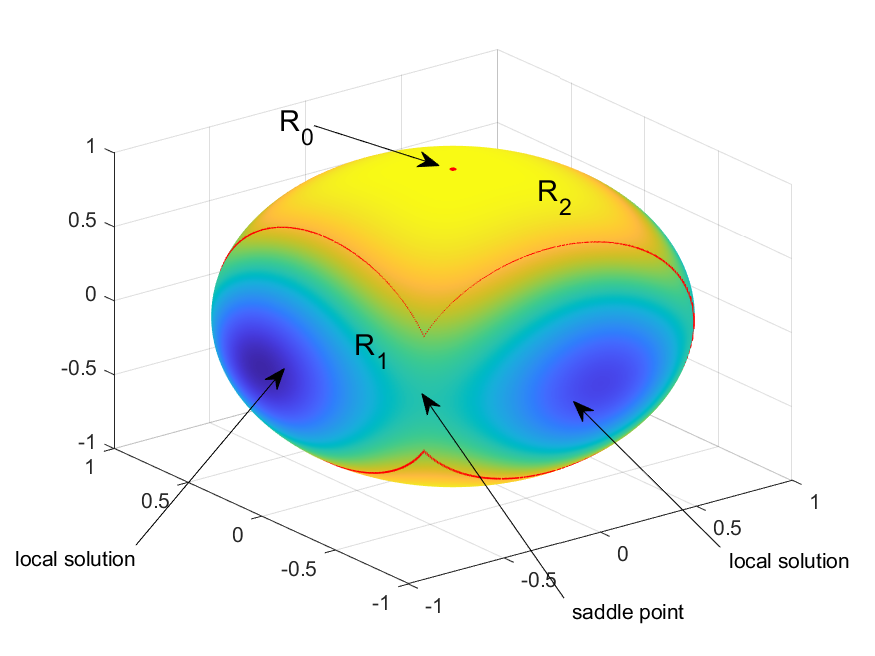}}
   \caption{Visualization of the landscape of solving (\ref{obj_pop})  on $\mathbb{S}^{2}$. Regions $R_0$, $R_1$ and $R_2$ are defined in (\ref{def:obj}). The ground truth $\mb A$ is set as $[\mb e_1, \mb e_2]$ with $\theta = 0.5$ and $C_{\star} = 0.65$.}
   \label{fig:sup_land}
  \end{figure}
    In light of Theorem \ref{thm:obj_pop}, we now provide guarantees for the solution to the finite sample problem (\ref{obj}) in the following theorem. Define the sample analogue of the null region $R_0$ in (\ref{def_R0}) as 
        \begin{equation}\label{def_null_region}
            R_0'(c_\star) ~ \doteq~   \Brac{\mb q\in \mathbb{S}^{p-1}: \|\bA^T \bq\|_{\infty}^2 ~ \le ~ c_\star}
        \end{equation}
    for any given  value $c_\star\in [0,1)$.  
    % For any sequences $a_n$ and $b_n$, we write $a_n \lesssim b_n$ for $a_n \le C b_n$ for some constant $C$ that does not depend on $n$. 
    In our analysis, we are mainly interested in the order of rates of convergence of our estimator and do not pursue derivation of explicit expression of the involved constants.
    
    \begin{theorem}[Finite sample case]\label{thm:obj_sample}
        Under Assumptions \ref{ass_X} and \ref{ass_A_orth}, assume $\theta \in (0,1/9]$ and
        \begin{equation}\label{cond_n_orth}
             n \ge C\max\left\{
                {r^2 \over c_\star }, ~  {\log^2 n}
             \right\}{r\log n \over \theta c_\star}
        \end{equation}
        for some sufficiently large constant $C>0$ and any $c_\star\in (0,1/4]$. 
        Then with probability at least $1 - cn^{-c'}$, 
        any local solution $\bar \bq$ to (\ref{obj}) that is not in $R_0'(c_\star)$
        satisfies
        \begin{equation}\label{rate_orth}
            \left\|\bar \bq - \bA \bP_{\cdot 1}\right\|_2 ~  \lesssim ~  \sqrt{r^2 \log n \over \theta n} + \left(\theta r^2 + {\log^2 n \over \theta}\right){r \log n \over n}
        \end{equation}
        for some signed permutation matrix $\bP$. Here $c$ and $c'$ are some absolute positive numeric constants.
        % The above statement holds with probability at least $1-\paren{rn}^{-1}-e^{-c_2 r\log\paren{\theta r^2\alpha_n/\delta}}-2n^{-c_3}$. Here $c_1,c_2, c_3$, $C$ and $C'$ are some absolute positive constants.
    \end{theorem}

    The proof of Theorem \ref{thm:obj_sample} can be found in Appendix \ref{proof:thm:obj_sample}. The geometric analysis of the landscape of the optimization problem (\ref{obj}) is 
    in spirit similar to  that of Theorem \ref{thm:obj_pop}, but has an additional technical difficulty of taking into account the deviations between the finite sample objective $F(\bq)$ in (\ref{obj}) and the population objective $f(\bq)$ in (\ref{obj_pop}), as well as the deviations of both their gradients and hessian matrices. Such deviations also affect both the size of $R_0'(c_\star)$ in (\ref{def_null_region}), an enlarged region of $R_0$ in (\ref{def_R0}), via condition (\ref{cond_n_orth}), and the estimation error of the local solution $\bar\bq$.

    In Lemmas \ref{lem:function_value_concentration}, \ref{lem:gradient_concentration} and \ref{lem:hessian_concentration} of Appendix \ref{app_concentration_orth}, we provide finite sample deviation inequalities of various quantities between $F(\bq)$ and $f(\bq)$. Our analysis characterizes the explicit dependency 
    %of the estimation error of our estimator 
    on dimensions $n$, $p$ and $r$, as well as on the sparsity parameter $\theta$. In particular, our analysis is valid for fixed $p$, $r$ and $\theta$, as well as growing $p = p(n)$, $r = r(n)$ and $\theta = \theta(n)$. 
    
    The estimation error of our estimator in (\ref{rate_orth}) depends on both the rank $r$ and the sparsity parameter $\theta$, but is independent of the higher dimension $p$. The smaller $\theta$ is, the larger estimation error (or the stronger requirement on the sample size $n$) we have. This is as expected since one needs to observe enough information to accurately estimate the population-level objective in (\ref{obj_pop}) by using (\ref{obj}). On the other hand, recalling from Remark \ref{rem_theta} that a larger $\theta$ could lead to a worse geometry landscape. Therefore, we observe an interesting trade-off of the magnitude of $\theta$ between the optimization landscape and the statistical error. 
    % Furthermore, Theorem \ref{thm:obj_sample} ensures the validity of our approach even  when $\theta$ is of constant order, in which case the overall sparsity of $\bX$ can be as large as order of $nr$.

    \subsection{Theoretical guarantees for general full column rank $A$}\label{sec:theo_obj_gene}
    
    In this section, we provide theoretical guarantees for our procedure of recovering a general full column rank matrix $\bA$ under Assumption \ref{ass_A_rank}.

        Recall from Section \ref{sec_gene} that our approach first preconditions $\bY$ by using $\bD$ from (\ref{def_D}). 
        The following proposition provides guarantees for the preconditioned $\bY$, denoted as $\bar \bY = \bD \bY$. 
        The proof is deferred to Appendix \ref{app_proof_prop_precond}. 
        Write the SVD of $\bA = \bU_A\bD_A\bV_A^T$ with $\bU_A \in \R^{p\times r}$ and $\bV_A\in \R^{r\times r}$ being, respectively, the left and right singular vectors. 
        
        \begin{proposition}\label{prop_precond}
            Under Assumptions \ref{ass_X} and \ref{ass_A_rank}, assume $n \ge C r/\theta^2 $ for some sufficiently large constant $C>0$. 
            With probability greater than $1-2e^{-c'r}$, one has 
            \[
                \bar \bY  =  \bar \bA \bar \bX  + \bE 
            \]
            where $\bar \bA = \bU_A\bV_A^T$, $\bar\bX =  \bX /  \sqrt{\theta n\sigma^2}$ and $\bE = \bar \bA \mb\Delta \bar \bX$ with
            \begin{align}
                \|\mb\Delta\|_{\op} ~ \le ~ c''{1\over \theta}\sqrt{r\over n}.
            \end{align}
            Here $c'$ and $c''$ are some absolute positive constants. 
        \end{proposition}
        
        Proposition \ref{prop_precond} implies that, when $n \ge Cr/\theta^2$, the preconditioned $\bY$ satisfies 
        \[
          \bar \bY = \bar \bA (\bI_r + \mb\Delta) \bar \bX \approx \bar \bA \bar\bX 
        \]
        with $\oA^T\oA = \bI_r$. This naturally leads us to apply our procedure in Section \ref{sec_orth} to recover columns of $\oA$ via (\ref{obj_gene}).
       We formally show in Theorem \ref{thm:obj_gene_sample} below that any local solution to (\ref{obj_gene}) approximates recovering one column of $\oA$ up to a signed permutation matrix. Similar to (\ref{def_null_region}), define 
        \begin{equation}\label{def_null_region_prime}
        R_{0}^{''}(c_\star) ~ \doteq~   \Brac{\mb q\in \mathbb{S}^{p-1}: \norm{\oA^T \bq}{\infty}^2 ~ \le ~ c_\star}
    \end{equation}
    for some given value $c_{\star}\in [0,1)$.
    %     Let 
    %   \begin{equation*}
    %     \omega_n = \theta r^2\log^2(nr)\log^2n. 
    %   \end{equation*}

    \begin{theorem}\label{thm:obj_gene_sample}
        Under Assumption \ref{ass_X} and \ref{ass_A_rank}, assume $\theta\in \left(0,1/9\right]$ and 
        \begin{align}\label{cond_n_general}
           & n\ge \\\nonumber
             & C {r\over c_\star\theta }\max\left\{
            \log^3 n,   {\log n\over c_\star \theta\sqrt \theta},   {\log^2n\over c_\star\theta},   {r \over c_\star\sqrt \theta},  {r^2\log n \over c_\star}
          \right\}.  
        \end{align}
        Then with probability at least $1-cn^{-c'}-4e^{-c''r}$, any solution $\bar \bq$ to (\ref{obj_gene}) that is not in Region $R_0^{''}\left(c_{\star}\right)$ satisfies
        \begin{align*}
        &\left\|\bar \bq - \ \oA\bP_{\cdot 1}\right\|_2 \\
        &\lesssim  \sqrt{r\log n\over \theta^2 n} +\sqrt{r^2\log n \over \theta n}+ \left(\theta r^2 + {\log^2 n\over \theta}\right) {r \log n \over n}
        \end{align*}
        for some signed permutation matrix $\bP$.  Here $c, c'$ and $c''$ are some absolute numeric constants.
        % The above statement holds with probability at least $1-e^{-c'r\log\paren{\omega_n/\delta}}-O(n^{-c''})-O(e^{-c'''r/\theta})$. 
        % Here $c'$,$c''$,$c'''$,$c_1'$,$c_1''$,$c_2$,$c_2'$, $C$ ,$C'$ and $C''$ are some absolute positive constants.
    \end{theorem}
    The proof of Theorem \ref{thm:obj_gene_sample} can be found in  Appendix \ref{proof:thm:obj_gene_sample}. Due to the preconditioning step, the requirement of the sample size in (\ref{cond_n_general}) is slightly stronger than (\ref{cond_n_orth}), whereas the estimation error of $\bar\bq$ only has an additional $\sqrt{r\log n / (\theta^2 n)}$ term comparing to (\ref{rate_orth}).

    Theorem \ref{thm:obj_gene_sample} requires to avoid the null region $R_0''(c_\star)$ in (\ref{def_null_region_prime}). We provide a simple initialization in the next section that provably avoids $R_0^{''}$. Furthermore, every iterate of any descent algorithm based on such initialization is provably not in $R_0^{''}$ either.

     % Therefore, it is straightforward to see that (see (\ref{obj}) and (\ref{crit_ak})) recovers the semi-orthonormal matrix $\bar\bA$. Furthermore, the original $\bA$ can be recovered by inverting the preconditioning matrix $\bD$.

 \section{Complete Algorithm and Provable Recovery}\label{sec:alg}
  
  In this section, we present a complete pipeline for recovering $\bA$ from $\bY$. So far we have established that every local solution to (\ref{obj_gene}), that is not in  $R_0''(c_\star)$, approximately recovers one column of $\oA = \bU_A\bV_A^T$. To our end, we will discuss: (1) a data-driven initialization in Section \ref{sec:init} which, together with Theorem \ref{thm:obj_gene_sample}, provably recovers one column of $\oA$; (2) a deflation procedure  in Section \ref{sec:deflation} that sequentially recovers all remaining columns of $\oA$. 
    %   (3) a preconditioning procedure in Section \ref{sec:precondition} for the observed data $\bY$ when $\bA$ is a general full column rank matrix only satisfying Assumption \ref{ass_A_rank}.

  \subsection{Initialization}\label{sec:init}

    Our goal is to provide a simple initialization such that solving (\ref{obj_gene}) via any second-order descent algorithm provably recovers one column of $\oA$. According to Theorem \ref{thm:obj_gene_sample}, such an initialization needs to guarantee the following conditions. 
    \begin{itemize}
        \item \textbf{Condition I: } The initial point $\mb q^{(0)}$ does not fall into region $R_{0}^{''}(c_{\star})$ for some $c_{\star}$ satisfying (\ref{cond_n_general}) in Theorem  \ref{thm:obj_gene_sample}.
        \item \textbf{Condition II: } The updated iterates $\mb q^{(k)}$, for all $k\ge 1$, stay away from $R_0''(c_\star)$ as well.
    \end{itemize}
    We propose the following initialization 
    \begin{align}\label{def:q_init_gene}
        \mb q^{(0)}= {\bar\bY {\mb 1_n} \over \|\bar\bY {\mb 1_n}\|_2} \in \Sp^{p-1}.
        %\sum_{i=1}^{n}\bar{\mb Y}_{\cdot i}\right]
    \end{align}
    The following two lemmas guarantee that both \textbf{Condition I} and {\bf Condition II} are met for this choice. Their proofs can be found in Appendices \ref{app_proof_lemma_init} and \ref{app_proof_lemma_iter}.
    
    \begin{lemma}\label{lem:gene_init}
    Under Assumption \ref{ass_X} and \ref{ass_A_rank}, assume $\theta\in \left(0,1/9\right]$ and 
    \begin{equation}\label{cond_n_init}
        n \ge C {r^2\over \theta }\max\left\{
            \log^3 n,  {r\log n\over\theta\sqrt \theta},  {r\log^2n\over \theta},  {r^2 \over \sqrt \theta},  {r^3\log n}
          \right\}
   \end{equation}
    holds, then, with probability at least $1-2e^{-cr}$, the initialization $\mb q^{(0)}$ in (\ref{def:q_init_gene}) is not in region $R^{''}_{0}(c_\star)$ with $c_\star = 1/(2r)$. 
    \end{lemma}

    \begin{lemma}\label{lem:gene_iter}
     Let $\bq^{(k)}$, for $k\ge 1$, be any updated iterate from solving (\ref{obj}) by using any descent algorithm with the initial point $\bq^{(0)}$ chosen as (\ref{def:q_init_gene}).  
     If 
     \begin{equation}\label{cond_n_iter}
      n \ge C {r^2\over \theta}\max\left\{
        \log^3 n, ~ {r\log n\over \theta\sqrt \theta},  ~ {r^2 \over \sqrt \theta},~ \theta^2r^2\log n\right\}
     \end{equation}
     holds, then, with probability at least $1-cn^{-c'}-2e^{-c''r}$ for some absolute numeric constants $c,c',c''>0$,
     one has
        $$
            \bq^{(k)}\notin R_0''\left({1\over 2r}\right),\quad \forall~  k\ge 1.
        $$
    \end{lemma}

    Combining Lemmas \ref{lem:gene_init} and \ref{lem:gene_iter} together with Theorem \ref{thm:obj_gene_sample} readily yields the following theorem. 
    
    \begin{theorem}\label{thm:complete}
    Under Assumptions \ref{ass_X} and \ref{ass_A_rank}, assume $\theta \in (0, 1/9]$ and (\ref{cond_n_init}) holds.  
    Let $\bar \bq$ be any local solution  to (\ref{obj_gene}) from any second-order descent algorithm with the initial point chosen as (\ref{def:q_init_gene}). 
    With probability at least $1-cn^{-c'} - 4e^{-c''r}$ for some absolute numeric constants $c,c',c''>0$, one has
        \[
            \|\bar \bq - \oA \bP_{\cdot 1}\|_2 ~  \lesssim ~   \sqrt{r\log n\over \theta^2 n} +\sqrt{r^2\log n \over \theta n}+  {r \log^3n \over \theta n} 
        \]
        for some signed permutation matrix $\bP$. 
    \end{theorem}
    
    Theorem \ref{thm:complete} provides the guarantees for using any second-order descent algorithms \cite{nesterov2006cubic,battiti1992first} to solve (\ref{obj_gene}) with the initialization chosen in (\ref{def:q_init_gene}).

    %optimizing the following problem.
    %\begin{align}
    %    \min_{\mb q\in \bb S^{p-1}}F\left(\mb q\right)\nonumber
    %\end{align}
% \begin{algorithm}[tb]
%   \caption{Single Column Recovery}
%   \label{Alg:one_col}
% \begin{algorithmic}
%   \STATE {\bfseries Input:} data $\mb Y$\
%   \STATE {\bfseries Output:} estimate $\mb q$\
%   \STATE Initialize $\mb q^{(0)} = \bar\by$ with $\bar \by = n^{-1}\sum_{t=1}^n\bY_{\cdot t}$.\
%   \REPEAT\
%     \STATE Compute the gradient $\nabla_{F}\left(\mb q^{(k)}\right)$\
%     \STATE Update $\mb q^{(k+1)}=P_{\bb S^{p-1}}\left(-\nabla_{F}\left(\mb q^{(k)}\right)\right)$\footnotemark \
%   \STATE Set $k=k+1$\
%   \UNTIL{convergence}\
% \end{algorithmic}
% \end{algorithm}
%  \footnotetext{We write $P_{\bb S^{p-1}}\left(\mb x\right)=\mb x / \norm{\mb x}2$ for any vector $\mb x\in \R^p$.}

\subsection{Recovering the full matrix $A$}\label{sec:deflation}
    % Aforementioned initialization avoids the flat region corresponds to maximizers in this problem, together with our characterization of the nonconvex landscape, any descent second algorithm finds one column of $\mb A$. For completeness, 
    
    Theorem \ref{thm:complete} provides the guarantees for recovering one column of $\oA$. 
    In this section, we discuss how to recover the remaining columns of $\oA$ by using the deflation method \cite{spielman2013exact,sun2016complete,qu2020geometric}.

%   \begin{algorithm}[H]
%   \SetAlgoLined
%   \KwData{data $\mb Y$}
%   \KwResult{estimate $\mb A$}
%     Initialize $\mb A=0 $.\newline
%   \For{$j = \{1,2,\ldots, r\}$}{
%     $\mathcal A_j = \textrm{span}(\ba_1,\ldots,\ba_j)$\newline
%     $\mb q^{\star}=\textrm{Single Column Recovery}$ with input $\left( P_{\mathcal A_j}^{\perp}\mb Y\right)$\footnotemark \newline
%     $\mb A[:,j]=\mb q^{\star}$
%   }
%   \caption{Full Matrix Recovery}\label{Alg:full}
% \end{algorithm} 

     Suppose that solving (\ref{obj_gene}) recovers 
      $\bar \ba_1$, the first column of $\oA$. For any $k\in \{2,\ldots,r\}$, write
      $\mathcal A_k=\text{span}(\bar\ba_1, \ldots, \bar\ba_{k-1})$, the space spanned by all previously recovered columns of $\bA$ at step $k$. Further define $P_{\mathcal A_k}$ as the projection matrix onto $\mathcal A_k$ and write $P_{\mathcal A_k}^{\perp} = \bI_p - P_{\mathcal A_k}$. We propose to solve the following problem to recover a new column of $\oA$,
      \begin{align}\label{crit_ak}
        \min_{\bq} ~  -{\theta n \over 12} \norm{\bq^T P_{\mathcal A_k}^{\perp}\bar\bY}4^4,\quad
        \st\quad        \norm{\mb q}2=1.
      \end{align}
      To facilitate the understanding, consider $k = 2$ and $P_{\mathcal A_k}^\perp = P_{\bar\ba_1}^\perp$. Then (\ref{crit_ak}) becomes
      \begin{align}\label{crit_a2}
        \min_{\bq}~ -{\theta n \over 12} \norm{\bq^T P_{\bar\ba_1}^{\perp}\bar\bY }4^4,\quad
        \st \quad \norm{\mb q}2=1.
      \end{align}
      From Proposition \ref{prop_precond}, we observe that 
     \begin{align}
        P_{\bar\ba_1}^{\perp}\bar\bY  \approx P_{\bar\ba_1}^{\perp}\oA\oX = \oA_{(-1)}\oX_{(-1)},
      \end{align} 
      where we write $\oA_{(-1)}\in \R^{p\times (r-1)}$ and $\oX_{(-1)}\in \R^{(r-1)\times n}$ for $\oA$ and $\oX$ with the $1$th column and the $1$th row removed, respectively. Then it is easy to see that recovering one column of $\oA_{(-1)}$ from $P_{\bar\ba_1}^{\perp}\bar\bY$ is the same problem as recovering one column of $\bar\bA$ from $\bar\bY$ with $r$ replaced by $r-1$, hence can be done via solving (\ref{crit_a2}). Similar reasoning holds for any $2\le k\le r$. 
      
     As soon as we recover $\oA$, the original $\bA$ is recovered by $\bD^{-1}\oA/\left\|\bD^{-1}\oA\right\|_{\op}$(under Assumption \ref{ass_A_rank}) with $\bD$ defined in (\ref{def_D}). 
      For the reader's convenience, we summarize our whole procedure of recovering $\bA$ in Algorithm \ref{Alg:full}.

    % \xb{The following algorithm needs to be updated.}
    
    \begin{algorithm}
      \SetAlgoLined
      \KwData{a matrix $\mb Y\in\R^{p\times n}$ with rank $r\le \min\{n,p\}$}
      \KwResult{matrix $\mb A\in \R^{p\times r}$}
        Compute $\bD$ from (\ref{def_D}) and obtain $\bar\bY = \bD\bY$\;
        Set $\mathcal A_j =\emptyset$ and initialize $\bq^{(0)}$ as (\ref{def:q_init_gene})\;
      \For{$j = \{1,2,\ldots, r\}$}{
        %Initialize $\mb a^{(0)} = \bar\by$ with $\bar \by = n^{-1}\sum_{t=1}^n\bY_{\cdot t}$\;
        Solve $\mb a_j^\star$ from (\ref{crit_ak}) 
        by using $\bq^{(0)}$ and any second-order descent algorithm\;
        Update $\oA_{\cdot j} = \mb a_j^\star$\;
        Set $\mathcal A_j = \textrm{span}(\oA_{\cdot 1},\ldots,\oA_{\cdot j})$\;
      }
      Compute $\bA = \bD^{-1}\oA/\left\|\bD^{-1}\oA\right\|_{\op}$.
      \caption{Sparse Low Rank Decomposition}\label{Alg:full}
    \end{algorithm}

%       \begin{algorithm}[tb]
%   \caption{Full Matrix Recovery}
%   \label{Alg:full}
% \begin{algorithmic}
%   \STATE {\bfseries Input:} data $\mb Y$
%   \STATE {\bfseries Output:} estimate $\mb A$
%   \STATE Initialize $\bA = \mb 0$.
%   \FOR{$j = \{1,2,\ldots, r\}$}
%   \STATE $\mathcal A_j = \textrm{span}(\ba_1,\ldots,\ba_j)$
%   \STATE $\mb q^{\star}=\textrm{Single Column Recovery}$ with input $\left( P_{\mathcal A_j}^{\perp}\mb Y\right)$\footnotemark 
%   \STATE $\mb A[:,j]=\mb q^{\star}$
%   %\STATE Set $j=j+1$
%   \ENDFOR
% \end{algorithmic}
% \end{algorithm}
\footnotetext{$P_{\mathcal A_j}^{\perp}$ is the projection matrix onto the orthogonal space of $\mathcal A_j$.}

\section{Experiments}\label{sec:exp}

In this section, we verify the empirical performance of our proposed algorithm for recovering $\bA$ under model (\ref{eqn:model}) in different scenarios.
% Due to the space limit, we defer more experiments to the Appendix of this paper.

\subsection{Experiment setup}\label{exp:setting}

    We start by describing our data-generating mechanism of $\bY = \bA\bX$. 
    The general full rank $\mb A$ is generated as $\mb A_{ij} \overset{i.i.d.}{\sim} \mc N(0, 1)$ with its operator norm scaled to 1. On the other hand, entries of the sparse coefficient matrix $\mb X\in \bb R^{r\times n}$ are generated i.i.d. from Bernoulli-Gaussian with parameter $\theta$ and $\sigma^2 = 1$. 
    
    Recall that $\oA = \bU_A\bV_A^T$ with $\bU_A$ and $\bV_A$ consisting of the left and right singular vectors of $\bA$, respectively. 
    We first evaluate the performance of our procedure (\ref{obj_gene}) in terms of the probability of successfully recovering one column of $\oA$.  Specifically, let $\bar\bq\in \Sp^{p-1}$ be our estimate from (\ref{obj_gene}), we compute
    \begin{align}\label{exp:error_def_onecol}
        \textbf{Err}\paren{\bar\bq}=\min_{1\leq i\leq r}\paren{1-|\innerprod{\bar\bq}{\bar  {\mb a}_i}|}
    \end{align} 
    with $\oA = (\bar{\ba}_1, \ldots, \bar{\ba}_r)$. 
    If $\textbf{Err}\paren{\bar \bq}\leq \rho_{e}$ for some small value $\rho_e>0$, we say the vector $\bar\bq$ successfully recovers one column of $\oA$.  % Figure \ref{fig:a} and \ref{fig:b}. 
    We consider two scenarios in Section \ref{exp:prab_theta_pr} to evaluate the probability of recovering one column of $\oA$. In the first case, we vary simultaneously $\theta$ and $r$ while in the second case we change $n$ and $r$.

   In Section \ref{sec:exp_general} we also examine the performance of our proposed Algorithm \ref{Alg:full} for recovering the whole matrix $\mb A$ by using the following normalized Frobenius norm between any estimate $\mb A_{est}$ and the true $\mb A$:
    \begin{align}\label{exp:error_def}
        \min_{\mb P} \quad&\frac{1}{\sqrt{r}}\norm{\mb A_{est}-\mb A\mb P}{F}\\ \nonumber
        \textrm{s.t.} \quad &\mb P \textrm{ is a signed permutation matrix.}\nonumber
    \end{align}
     %as a measurement in Figure \ref{fig:c}.

    % We then evaluate the performance of our procedure, Algorithm \ref{Alg:full} in Section \ref{sec:deflation}, for recovering the full matrix $\bA$ based on the criterion (\ref{exp:error_def}).
   In Section \ref{exp:sup_compare}, we further compare the performance of our proposed method with two  algorithms that are popular in solving the sparse principal component analysis (SPCA) problem.  
   
   In the aforementioned settings, we choose the projected Riemannian gradient descent (for the one-column recovery result in Figure \ref{fig:a}) and the general power method \cite{journee2010generalized} (for the full matrix recovery result in Figures \ref{fig:c_sup}, \ref{fig:a_sup} and \ref{fig:b_sup}) in Algorithm \ref{Alg:full}. Detailed specifications of these two descent methods as well as the comparison between them are stated in Appendix \ref{exp:comp_pm_pgd}.
%   A more efficient algorithm based on general power method \cite{journee2010generalized} is also considered and discussed in Appendix \ref{exp:comp_pm_pgd}.}

     \subsection{Successful rate of one column recovery}\label{exp:prab_theta_pr}
    In this part, we examine the probability of successfully recovering one column in $\bar{\bA}$ . Here we use (\ref{exp:error_def_onecol}) for computing recovering error and $\rho_{e}$ is set to $0.01$ in our simulation.
    \paragraph{Varying $\theta$ and $r$}
        We fix $p=100$ and $n=5\times10^{3}$ while vary $\theta \in\{0.01, 0.04, \ldots, 0.58\}$ and $r\in \{10, 30, \ldots, 70\}$. For each pair of $\left(\theta, r\right)$, we repeatedly generate $200$ data sets and apply our procedure in (\ref{obj_gene}). The averaged recovery probability of our procedure over the $200$ replicates is shown in Figure \ref{fig:a}. The recovery probability gets larger as $r$ decreases, in line with Theorem \ref{thm:obj_gene_sample}. We also note that the recovery increases for smaller $\theta$. This is because smaller $\theta$ renders a more benign geometric landscape of the proposed non-convex problem, as detailed in Remark \ref{rem_theta}. On the other hand, the recovery probability decreases when $\theta$ is approaching to $0$. As suggested by Theorem \ref{thm:obj_sample}, the statistical error of estimating $\bA$ gets inflated as $\theta$ gets too small.
        
     \begin{figure} 
        \centering
         \includegraphics[width=0.5\columnwidth]{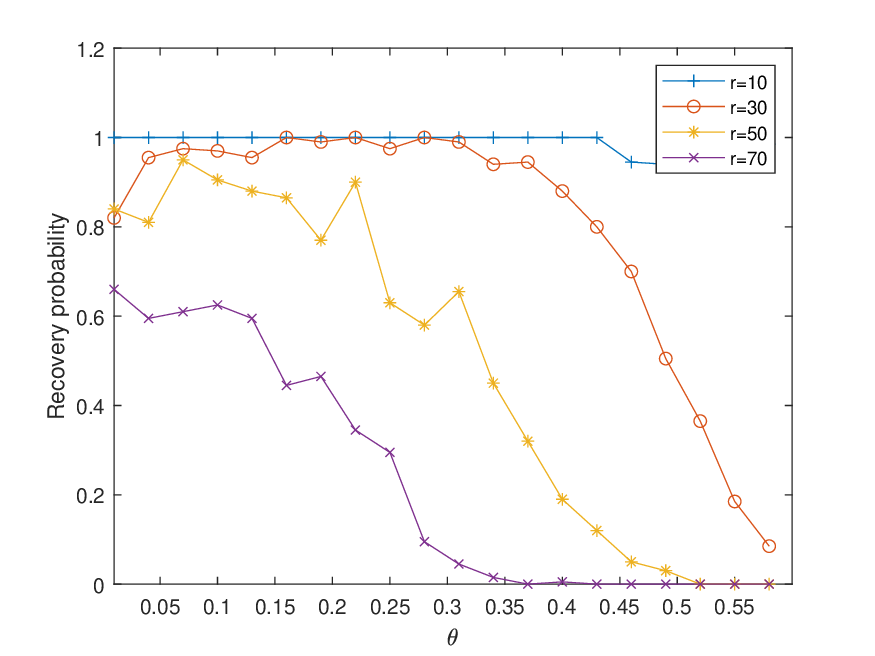}
         \\
         \includegraphics[width=0.5\columnwidth]{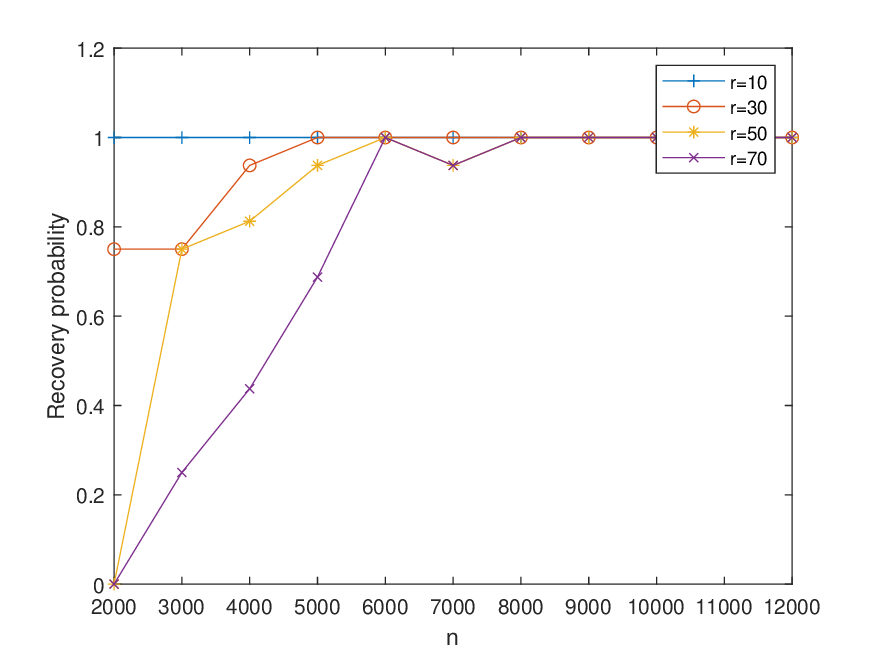}
      \caption{The averaged probability of successful recovery for (top) different $\theta$ and $r$ with $p=100$ and $n=5\times 10^{3}$ and (bottom) different $n$ and $r$ with $p=100$ and $\theta=0.1$.}
      \label{fig:a}
    \end{figure}
    \paragraph{Varying $n$ and $r$} 
    Here we fix $p=100$ and the sparsity parameter $\theta=0.1$. We vary $r\in \{10, 30, \ldots, 70\}$ and $n\in \{2000, 3000, \ldots, 12000\}$. %Follow the setting in experiment above, probability for each pair of $\left(\theta, r\right)$ is sampled from $200$ trials. 
    Figure \ref{fig:a} shows the averaged recovery probability of our procedure over 200 replicates in each setting. Our procedure performs increasingly better as $n$ increases, as expected from Theorem \ref{thm:obj_sample}.

  \subsection{Estimation error of a general full column rank $ A$}\label{sec:exp_general}
  
 To generate a full column rank matrix $\mb A$, we use $\mb A_{ij} \overset{i.i.d.}{\sim} \mc N(0, 1)$. We follow the same procedure in Section \ref{sec:exp} for generating $\mb X$. We choose $n=1.2\times 10^{4}$ and $p=100$, and vary $\theta \in\{0.01, 0.04, \ldots, 0.58\}$ and  $r\in \{10, 20, \ldots, 50\}$.
 Figure \ref{fig:c_sup} depicts the performance of our method under different choices of $\theta$ and $r$. The error of estimating $\mb A$ gets smaller when either $\theta$ or $r$ decreases for $\theta\geq 0.1$. Also for relatively small $\theta$ ($\theta < 0.1$) we find that the error increases when $\theta$ gets smaller. These findings are in line with our result in Theorem \ref{thm:complete}. The tradeoff of $\theta$ that we observe here agrees with our discussion in Remark \ref{rem_theta}.
 
%  Figure \ref{fig:c2_supl} shows that error increase when $\theta$ approaching close to $0$ which is in line with the result in Theorem \ref{thm:complete}.
%When $r$ is small enough, our method successfully recover whole general full column rank matrix $\mb A$. 
  
%   \begin{figure}[!htbp] 
%      \centering
%      \begin{subfigure}[b]{0.45\textwidth}
%          \centering
%          \includegraphics[width=\textwidth]{Supplement_Exp_result/r_theta_general_recover_whole.eps}
%          \caption{Errors of estimating a full column rank $A$ for different choices of $\theta$ and $r$.}
%          \label{fig:c_sup}
%      \end{subfigure}
%      \hfill
%      \begin{subfigure}[b]{0.45\textwidth}
%          \centering
%          \includegraphics[width=\textwidth]{Supplement_Exp_result/r_sm_theta_general_recover_whole.eps}
%          \caption{Errors of estimating a full column rank $A$ for relatively small of $\theta$ and $r$.}
%          \label{fig:c2_supl}
%      \end{subfigure}
          
%  \end{figure}
     \begin{figure} 
     \centering
    \includegraphics[width=0.5\columnwidth]{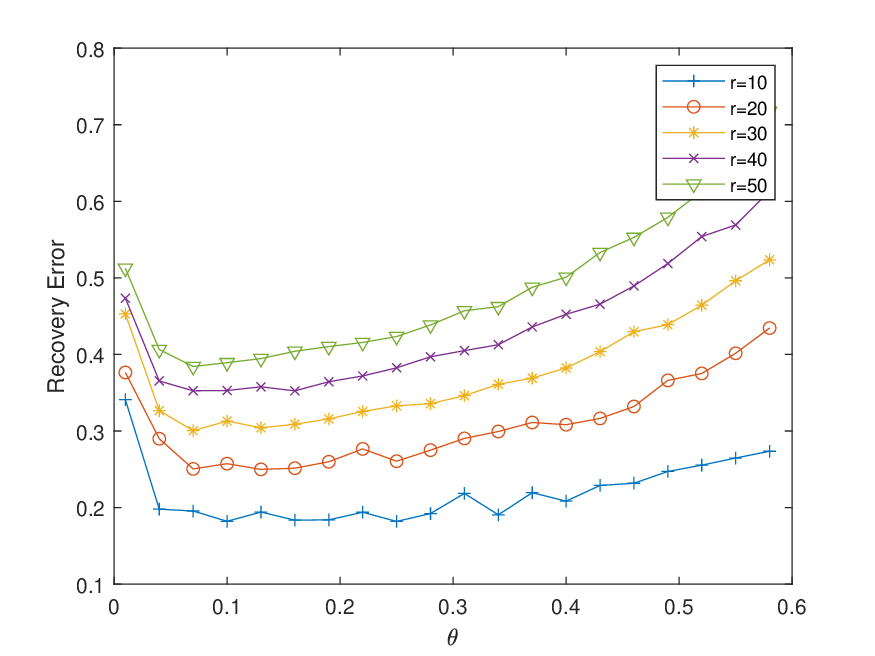}
        \caption{Errors of estimating a full column rank $A$ for different choices of $\theta$ and $r$.}
         \label{fig:c_sup}
\end{figure}
 
 \subsection{Comparison with algorithms used for  SPCA}\label{exp:sup_compare}
 In this section, we compare the performance of our method with two SPCA procedures. Among the various algorithms for solving the SPCA problem, we compare with the LARS algorithm \cite{Zou06spca,zou2005regularization} and the alternating direction method (ADM) \cite{bertsekas2003parallel,tseng2001convergence}. 
 The LARS solves 
  \begin{align}\nonumber
     \min_{\mb v_i}\quad\norm{\mb Z_i-\mb Y\mb v_i }2+ \lambda\norm{\mb v_i}1  +\lambda'\norm{\mb v_i}2,
 \end{align}
 for each $1\le i\le r$,
 to recover rows of $\bX$. Here, $\mb Z_i$ is the $i$th principle component of $\mb Y$. Denote by $\hat\bX$ the estimated $\bX$ from LARS. The matrix $\bA$ is then recovered by $\mb Y \hat\bX^T (\hat\bX\hat\bX^T)^{-1}$ with its operator norm re-scaled to 1.
 On the other hand, the ADM  algorithm solves
      \begin{align}\nonumber
        \min_{\bu, \bv} \quad\norm{\bY - \bu \bv^T}{F}^2 ~  + ~ \lambda \|\bv\|_1 \quad
        \st \quad \|\bu\|_2=1,
    \end{align}
to recover one column of $\bA$ (from the optimal $\bu$) by alternating minimization between $\bu$ and $\bv$. The above procedure is successively used to recover the rest columns of $\bA$ by projecting $\bY$ onto the complement space spanned by the recovered columns of $\bA$ \cite{ShenHuangspca}. 

Since both SPCA procedures aim to recover $\bA$ with orthonormal columns, to make a fair comparison, we set $\bA$ equal to the left singular vectors of $\mb Z\in\bb R^{p\times r}$ with $\mb Z_{ij} \overset{i.i.d.}{\sim} \mc N(0, 1)$.
Figure \ref{fig:a_sup} and \ref{fig:b_sup} depict the estimation error of three methods in terms of (\ref{exp:error_def}) for various choices of $r$ and $\theta$, respectively. The estimation errors of all three methods get larger when either $r$ and $\theta$ increases. 
LARS, however, has the worse performance in all scenarios. %Compared to the LARS, our algorithms perform significantly better under settings with greater $\theta$ and $n$. 
Compared to the ADM method, our method has similar performance for relatively small $\theta$ ($\theta < 0.4$) but has significantly better performance for moderate $\theta$ ($\theta > 0.4$). It is also worth mentioning that, in contrast to the established guarantees of our method, there is no theoretical justification that the ADM method recovers the ground truth.

  \begin{figure}[!htbp] 
     \centering
     \begin{subfigure}[b]{0.45\textwidth}
         \centering
         \includegraphics[width=\textwidth]{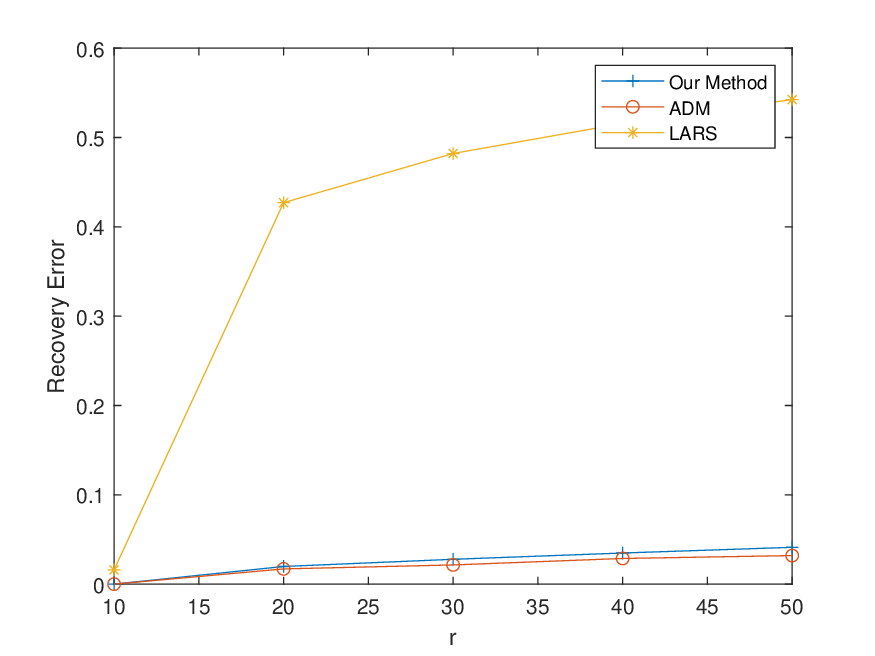}
         \caption{Recovery probability verse $r$}
         \label{fig:a_sup}
     \end{subfigure}
     \hfill
     \begin{subfigure}[b]{0.45\textwidth}
         \centering
         \includegraphics[width=\textwidth]{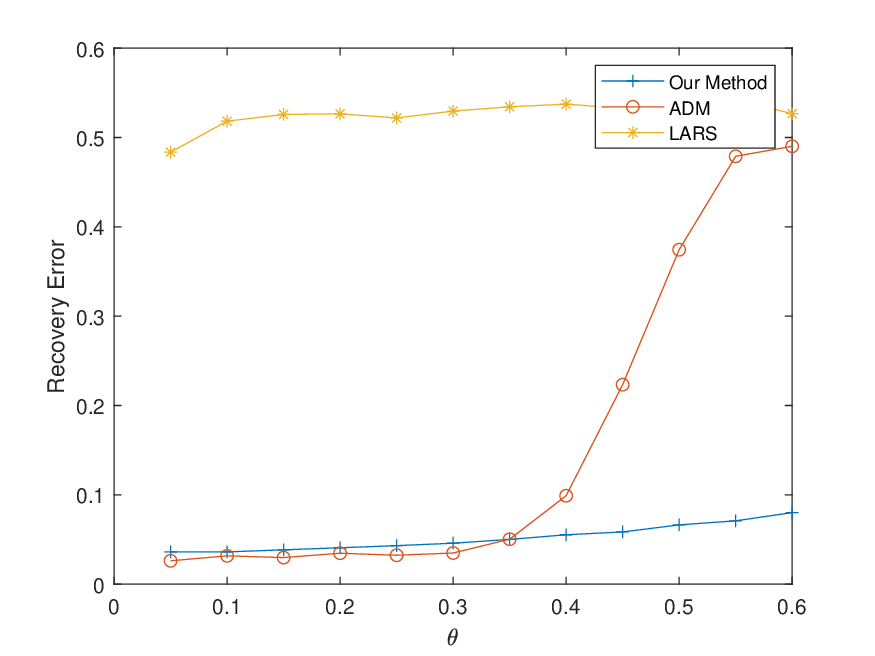}
         \caption{Recovery probability verse $\theta$}
         \label{fig:b_sup}
     \end{subfigure}
        \caption{Comparison with LARS and ADMM: estimation errors of three methods for varying $r$ (top) and $\theta$ (bottom) with $p=100$ and $n=1.2\times 10^{4}$. In Figure \ref{fig:a_sup}, $\theta$ is fixed to $0.1$ and in Figure \ref{fig:b_sup} $r$ is chosen as $30$.}
\end{figure}

% \item \textbf{Phase transition}\label{exp:phase_tran}\\
%     In our cases, sparsity $\theta$ is fixed to $0.1$ and number of observation is fixed to $5\times10^{2}$. Here we test the limit of number of atoms $r$ and dimension of each atom $p$ related for success of recovery. For each pair of $\paren{p,r}$ we calculate the probability for successful recovery with $40$ trials and then get the final result. As we can see in the Figure \ref{fig:d}, the probability of a successful recovery only related to value of number of atoms $r$ under our settings. This observation is also consistency with the result in Figure \ref{fig:a} and \ref{fig:b}

\section{Conclusion and Future Work}\label{sec:concl}
    In this paper, we study the unique decomposition of a low-rank matrix $\bY$ that admits a sparse representation. Under model $\bY = \bA \bX$ where $\bX$ has i.i.d. Bernoulli-Gaussian entries and $\bA$ has full column rank, we propose a nonconvex procedure that provably recovers $\bA$, a quantity that can be further used to recover $\bX$. We provide a complete analysis for recovering one column of $\bA$, up to the sign, by showing that any second-order descent algorithm with a simple and data-driven initialization provably attains the global solution, despite the nonconvex nature of this problem. 
    
    There are several directions worth further pursuing. For instance, it is of great interest to develop a complete analysis of the deflation procedure for recovering the full matrix $\bA$. %For a general $\bA$ that only has full column rank rather than orthonormal columns, it is of interests to complete the analysis of our proposed estimator for estimating $\bA$ on top of the preconditioning properties in Proposition \ref{prop_precond}. Finally, i
    It is worth studying this decomposition problem  in presence of some additive errors, that is, $\bY = \bA \bX + \bE$. Our current procedure only tolerates $\bE$ that has small entries. We leave the problem of modifying our procedure to accommodate a moderate / large $\bE$ to future research.

{\small
\bibliographystyle{plain}
\bibliography{./ncvx.bib}
}

\appendices

\begin{appendices}

\section{Technical comparison with \cite{qu2019analysis, zhai2020complete}}\label{app_related}

%\subsection{Comparison with \cite{qu2019analysis} and \cite{zhai2020complete}}

   \cite{qu2019analysis} studies the unique recovery of $\bY = \bA \bX$ under the setting of over-complete dictionary learning, that is, $\bA$ has full {\bf row} rank with $p\le r$. Although our criterion in (\ref{obj}) is similar to that used in  \cite{qu2019analysis}, the low-rank structure of $\bY$ in our setting implies that $\bA$ has full {\bf column } rank with $p>r$, which brings fundamental differences in both the rationale of using (\ref{obj}) and its subsequent analysis. To be specific, under the setting in \cite{qu2019analysis}, the matrix $\bA =  (\ba_1,\ldots, \ba_r)\in \R^{p\times r}$, with $r\ge p$, is assumed to be unit norm tight frame (UNTF), in the sense that 
   \[
        \bA\bA^T = {r \over p}\bI_p,\qquad \|\ba_i\|_2 = 1, \qquad \mu = \max_{i\ne j}|\langle \ba_i, \ba_j \rangle| \ll 1.
   \]
   Under this condition and Assumption \ref{ass_X}, the objective $F(\bq)$ in (\ref{obj}) satisfies (see, display (2.4) in \cite{qu2019analysis})
   \begin{align}
           \E[F(\bq)] = - {1\over 4}(1-\theta)\|\bA^T \bq\|_4^4 - C \label{obj_qu2019}
   \end{align}
   where $C$ is some numerical value that does not depend on $\bq$. Therefore, solving (\ref{obj}), for large $n$, approximately maximizes $\|\bA^T\bq\|_4^4$ over the unit sphere in the context of \cite{qu2019analysis}. 
   
   There are at least three major differences to be noted. First, as $\bA$ is UNTF in the setting of \cite{qu2019analysis}, columns of $\bA$ are not orthonormal, or equivalently, $\mu > 0$. As a result, 
   Lemma \ref{lem_key} does not hold for their setting. In another word, even one can directly solve (\ref{crit_ell_4}), the solution does not exactly recover one column of $\bA$. Indeed, Proposition B.1 in \cite{qu2019analysis} shows that the difference between the solution to (\ref{crit_ell_4}) and one column of $\bA$ is small when $\mu \ll 1$ but is not exactly equal zero unless $\mu = 0$. By contrast, when $\bA$ satisfies $\bA^T\bA = \bI_r$ in our setting, the exact recovery of columns of $\bA$ is achievable via solving (\ref{crit_ell_4}) as shown in Lemma \ref{lem_key}.
   
   Second, due to $\rank(\bA) = r$,  solving (\ref{obj}) in our setting approximately maximizes (\ref{obj_pop}), the objective of which is a convex combination of $\|\bA^T\bq\|_4^4$ and $\|\bA^T\bq\|_2^4$ with coefficients depending on the sparsity parameter $\theta$. Thus, the expected objective in our setting no longer coincides with that in \cite{qu2019analysis} and in fact is more complicated due to the extra term $\|\bA^T\bq\|_2^4$. This additional term brings more complications in our analysis of the geometry landscape of (\ref{obj}) and requires more delicate arguments. Indeed, in view of (\ref{obj}), although it is clear that columns of $\bA$ are still the global maximizers regardless of the presence of $\theta \norm{\mb q^T\mb A}2^4$, it is non-trivial to establish how $\theta \norm{\mb q^T\mb A}2^4$ affects the geometric landscape, such as properties of all stationary points.  Our population-level results fully characterize the benign regime at the presence of $\theta \norm{\mb q^T\mb A}2^4$ (see, Lemmas \ref{lem_R2}, \ref{lem:Case2} and \ref{lem:Case3}). To account the existence of $\theta \norm{\mb q^T\mb A}2^4$, we have to base on not only a different partition of the unit sphere but also a different metric, the sup-norm, of the partition, from the analysis used in \cite{qu2019analysis}.
As we move to the sample-level analysis, extra cares need to be taken, see, for instance, Lemmas \ref{lem:sample_R2_nega_short}, \ref{lem:Case2_sample} and \ref{lem:Case3_sample}. 
   
   Third, the low-rank structure of $\bY$ leads to a null region of solving (\ref{obj}), that is, the region of $\bq$ such that $\bq^T\bY = \bq^T\bA\bX = \0$ (see, (\ref{def_R0}) for the population-level analysis and (\ref{def_null_region}) -- (\ref{def_null_region_prime}) for the finite sample results). This null region does not appear when the matrix $\bA$ is UNTF. In the presence of such a region, to provide guarantees for any second-order gradient-based algorithm, we need to provide a proper initialization outside of this region and to further carefully prove that every iterate does not fall into the null region. Such analysis brings more technical challenges to our analysis than \cite{qu2019analysis} (see, for instance Lemmas \ref{lem:gene_init} and \ref{lem:gene_iter}).\\

   \cite{zhai2020complete} studies the problem of recovering a complete, orthonormal, matrix $\bA\in \mathbb{O}_p$ via 
     \begin{align*}
         \max_{\mb A \in \mathbb{O}_{p}}\quad \norm{\mb A^T\mb Y}4^4.
     \end{align*}
     \cite[Theorem 1]{zhai2020complete} shows that the maximizer of the above problem is close to the desired target. 
    However, due to the intrinsic difficulty of analyzing the stiefel manifold, there is no global convergence analysis and it is still an open problem that whether local solutions or saddle points exist. By contrast, based on our landscape analysis of stationary points, we are able to provide {\em global} guarantees for {\em any} second-order descent algorithm.

%   {\color{red}
%   \subsection{Comparison with \cite{zhai2020complete}}
%   It worth mentioning that \cite{zhai2020complete} also study the unique recovery of $\mb A\in \mathbb{O}^p$ under orthogonal settings:
%      \begin{align}
%          \max_{\mb A \in \mathbb{O}_{p}}\quad \norm{\mb A^T\mb Y}4^4.\label{obj_zhai2020}
%      \end{align}
%       and \cite[Theorem 1]{zhai2020complete} shows that the maximizer of the above problem is close to the desired target which is in line with our lemma \ref{lem_key}. However, two main differences still exist:
   
%   First, similar to \textbf{UNTF} setting in \cite{zhai2020complete}, orthogonal setting of $\mb A$ also leads to similar population object function as stated in \ref{obj_qu2019}. That is to say, more cares need to be taken due to the extra $\norm{\mb A^T\mb q}2^4$ term of our object function \ref{obj_pop} which further results in more differences in landscape analysis and provable algorithm design. 
   
%   Second, even plenty of empirical results are provided in section $4.3$ in  \cite{zhai2020complete} on global convergence of (\ref{obj_zhai2020}). Due to the intrinsic difficulty of analyzing the stiefel manifold, there is no global convergence analysis and it is still an open problem that whether local solutions or saddle points exist. 
%   }

\section{Empirical studies on comparison of the general power method and the projected Riemannian gradient descent}\label{exp:comp_pm_pgd}
In this section we compare the performance of Algorithm \ref{Alg:full} by using the general power method (PM) \cite{journee2010generalized, qu2019analysis} and the projected Riemannian gradient descent (PRGD) \cite{qu2019analysis} in terms of the one-column recovery error and algorithmic convergence rates. 

For simplicity, we consider the matrix $\bA$ with orthonormal columns. Specifically, the ground-truth matrix $\mb A$ is set as the
left singular vectors of random matrix $\mb Z\in \mathbb{R}^{p\times  r}$ with $\mb Z_{ij}\overset{i.i.d.}{\sim} \mc N(0, 1)$ and $\mb X\in \mathbb{R}^{r\times n}$ is i.i.d Bernoulli-Gaussian with varying parameter $\theta$  and $\sigma^2=1$.

We compare the performance of using projected gradient descent and power method for solving problem \ref{obj}, as detailed below.
\begin{itemize}
    \item \textbf{Projected Riemannian Gradient Descent(PRGD):}
    We base on the update 
    \begin{align}\label{med:pgd}
        \mb q^{(k+1)}= P_{\mathbb{S}^{p-1}}\left(\mb q^{(k)}- \gamma^{(k)}\grad F(\mb q^{(k)})\right),
    \end{align}
    for $k \in \{1,2,\ldots\}$, where $\gamma^{(k)}:=\gamma^{(0)}k^{-0.9}$ is the step size of the $k$-th iteration with $\gamma^{(0)}=0.01$  being the initial step size  and $\grad F(\mb q^{(k)})$ being the Riemannian gradient of $F(\mb q^{(k)})$. Here $P_{\mathbb{S}^{p-1}}$ represents the projection matrix onto the unit sphere.\\
    
    \item \textbf{Power Method(PM):} We base on the update
    \begin{align}\label{med:pm}
        \mb q^{(k+1)}= P_{\mathbb{S}^{p-1}}\left(\nabla F(\mb q^{(k)})\right),
    \end{align}
    for $k \in \{1,2,\ldots\}$,
    where  $\nabla F(\mb q^{(k)})$ being the gradient of $F(\mb q^{(k)})$ at each step $k$.
\end{itemize}

\subsection{Comparison on one column recovery errors}
In this section we compare the one-column recovery error, defined in (\ref{exp:error_def_onecol}), of solving (\ref{obj}) by using (\ref{med:pgd}) and (\ref{med:pm}). We consider $r\in \left\{10, 20, \cdots, 80\right\}$ and $\theta=0.1$ in Figure \ref{fig:comp_error_r} while $r=10$, $\theta\in \left\{0.05, 0.1, 0.15,\cdots, 0.6\right\}$ in Figure \ref{fig:comp_conv_r} when fix $n=5\times 10^3$ for both cases.
%Here the error of one column recovery is defined in (\ref{exp:error_def_onecol}) for recovered $\mb q$ and target solution $\mb a_i$ where $i\in \left\{1, 2,\cdots, r\right\}$. 
For comparison, both (\ref{med:pgd}) and (\ref{med:pm}) use the same initialization $\mb q^{(0)}$ in (\ref{def:q_init_gene}) as well as the same stopping criterion, $k \le 5000$. 
Figure \ref{fig:comp_error_r} and \ref{fig:comp_error_theta} depict the one-column recovery errors of both methods, averaged across $50$ repetitions, for each choice of $r$ and $\theta$ respectively. These two methods have nearly the same recovery errors for all $r$ and $\theta$, implying that they both reach the same global minimum. This observation is in line with Theorem \ref{thm:complete}.

\subsection{Comparison on algorithmic convergence rates}

In this part we check the time of reaching convergence for PRGD and PM under the same data generating setup as described in the previous section. Here both methods base on the same initialization and use $ \textbf{Err}\paren{\mb q^{(k)}}\leq 0.01$ as the stopping criteria. 
Figure \ref{fig:comp_conv_r} and \ref{fig:comp_conv_theta} show the averaged running time before convergence for both methods across 50 repetitions. Clearly, PM has a faster convergence rate comparing to PRGD, especially for large $r$. This empirical efficiency of using PM has also been observed by \cite{zhai2020complete,qu2019analysis,xue2021efficient}. However, theoretical justifications for under-complete model on this aspect is still an open problem, deserving future investigation.

 \begin{figure}[!htbp] 
     \centering
     \begin{subfigure}[b]{0.45\textwidth}
         \centering
         \includegraphics[width=\textwidth]{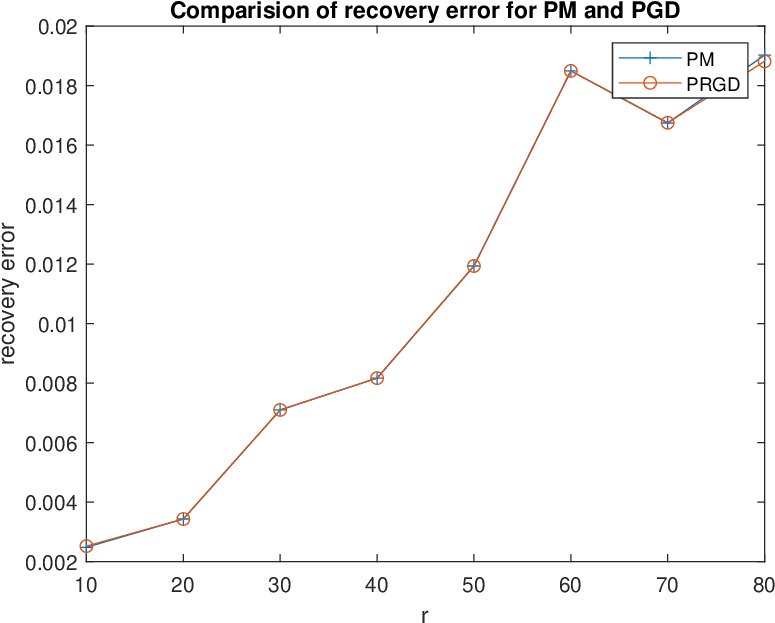}
         \caption{One column recovery errors of using PRGD and PM when varing $r$}
         \label{fig:comp_error_r}
     \end{subfigure}
     \hfill
     \vspace{.1in}
     \begin{subfigure}[b]{0.45\textwidth}
         \centering
         \includegraphics[width=\textwidth]{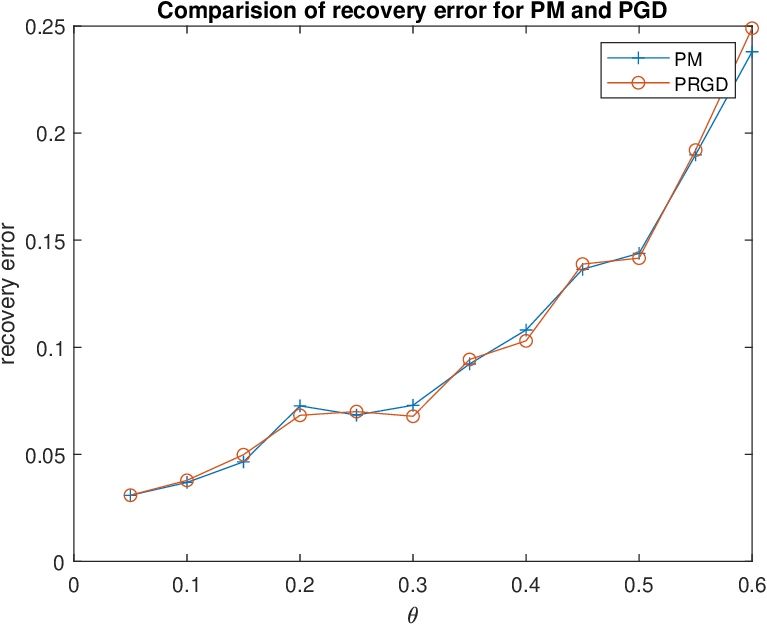}
         \caption{One column recovery errors of using PRGD and PM when varing $\theta$}
         \label{fig:comp_error_theta}
     \end{subfigure}
          \hfill
     \begin{subfigure}[b]{0.45\textwidth}
         \centering
         \includegraphics[width=\textwidth]{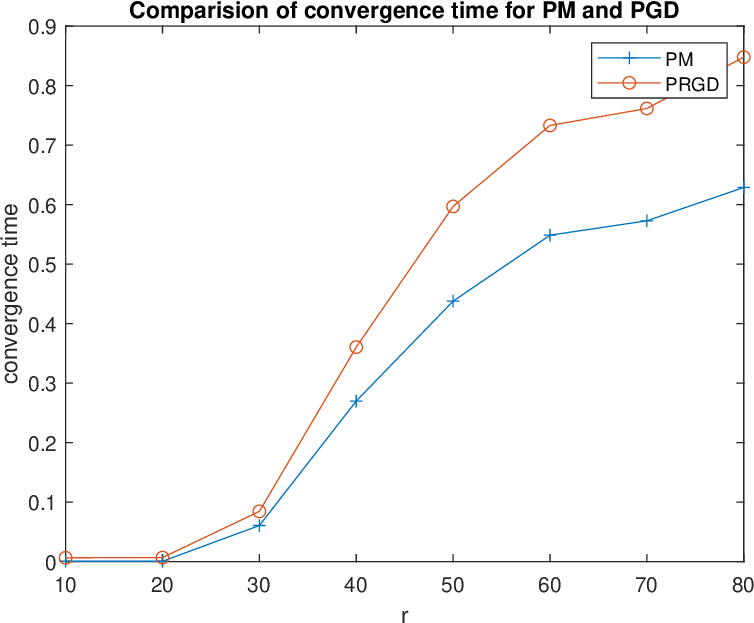}
         \caption{Time until convergence of using PRGD and PM when varing $r$}
         \label{fig:comp_conv_r}
     \end{subfigure}
          \hfill
     \begin{subfigure}[b]{0.45\textwidth}
         \centering
         \includegraphics[width=\textwidth]{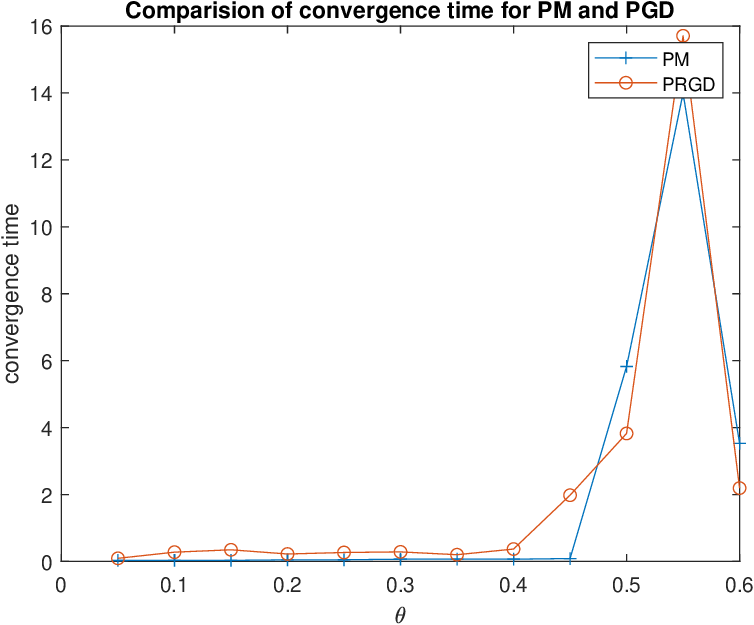}
         \caption{Time until convergence of using PRGD and PM when varing $\theta$}
         \label{fig:comp_conv_theta}
     \end{subfigure}
        \caption{Comparison of PRGD and PM: the top pannels compares the recovery errors where we varies $r$ in figure \ref{fig:comp_error_r} and $\theta$ in figure \ref{fig:comp_error_theta} while bottom two figures depict the convergence time for different $r$ and $\theta$ respectively.  We set $p=100$ and $n=5\times 10^3$ for all figures.}
\end{figure}

\section{Main proofs}

\subsection{Proof of Lemma \ref{lem_key}}\label{app_proof_lem_key}
\begin{proof}
  First, note that, for any $\bq\in \bb S^{p-1}$,
  \begin{equation}\label{eq_chain}
        \|\bA^T \bq\|_4^4 ~ = \sum_{j=1}^r \paren{
          \ba_j^T \bq
        }^4 \le \max_{1\le j\le r}  \paren{
          \ba_j^T \bq
        }^2\sum_{j=1}^r \paren{
          \ba_j^T \bq
        }^2 \le  \max_{1\le j\le r}\|\ba_j\|_2^2 ~ \lambda_1(\bA\bA^T) = 1.
  \end{equation}
  Here $\lambda_1(\bA \bA^T)$ denotes the largest eigenvalue of $\bA \bA^T$ and is equal to $\lambda_1(\bA^T \bA) = 1$. Also note that the maximal value one is achieved by $\bq = \pm \ba_i$ for all $1\le i\le r$ as 
  \[
     \|\bA^T \ba_i\|_4^4  ~ = 1.
  \]
  To prove there is no other maximizer than columns of $\bA$, we observe that the first inequality of (\ref{eq_chain}) holds with equality if and only if 
  \[
       (\ba_j^T \bq)^2 = {1\over s}, \qquad \forall j\in S
  \]
  where 
  \[
    S = \Brac{j\in [r]: ~ \ba_j^T \bq\ne 0}
  \]
  and $s = |S|$. We thus have 
%  \begin{align*}
%       \cancel{\bq = {1\over \sqrt{s}}\sum_{j\in S}\ba_j.}
%   \end{align*}

  \[
     \bq = {1\over \sqrt{s}}\sum_{j\in S}\pm\ba_j. 
  \]

  This choice of $\bq$ leads to 
  $
        \|\bA^T \bq\|_4^4 = {1 / s} 
  $
  which is equal to one if and only if $s = 1$. 
\end{proof}

\subsection{Proof of Lemma \ref{lem:obj}}\label{app_proof_lem_obj}
    
  \begin{proof}
  Pick any $\mb q$ and $C$. One has 
  \begin{align}
      \E\brac{ F(\mb q)} &= -\frac1{12 \theta \sigma^4 n} \E\brac{
       \|\mb{q}^T \mb{AX}\|_4^4
      } = -{1 \over 12 \theta \sigma^4 }  \E\brac{
       |\mb{q}^T \mb{AX_{\cdot i}}|^4
      }
  \end{align}
  by the i.i.d. assumption of columns of $\mb X$. Write $\mb\zeta = \mb {A}^{T} \mb q$ and use Assumption \ref{ass_X} to obtain 
  \begin{align}
          \E\brac{ F(\mb q)} = -{1 \over 12 \theta \sigma^4 } \E\brac{
    \paren{\sum_{j=1}^r \mb \zeta_j \mb B_{ji}\mb Z_{ji}}^4
      }.
  \end{align}
  Since $\mb B_{ji}$ is independent of $\mb Z_{ji}$ and
  \begin{align}
        \sum_{j=1}^r \mb \zeta_j \mb B_{ji}\mb Z_{ji} \Big| \mb B_{\cdot i}\sim N\paren{0, \sigma^2 \sum_{j=1}^r \mb \zeta_j^2\mb B_{ji}^2}
  \end{align}
  from Assumption \ref{ass_X},
  we obtain 
  \begin{align}
      \E\brac{
    \paren{\sum_{j=1}^r \mb \zeta_j \mb B_{ji}\mb Z_{ji}}^4
      } &= 3\sigma^4\E\brac{
    \paren{\sum_{j=1}^r \mb \zeta_j^2\mb B_{ji}^2}^2
      }\nonumber\\
      &= 3\sigma^4\E\brac{
    \sum_{j=1}^r \mb \zeta_j^4\mb B_{ji}^4
      } + 3\sigma^4\E\brac{
    \sum_{j\ne \ell} \mb \zeta_j^2\mb \zeta_{\ell}^2\mb B_{ji}^2\mb B_{\ell i}^2
      }\nonumber\\
      &= 3\sigma^4\theta
    \sum_{j=1}^r \mb \zeta_j^4\mb  + 3\sigma^4\theta^2
    \sum_{j\ne \ell} \mb \zeta_j^2\mb \zeta_{\ell}^2\nonumber\\
    &= 3\sigma^4\theta\brac{ (1-\theta)
    \sum_{j=1}^r \mb \zeta_j^4\mb  + \theta
    \paren{\sum_{j=1} \mb \zeta_j^2}^2}\nonumber\\
    &= 3\sigma^4\theta \brac{
    (1-\theta)\|\mb \zeta\|_4^4 + \theta \|\mb \zeta\|_2^4
    }.
  \end{align}
  The result then follows.  
  \end{proof}

\subsection{Proof of Theorem \ref{thm:obj_pop}}\label{app_proof_thm_obj_pop}

We prove Theorem \ref{thm:obj_pop} by proving Lemmas \ref{lem_R1} and \ref{lem:NegativeCurR2} in Sections \ref{app_proof_R1} and \ref{app_proof_R2}, respectively.  

To analyze the solution to (\ref{obj_pop}), we need the following Riemannian gradient and Hessian matrix of $f(\mb q)$ constrained on the sphere $\norm{ \mb q }2= 1$  
\begin{align}\label{eq_grad}
    \grad f(\mb q) &= ~ -P_{\mb q^\perp}\brac{
    (1-\theta)\sum_{j=1}^r \mb a_j (\mb q^T \mb a_j)^3 + \theta \|\mb q^T \mb A\|_2^2 \mb A\mb A^T \mb q}\\
    \Hess f(\mb q) &=  ~ \Hess_{\ell_4}f(\mb q)+\Hess_{\ell_2}f(\mb q)\label{eq_hess}
\end{align}
where
\begin{align}\label{eqn:Hess1And2}
\Hess_{\ell_4}f(\mb q) &= -(1-\theta)P_{\mb q^\perp}\brac{
    3 \sum_{j=1}^r \mb a_j\mb a_j^T (\mb q^T \mb a_j)^2  -  \|\mb q^T\mb A\|_4^4 \mb I }P_{\mb q^\perp},\\
\Hess_{\ell_2}f(\mb q) &= - \theta P_{\mb q^\perp}\brac{  \|\mb q^T \mb A\|_2^2\mb A\mb A^T + 2 \mb A\mb A^T \mb q\mb q^T \mb A \mb A^T- \|\mb q^T \mb A\|_2^4 \mb I}P_{\mb q^\perp}.
\end{align}
Recall that, for any $C_\star \in (0,1)$, we partition $\bb S^{p-1}$ into 
\[
    R_1(C_{\star})=\Brac{\mb q \in \bb S^{p-1}:\norm{\mb A^{T}\mb q}\infty^{2}\geq  C_{\star} },\qquad R_2(C_\star) = \bb S^{p-1} \setminus \Bigl(R_1(C_\star) \cup R_0\Bigr).
\]

\subsubsection{Geometric analysis for $q\in R_2$}\label{app_proof_R2}

We prove the following lemma which shows the existence of negative curvature for any $\mb q \in R_2$. 

\begin{lemma}\label{lem_R2}
    Assume $\theta < 1 / 3$.
    For any point $\mb q\in R_2(C_\star)$ with 
    \[
        C_\star \le {1-3\theta \over 2},
    \]
    there exists $\mb v$ such that
    \begin{align}
         \mb v^{T}\Hess f\paren{\mb q}\mb v < 0.
    \end{align}
    In particular, if $\theta \le 1/6$, for any point $\mb q\in R_2(C_\star)$ with 
    \[
        C_\star \le {1\over 3\sqrt 2},
    \]
    there exists $\mb v$ such that
    \begin{align}
         \mb v^{T}\Hess f\paren{\mb q}\mb v < -\frac{11-5\sqrt{2}}{9}\zinf^2.
    \end{align}
 \end{lemma}

\begin{proof}
Fix $C_\star$. Pick any $\bq \in R_2(C_\star)$ and write $\mb \zeta=\mb A^{T}\mb q$ for simplicity. Assume $|\zeta_i| = \zinf$ for some $i\in [r]$. Recall that $D_{\mb \zeta}^{\circ2}=\diag\paren{\mb \zeta^{\circ2}}$ with $\mb \zeta^{\circ2} = \{\zeta_j^2\}_{j\in [r]}$. From (\ref{eqn:Hess1And2}), we have 
\begin{align}\nonumber
    &\mb a_i^{T}\Hess_{\ell_4}f(\bq)\mb a_i\\ \nonumber
    &=(1-\theta)\brac{-3\mb a_i^T\mb A\mb D_{\zeta}^{\circ2} \mb A^{T}\mb a_i+6\mb \zeta_i\mb \zeta^{T}\mb D_{\zeta}^{\circ2}\mb A^{T}a_i-3\mb \zeta_i^{2}\norm{\mb \zeta}4^4-\norm{\mb \zeta}4^4(\mb \zeta_i^{2}-\norm{\mb a_i}2^2)  }\\\nonumber
    &=(1-\theta)\brac{-3\mb\zeta_i^2 +6\mb \zeta_i^4-3\mb \zeta_i^{2}\norm{\mb \zeta}4^4-\norm{\mb \zeta}4^4(\mb \zeta_i^{2}-1)  }\\\nonumber
    &=(1-\theta)\brac{-3\zinf^2 +6\zinf^4-4\zinf^{2}\norm{\mb \zeta}4^4 + \norm{\mb \zeta}4^4 }\\
    &\le (1-\theta)\brac{-2\zinf^2 +6\zinf^4-4\zinf^{6} }\label{ieq:R2_hessian_L_4}
\end{align}
where in the last line we used $\|\mb \zeta\|_4^4 \le \|\mb \zeta\|_2^2 \zinf^2 \le \zinf^2$ and $\|\mb \zeta\|_4^4 \ge \zinf^4$.
On the other hand, we obtain
\begin{align}\nonumber
    \mb a_{i}^{T}\Hess_{\ell_2}f(\bq)\mb a_i &=\theta\brac{-2\mb \zeta_i^2+6\norm{\mb \zeta}\infty\norm{\mb \zeta}2^2 \mb \zeta_i-4\norm{\mb \zeta}\infty^2\norm{\mb \zeta}2^4-\norm{\mb \zeta}2^2\norm{\ba_i^{T}\bA}2^2+\norm{\mb \zeta}2^4}\\\nonumber
    &\leq \theta\brac{-2\norm{\mb \zeta}\infty^2+6\norm{\mb \zeta}\infty^2-4\norm{\mb \zeta}\infty^6+\norm{\mb \zeta}2^2\paren{\norm{\mb \zeta}2^2-1}}\\
    &\leq \theta\brac{4\norm{\mb \zeta}\infty^2-4\norm{\mb \zeta}\infty^6}\label{ieq:R2_hessian_L_2}
\end{align}
where in the second and third lines we used $\norm{\mb \zeta}2^2\leq 1$. Combine (\ref{ieq:R2_hessian_L_4}) and (\ref{ieq:R2_hessian_L_2}) to obtain
\begin{align}\nonumber
    \mb a_{i}^{T}\Hess f(\bq) \mb a_i
    &\leq-4\norm{\mb \zeta}\infty^6+6\paren{1-\theta}\norm{\mb \zeta}\infty^4-2(1-3\theta)\norm{\mb \zeta}\infty^2\\
    &=-4\norm{\mb \zeta}\infty^2\biggl\{\norm{\mb \zeta}\infty^4-\frac{3\paren{1-\theta}}{2}\norm{\mb \zeta}\infty^2+{1-3\theta \over 2}\biggl\}\label{ieq:R2_hessian_nega2}
\end{align}
Define
\begin{align}
    g\paren{ x}=x^2-\phi x+\omega,\qquad \text{with }\quad
    \phi=\frac{3\paren{1-\theta}}{2},\quad\omega={1-3\theta \over 2}.
\end{align}
It remains to prove $\mb a_{i}^{T}\Hess f(\bq)\mb a_i\leq-4\norm{\mb \zeta}\infty^2g\paren{\zinf^2 } < 0.$ To this end, note that $\omega > 0$ under $\theta < 1/3$. 
Since
\begin{align}
    \phi^2-4\omega=\frac{9\paren{1-\theta}^2}{4}-2 + 6\theta =\paren{\frac{3\theta + 1}{2}}^2>0,
\end{align}
we know that, for all 
\begin{align}\label{ieq:R2_Linf_upper}
    \norm{\mb \zeta}\infty^2\leq\frac{\phi-\sqrt{\phi^2-4\omega}}{2} = {1-3\theta \over 2},
\end{align}
$g(\zinf^2) \ge 0$ and $g(\zinf^2)$ increases as $\zinf^2$ gets smaller. Recall that $\bq \in R_2(C_\star)$ implies 
$\zinf^2 < C_\star$. 
Thus, as long as 
\[
        C_\star \le  {1-3\theta \over 2},
\]
we conclude 
$g(\zinf^2) > g(C_\star) \ge  0$ hence 
\begin{align}
    \mb a_{i}^{T}\Hess f(\bq)\mb a_i\leq-4\norm{\mb \zeta}\infty^2g\paren{\zinf^2 } < 0.
\end{align}
This completes the proof of the first statement. The second one follows by taking $C_\star \le 1/(3\sqrt 2).$
\end{proof}

\subsubsection{Geometric analysis for $q\in R_1$}\label{app_proof_R1}
In this section we prove that any local solution to (\ref{obj_pop}) in $R_1$ recovers one column of $\mb A$, as stated in the following lemma. 

\begin{lemma}
    Assume $\theta < 1$. Any local solution $\bar{\mb q} \in R_1(C_\star)$ to (\ref{obj_pop}) with 
    $$
        C_\star > {1\over 2}\sqrt{\theta \over 1-\theta}
    $$  
    recovers one column of $\mb A$, that is, 
      $$
           \bar{\mb q} =\pm\mb A \mb e_i
      $$
      for some standard basis vector $\mb e_i$.
\end{lemma}

\begin{proof}
We prove the result by showing that any critical point of (\ref{obj_pop}) in $R_1(C_\star)$ is either a saddle point, or it satisfies the second order optimality condition and is equal to one column of $\mb A$.

Our proof starts by characterizing all critical points of (\ref{obj_pop}). 
For any critical point $\mb q$ of (\ref{obj_pop}), by writing $\mb \zeta = \mb A^T \mb q$, letting the gradient (\ref{eq_grad}) equal to zero gives
\begin{align}
 (1-\theta)\mb A\mb \zeta^{\circ3}-(1-\theta)\mb q\norm{\mb \zeta}4^4+\theta \norm{\mb \zeta}2^2\mb A\mb \zeta-\theta \mb q\norm{\mb \zeta}2^4  ~ = ~ 0.
\end{align}
Pick any $1\leq i\leq r$. Multiply both sides by $\mb a_i^{T}$ to obtain
\begin{align}
(1-\theta)\mb a_i^{T}\mb A\mb \zeta^{\circ3}-(1-\theta)\zeta_i\norm{\mb \zeta}4^4+\theta\norm{\mb \zeta}2^2 \mb a_{i}^{T}\mb A\mb \zeta- \theta\norm{\mb \zeta}2^4 \zeta_{i} ~ = ~ 0
\end{align}
with $\mb \zeta^{\circ3}$ means $\{\zeta_j^3\}_{j\in [r]}$. By using
\begin{align}
    &\mb a_i^{T}\mb A\mb \zeta^{\circ3} = \|\mb a_i\|_2^2 ~ \zeta_i^3 + \sum_{j\ne i}\innerprod{\mb a_i}{\mb a_j}\zeta_j^3 =  \zeta_i^3\\
    &\mb a_i^{T}\mb A \mb \zeta = \|\mb a_i\|_2^2 ~ \zeta_i + \sum_{j\ne i}\innerprod{\mb a_i}{\mb a_j}\zeta_j = \zeta_i, 
\end{align}
under Assumption \ref{ass_A_orth},
after a bit algebra and rearrangement, we obtain
\begin{align}\label{eq_zeta_cubic_equation}
\zeta_{i}^{3}-\alpha \zeta_{i}=0
\end{align}
where 
\begin{align}\label{eq_alpha_val} 
    \alpha &~ =~ \norm{\mb \zeta}4^4 + {\theta \over 1-\theta}\paren{\norm{\mb \zeta}2^2 - 1}\norm{\mb \zeta}2^2.
\end{align}
We then have that, for any critical point $\mb q \in R_1$, $\mb \zeta = \mb A^T \mb q$ satisfies (\ref{eq_zeta_cubic_equation}) for all $1\le i\le r$. Furthermore, since Lemma \ref{lem:alpha_and_beta_bound}, stated and proved in Section \ref{app_sec_lemm_R1}, shows that 
$
    \alpha> 0,
$
we conclude that
 $\mb \zeta$ belongs to one of the following three cases:
\begin{enumerate}
    \item \textbf{Case 1}\label{case:1}: %$\mb\zeta$ is entry-wise small in magnitude such that 
    $\zinf = 0$;
    \item \textbf{Case 2}\label{case:2}: There exists $i\in [r]$ such that 
    \[
            |\zeta_i| = \sqrt{\alpha},\qquad \zeta_j = 0,\quad \forall j\in [r]\setminus \{i\};
    \]
    \item \textbf{Case 3}\label{case:3}: There exists at least $i,j\in [r]$ with $i\ne j$ such that 
    \[
        |\zeta_i| = |\zeta_j| = \sqrt{\alpha}.
    \]
\end{enumerate}

Note that the definition of $R_1$ excludes $R_0$ defined in (\ref{def:obj}), hence rules out {\bf Case 1}. We then provide analysis for the other two cases separately. Specifically, 
for any $\mb \zeta$ belonging to {\bf Case 2}, Lemma \ref{lem:Case2} below proves that $\mb \zeta$ satisfies the second order optimality condition, hence is a local solution. Furthermore, $\mb \zeta$
is equal to one column of $\mb A$ up to the sign.

\begin{lemma}\label{lem:Case2}
   Let $\mb q$ be any critical point in $R_1(C_\star)$ and let $\mb \zeta = \mb A^T \mb q$. If there exists $i \in [r]$ such that
   \[
     |\zeta_i| = \sqrt{\alpha},\qquad |\zeta_j| = 0,\quad \forall j\in [r]\setminus \{i\},
   \]
   with $\alpha$ defined in (\ref{eq_alpha_val}),
   then there exists some signed permutation $\bP$ such that 
   \begin{align}\label{ieq:lowbound_for_zeta_l}
       \bq = \bA \bP_{\cdot 1}
   \end{align}
   Furthermore,
   \begin{align}
       \mb v^{T}\Hess f(\mb q)\mb v ~ \geq ~ (1-\theta)\|P_{\mb q}^\perp \mb v\|_2^2,\qquad \forall\mb v \textrm{ such that } P_{\mb q}^{\perp}\mb v \ne 0.
   \end{align}
   %which shows that the given critical point $\mb q$ satisfy the second order optimality.
\end{lemma}
     \begin{proof}
    Lemma \ref{lem:Case2} is proved in Section \ref{sec:Case2_proof}.
  \end{proof}

Finally, we show in Lemma \ref{lem:Case3} below that any $\mb \zeta$ belonging to {\bf Case 3} is a saddle point, hence is not a local solution.
\begin{lemma}\label{lem:Case3}

    For any critical point $\mb q\in R_1(C_\star)$ with  $\mb\zeta= \mb A^{T}\mb q$ and $\alpha$ as defined in (\ref{eq_alpha_val}), if there exists $k\paren{k\geq 2}$ non-zero elements such that 
    \[
        |\zeta_{\pi(1)}| =  |\zeta_{\pi(2)}| = \cdots= |\zeta_{\pi(k)}| = \sqrt{\alpha}
    \]
    for some permutation $\pi:[r] \to [r]$,
    then there exists $\mb v$ with $P_{\bq}^{\perp}\mb v \ne 0$ such that 
  \begin{align}
      \mb v^{T}\Hess f(\mb q) \mb v ~ \le ~   -\frac{2\paren{1-\theta}}{k} \|P_{\mb q}^\perp \mb v\|_2^2 ~ <~ 0.
  \end{align}
\end{lemma}
    \begin{proof}
    Lemma \ref{lem:Case3} is proved in Section \ref{sec:Case3_proof}.
  \end{proof}

Summarizing the above two lemmas conclude that all local solutions in $R_1$ lie in {\bf Case 2}, hence completes the proof of Lemma \ref{lem_R1}.
\end{proof}

\subsubsection{Additional lemmas used in Section \ref{app_proof_R1} and \ref{proof:thm:obj_sample}}\label{app_sec_lemm_R1}

% \begin{lemma}\label{lem:prop1}
%  For any $\mb q\in R_1(C_\star)$, the vector $\mb \zeta=\mb A^{T}\mb q$ satisfies
% \begin{align}
%     \frac{\norm{\mb\zeta}4^6}{\norm{\mb \zeta}3^3} ~ \geq 
%     ~ C_{\star}^3.
% \end{align}
% \begin{proof}
% For any $\mb\zeta$, we obtain
%   \begin{align}\label{ieq:zeta_4_6_over_zeta_3_3}
%   \frac{\norm{\mb \zeta}4^2}{\norm{\mb \zeta}3^1}\geq\frac{\norm{\mb \zeta}\infty^2}{\norm{\mb \zeta}3^1}\geq \frac{C_{\star}}{\norm{\mb \zeta}3^1}\geq C_\star,
%   \end{align}
%   where the first inequality is due to $\norm{\mb \zeta}4^4\geq\norm{\mb \zeta}\infty^4$, the second inequality invokes the definition of $R_1(C_\star)$ and the last inequality comes from $\norm{\mb \zeta}3^3\leq\norm{\mb \zeta}2^2\norm{\mb\zeta}\infty\leq \zinf \le 1$. The result follows immediately. 
% \end{proof}
% \end{lemma}

The following lemma gives the upper and low bounds for $\alpha$ defined in equation  (\ref{eq_alpha_val}).

\begin{lemma}\label{lem:alpha_and_beta_bound}
 For any $\mb q\in R_1(C_\star)$, let $\mb \zeta=\mb A^{T}\mb q$ and $\alpha$ be defined in (\ref{eq_alpha_val}). We have
  \begin{align}
      %\norm{\mb \zeta}4^4\brac{1+\frac{\theta}{1-\theta}\paren{1-\frac{1}{C_{\star}^2}}} 
      \norm{\mb \zeta}4^4\brac{1-\frac{\theta}{4\paren{1-\theta}C_\star^2}} ~ \le ~ \alpha ~ \le ~ \norm{\mb \zeta}4^4.
  \end{align}
  As a result, when 
  $$
    C_\star^2 > {\theta \over 4(1-\theta)},
  $$
  we have $\alpha > 0$.
  \end{lemma}
  \begin{proof}
  The upper bound of $\alpha$ follows from 
  \[
         \alpha= \norm{\mb \zeta}4^4\brac{1+\frac{\theta}{\paren{1-\theta}\norm{\mb \zeta}4^4}\paren{\norm{\mb \zeta}2^4-\norm{\mb \zeta}2^2}}\leq \norm{\mb \zeta}4^4
  \]
  by using $\norm{\mb \zeta}2^2\leq 1$ and $\norm{\mb \zeta}2^4\le\norm{\mb \zeta}2^2$. 
  To prove the lower bound, we have
      \begin{align}
          \alpha &= \norm{\mb \zeta}4^4\brac{1+\frac{\theta}{\paren{1-\theta}\norm{\mb \zeta}4^4}\norm{\mb \zeta}2^2\paren{\norm{\mb \zeta}2^2-1}}\nonumber\\
          &\geq\norm{\mb \zeta}4^4\brac{1-\frac{\theta}{\paren{1-\theta}\zinf^4}\norm{\mb \zeta}2^2\paren{1 - \norm{\mb \zeta}2^2}} & \text{by }\|\mb\zeta\|_2^2 \le 1\nonumber\\
          &\geq\norm{\mb \zeta}4^4\brac{1-\frac{\theta}{4\paren{1-\theta}\zinf^4}} & \text{by }\|\mb\zeta\|_2^2(1 - \|\mb\zeta\|_2^2)\le 1/4\nonumber\\
          &\geq\norm{\mb \zeta}4^4\brac{1-\frac{\theta}{4\paren{1-\theta}C_\star^2}}
      \end{align}
%       The last inequality uses $\zinf^2 \ge C_\star$ from the definition of $R_1(C_\star)$. %This completes the proof.
   
%   \sout{To prove the lower bound, we have}
%      \begin{align}
%          \cancel{\alpha &= \norm{\mb \zeta}4^4\brac{1+\frac{\theta}{\paren{1-\theta}\norm{\mb \zeta}4^4}\paren{\norm{\mb \zeta}2^4-\norm{\mb \zeta}2^2}}\nonumber\\
%          &=\norm{\mb \zeta}4^4\brac{1+\frac{\theta}{\paren{1-\theta}}\paren{\frac{\norm{\mb \zeta}2^4}{\norm{\mb \zeta}4^4}-\frac{\norm{\mb \zeta}2^2}{\norm{\mb \zeta}4^4}}}\nonumber\\
%          &\geq\norm{\mb \zeta}4^4\brac{1+\frac{\theta}{1-\theta}\paren{1-\frac{1}{C_{\star}^2}}}}
%      \end{align}
%      \sout{ The last inequality uses}
%      \begin{align}
%           \cancel{\frac{\norm{\mb \zeta}2^4}{\norm{\mb \zeta}4^4}\geq\frac{\norm{\mb \zeta}2^4}{\norm{\mb \zeta}\infty^2\norm{\mb \zeta}2^2}=\frac{\norm{\mb \zeta}2^2}{\norm{\mb \zeta}\infty^2}\geq1}
%      \end{align}
%       \sout{and}
%       \begin{align}
%           \cancel{\frac{\norm{\mb \zeta}2^2}{\norm{\mb \zeta}4^4}\leq \frac{1}{\norm{\mb \zeta}\infty^4}\leq \frac{1}{C_{\star}^2}.  }  
%         \end{align}
%       \sout{ by the definition of $R_1(C_\star)$.}
     \end{proof}

%  {\color{gray}
% \begin{lemma}[Lemma $B.6$, \cite{qu2019analysis}]
% \label{lemma_3_7}
%   Suppose $\left|\innerprod{\mb a_i}{\mb a_j}\right|\leq \mu \textrm{ and }\mu\leq \frac{1}{2}$, Let $\mb v\in span \{a_l,a_m\}$ such that $\norm{\mb v}2^2=1 \textrm{ and } \mb v=c_l\mb a_l+c_m\mb a_m$ then we have
%   \begin{flalign}
%      c_l^2+c_m^2+\paren{c_1^2+c_2^2}\paren{\mb a_l^T\mb a_m}^2+4c_lc_m\paren{\mb a_l^T\mb a_m}\geq 1-10\mu.
%   \end{flalign}
%  \end{lemma}
%  }

\subsubsection{Proof of Lemma \ref{lem:Case2}}\label{sec:Case2_proof}

   \begin{proof}
   Let $\mb q$ be any critical point in $R_1(C_\star)$ with $C_{\star} > 0$. Write $\mb \zeta = \mb A^T \mb q$ and suppose
   \[
     |\zeta_\ell| = \sqrt{\alpha},\qquad |\zeta_j| = 0,\quad \forall j\in [r]\setminus \{\ell\},
   \]
   with $\alpha$ defined in (\ref{eq_alpha_val}). Our proof contains two parts. We first show that $\mb q = \mb a_l$ (we assume $\bP$ is identity for simplicity) and then show that $\mb q$ satisfies the second order optimality condition. 

      \paragraph{Recovery of $\ba_\ell$:} First notice that
      \begin{align}
      \zeta_l^2=\alpha=\norm{\mb \zeta}4^4\brac{1+\frac{\theta}{\norm{\mb \zeta}4^4\paren{1-\theta}}\paren{\norm{\mb \zeta}2^4-\norm{\mb \zeta}2^2}}
      \end{align} 
      Since $\norm{\mb \zeta}2^2=\zeta_l^2$ and $\norm{\mb \zeta}4^4=\zeta_l^4$, we immediately have
      \begin{align}
          \alpha=\alpha^2\brac{1+\frac{\theta}{\alpha^2\paren{1-\theta}}\paren{\alpha^2-\alpha}}.
      \end{align}
      Solving it gives $\alpha=1$, which implies 
      $\zeta_{\ell}^2 = |\innerprod{\mb a_l}{\mb q}|^2 = 1$, as desired.
      
\paragraph{Second order optimality:}
    We prove 
        \[
            \mb v^{T}\Hess {f}(\mb q)\mb v=\mb v^{T}\brac{\Hess_{\ell_2}{f}(\mb q)+\Hess_{\ell_4}{f}(\mb q)}\mb v  > 0
        \]
        for all $\mb v$ such that $P_{\mb q}^{\perp}\mb v \ne 0$.
    
        Recall from (\ref{eqn:Hess1And2}) that 
        $$
        \Hess_{\ell_4}{f}(\mb q) ~ =~  -(1-\theta)P_{\mb q^\perp}\brac{
            3 \sum_{j=1}^r \mb a_j\mb a_j^T (\mb q^T \mb a_j)^2  -  \|\mb q^T\mb A\|_4^4 \mb I }P_{\mb q^\perp}.
        $$
        Without loss of generality, let $\mb v\in \bb S^{p-1}$ be any vector such that $\mb v\perp \mb q$. Recall that $\mb\zeta = \bA^T \bq$. Then
        \begin{align}
            \mb v^{T}\Hess_{\ell_4}{f}(\mb q)\mb v &~=~ 
            %(1-\theta)\brac{ -3\mb v^{T}\mb A \mb D_{\mb \zeta}^{\circ 2}\mb A\mb v+\norm{\mb \zeta}4^4}\nonumber\\
            (1-\theta)\brac{-3\sum_{j=1}^r \paren{\mb a_j^{T}\mb v}^{2}\zeta_j^{2}+\norm{\mb \zeta}4^4}\nonumber\\
            &~ =~ (1-\theta)\brac{-3\paren{\mb a_\ell ^{T}\mb v}^{2} + 1}.
        \end{align}
        where we used $\zeta_\ell^2 = 1$ and $\zeta_j = 0$ for all $j\ne \ell$ together with $\|\mb \zeta\|_4^4 = 1$ in the second line. 
        In addition, we find
        \begin{align}
            \paren{\mb a_\ell ^{T}\mb v}^{2}  =\left|\innerprod{\mb a_l}{\mb v}\right|^2&=\left|\innerprod{\mb q}{\mb v}\right|^2=0
        \end{align}
        so that 
        \begin{align}\label{eq:case2_hess2_low1}
            \mb v^{T}\Hess_{\ell_4}f(\mb q)\mb v& =1-\theta.
        \end{align}
        On the other hand,
       \begin{align}\label{ieq:case2_hess2_low2}\nonumber
            \mb v^{T}\Hess_{\ell_2}(\mb q)\mb v& ~ = ~ \theta\brac{-2\paren{\mb v^T\mb A\mb \zeta}^2-\norm{\mb \zeta}2^2\norm{\mb A\mb v}2^2+\norm{\mb \zeta}2^4}\\\nonumber
            & ~ = ~ \theta\brac{-2\paren{\mb a_l^{T}\mb v\zeta_l}^2-\norm{\mb A\mb v}2^2+ 1}\\\nonumber
             & ~ = ~ \theta\brac{1 -\norm{\mb A\mb v}2^2}\\
             &~ \geq ~ 0
        \end{align}
        where we used $\lambda_1(\mb A\mb A^T) \le 1$ in the last line. 
        Combine equation (\ref{eq:case2_hess2_low1}) and inequality (\ref{ieq:case2_hess2_low2}) to obtain
        \begin{align}
            \mb v^{T}\Hess {f}(\mb q)\mb v ~ \ge ~ 1-\theta >  0,
        \end{align}
        completing the proof.
   \end{proof}

\subsubsection{Proof of Lemma \ref{lem:Case3}}\label{sec:Case3_proof}

  \begin{proof}
  Let $\bq$ be any critical point $\mb q \in R_1(C_\star)$ with $C_\star > 0$ and  $\mb \zeta=\mb A^{T}\mb q$ having at least $k$ non-zero entries for $2\le k\le r$. Without loss of generality, we assume
\begin{align}
    |\zeta_j|=\sqrt{\alpha} \qquad \forall j\leq k, \qquad \zeta_j=0 \qquad \forall j> k.
\end{align}
We show there exists $\mb v$ such that 
\[
     \mb v^{T}\Hess{f}(\mb q) \mb v =\mb v^{T}\paren{\Hess_{\ell_2}{f}(\mb q)+\Hess_{\ell_4}{f}(\mb q) } \mb v=-\frac{2\paren{1-\theta}}{k} \|P_{\mb q}^{\perp}\mb v\|_2^2 < 0.
\]
Without loss of generality, pick any vector $\mb v\in \bb S^{p-1}$ satisfying $\mb v\perp \mb q$ and $\mb v$ lies in the span of $\{\mb a_1, \mb a_2,\cdots,\mb a_k\}$. Write $\mb v=\sum_{j=1}^{k}c_{j}\mb a_j$. From (\ref{eq_hess}), we have
    \begin{align}\label{eqn_hess_v_case3}
        \mb v^{T}\Hess_{\ell_4}f(\mb q)\mb v 
        &~=~ (1-\theta)\brac{ -3\mb v^{T}\mb A \mb D_{\mb \zeta}^{\circ 2}\mb A\mb v+\norm{\mb \zeta}4^4} ~ =~ \paren{1-\theta}\brac{-3\sum_{j=1}^{k}\paren{\mb a_j^{T}\mb v}^2\zeta_j^2+\norm{\mb \zeta}4^4}.
    \end{align}
   Recall from the definition of $\alpha$ in (\ref{eq_alpha_val}) that
  \begin{align}
      \alpha=\norm{\mb \zeta}4^4\brac{1+\frac{\theta}{\norm{\mb \zeta}4^4\paren{1-\theta}}\paren{\norm{\mb \zeta}2^4-\norm{\mb \zeta}2^2}},
  \end{align}
  using $\norm{\mb \zeta}4^4=\sum_{j=1}^{r}\zeta_j^4=\sum_{j=1}^{k}\zeta_j^4=k\alpha^2$ and $\norm{\mb \zeta}2^2=\sum_{j=1}^{r}\zeta_j^2=\sum_{j=1}^{k}\zeta_j^2=k\alpha$ yields \begin{align}
      \alpha=k\alpha^2\brac{1+\frac{\theta}{k\alpha^2\paren{1-\theta}}\paren{k\alpha^2-k\alpha}}.
  \end{align}
  Solve the equation above to obtain 
  $
      \alpha= 1 / k,
  $
  hence 
  $$
    \|\mb \zeta\|_4^4 = {1\over k},\qquad |\zeta_j|^2 = {1\over k},\quad \forall j\le k.
  $$
  Plugging this into (\ref{eqn_hess_v_case3}) gives 
    \begin{align}\label{eq:case3_hess_L_4}\nonumber
        \mb v^{T}\Hess_{\ell_4}f(\mb q)\mb v 
        &~ =~ \paren{1-\theta}\brac{-\frac{3}{k}\sum_{j=1}^{k}\paren{\mb a_j^{T}\mb v}^2+\frac{1}{k^2}}\\\nonumber
        &~ =~ \paren{1-\theta}\brac{-\frac{3}{k}\sum_{j=1}^{k}c^2_j+\frac{1}{k}}\\
        & ~=~ -\frac{2\paren{1-\theta}}{k} 
    \end{align}
    where the second equality used $\sum_{j=1}^{k}(\mb a_j^{T}\mb v)^2 = \sum_{j=1}^{k}c^2_j=1$. 
    
    On the other hand, we have
    \begin{align}\nonumber
        \mb v^{T}\Hess_{\ell_2}f(\bq) \mb v 
        &~ =~ \theta\brac{-2\paren{\mb v^{T}\mb A\mb \zeta}^2-\norm{\mb \zeta}2^2\norm{\mb A^{T}\mb v}2^2+\norm{\mb \zeta}2^4}&\\\nonumber
        &~ \leq~  \theta\brac{-\norm{\mb \zeta}2^2\norm{\mb A^{T}\mb v}2^2+\norm{\mb \zeta}2^4}\\\nonumber
        &~ =~ \theta\brac{-\norm{\mb \zeta}2^2\sum_{j=1}^{k}c_j^2+\norm{\mb \zeta}2^4} \\
        &~ =~ 0.\label{ieq:case3_hess2_upper}
    \end{align}
    The second equation follows from $\norm{\mb A^{T}\mb v}2^2=\sum_{j=1}^{k}(\mb v^{T}\mb a_{j})^2=\sum_{j=1}^{k}c_j^2 = 1$ and the last step uses $\|\mb \zeta\|_2^2 = 1$.  Combining equation (\ref{eq:case3_hess_L_4}) and (\ref{ieq:case3_hess2_upper}) gives
    \begin{align}
        \mb v^{T}\Hess{f}(\mb q)\mb v \le -\frac{2\paren{1-\theta}}{k}
    \end{align}
    and completes the proof. 
  \end{proof}

  \newpage

  \subsection{Proof of Theorem \ref{thm:obj_sample}}\label{proof:thm:obj_sample}
  To prove Theorem \ref{thm:obj_sample}, analogous to (\ref{def:obj}), we give a new partition of $\bb S^{p-1}$ as
    \begin{align}\label{def:obj_prime}
    R_0' &\doteq R_0'(c_\star) =\Brac{
        \mb q \in \bb S^{p-1}:\norm{\mb A^{T}\mb q}\infty^2 \leq c_\star
    },\\\nonumber
    R_{1}^{'} &\doteq R_1(C_{\star})=\Brac{\mb q \in \bb S^{p-1}:\norm{\mb A^{T}\mb q}\infty^{2}\geq  C_{\star} },\\\nonumber
    R_{2}^{'} & = \bb S^{p-1} \setminus \paren{R_0' \cup R_1'}.
    \end{align}
    Here $c_\star$ and $C_{\star}$ are positive constants satisfying $0\leq c_\star \leq C_{\star}<1$.
    
    Let $\delta_1$ and $\delta_2$ be some positive sequences to be determined later. Define the random event 
    \begin{align}\label{def_event_E}
        \mathcal E = \Brac{
        \sup_{\bq \in \Sp^{p-1}}\norm{\grad f\paren{\mb q}-\grad F\paren{\mb q}}2\lesssim \delta_{1},~ \sup_{\bq \in \Sp^{p-1}}\|\Hess f\paren{\mb q}-\Hess F\paren{\mb q}\|_{\op} \lesssim \delta_{2}}.
    \end{align}
    Here $\grad f(\bq)$ and $\grad F(\bq)$ are the gradients of (\ref{obj_pop}) and (\ref{obj}), respectively, at any point $\bq\in \Sp^{p-1}$. Similarly, $\Hess f(\bq)$ and $\Hess F(\bq)$ are the corresponding Hessian matrices. 
    
    On the event $\mathcal E$, 
    the results of Theorem \ref{thm:obj_sample} immediately follow from the two lemmas below. Lemma \ref{lem:sample_R2_nega_short} shows that the objective $F\paren{\mb q}$ in (\ref{obj}) exhibits negative curvature at any point $\bq \in R_{2}^{'}$. Meanwhile, Lemma \ref{lem:R1_three_case_short} proves that any critical point in region $R_{1}^{'}$ is either a solution that is close to the ground-truth, or a saddle point with negative curvature that is easy to escape by any second-order descent algorithm. Lemmas \ref{lem:sample_R2_nega_short} and \ref{lem:R1_three_case_short} are proved in Section \ref{sample_proof_R2} and \ref{sample_proof_R1}, respectively.

    \begin{lemma}[Optimization landscape for $R_{2}^{'}$]\label{lem:sample_R2_nega_short}
          Assume $\theta \le 6/25$.
    For any point $\mb q\in R_{2}^{'}(C_\star)$ with 
    \[
        C_\star \le \frac{12}{25}-2\theta,
    \]
    if $\delta_2\leq c_\star / 25$ for some constant $c_\star \in (0, C_\star)$, then
    there exists $\mb v$ such that
    \begin{align}
         \mb v^{T}\Hess F\paren{\mb q}\mb v < 0.
    \end{align}
    In particular, if $\theta \le 1/9$, for any point $\mb q\in R_2(C_\star)$ with 
    $
        C_\star \le 1/4,
    $
    there exists $\mb v$ such that
    \begin{align}
         \mb v^{T}\Hess F\paren{\mb q}\mb v < -\frac{21}{100}\norm{\mb \zeta}\infty^2<-\frac{1}{5}c_\star.
    \end{align}
  \end{lemma}

  \begin{lemma}[Optimization landscape for $R_{1}^{'}$]\label{lem:R1_three_case_short} 
  Assume 
    \begin{align}
     \theta\leq 1/9,\quad \delta_1\leq 5\times 10^{-5}, \quad \delta_2\leq 10^{-3}.
    \end{align}
  Any local solution $\bar \bq \in R_1'(C_\star)$ with 
  $
     C_{\star}\geq 1/5
  $
  satisfies
    \begin{align}
     \norm{\bar{\mb q}-\mb A \bP_{\cdot 1}}2^2 ~ \leq~ C\delta_1%\leq0.038
  \end{align}
      for some signed permutation matrix $\bP$ and some constant $C>0$.
 \end{lemma}
 
 Finally, the proof of Theorem \ref{thm:obj_sample} is completed by invoking Lemmas \ref{lem:gradient_concentration} and \ref{lem:hessian_concentration} and using condition (\ref{cond_n_orth}) to establish that 
    $\delta_2 \le c_\star / 25$, $\delta_2\le 10^{-3}$ and $\delta_1 \le 5\times 10^{-5}$. Indeed, we have 
    \begin{align}
        \delta_1 & =  \sqrt{r^2 \log (M_n) \over \theta n} + {M_n\over n} {r \log (M_n) \over n},\\
        \delta_2 & = \sqrt{r^3 \log (M_n) \over \theta n} + {M_n\over n} {r \log (M_n) \over n}
    \end{align}
    where $M_n = C(n+r)\left(\theta r^2 + {\log^2 n / \theta}\right)$. Under (\ref{cond_n_orth}), we have $\log (M_n) \lesssim \log n$ whence $\delta_2 \le \min\{c_\star / 25, 10^{-3}\}$ requires 
    \[
        n \ge C\max \left\{
            {r^3\log n \over \theta c_\star^2},~ \left(\theta r^2 + {\log^2 n \over \theta}\right){r\log n \over c_\star}
        \right\}
    \]
    for sufficiently large $C>0$, which holds under (\ref{cond_n_orth}).

  \subsubsection{Proof of Lemma \ref{lem:sample_R2_nega_short}}\label{sample_proof_R2} 
  \begin{proof}
    Fix $C_\star$. Pick any $\bq \in R_2'(C_\star)$ and write $\mb \zeta=\mb A^{T}\mb q$ for simplicity. Assume $|\zeta_i| = \zinf$ for some $i\in [r]$. %Recall that $D_{\mb \zeta}^{\circ2}=\diag\paren{\mb \zeta^{\circ2}}$.
    Note that, on the event $\mathcal E$,
    \begin{align}
        \mb a_i^{T} \Hess F(\bq) \mb a_i&\leq\mb a_i^{T}\Hess f\paren{\mb q}\mb a_i+\mb a_i^{T}\brac{\Hess F\paren{\mb q}-\Hess f\paren{\mb q}}\mb a_i\nonumber\\
        &\leq\mb a_i^{T}\Hess f\paren{\mb q}\mb a_i+\delta_2\nonumber\\
        &\leq\mb a_i^{T}\Hess f\paren{\mb q}\mb a_i+\frac{1}{25}\norm{\mb \zeta}\infty^2
    \end{align}
    where in the last inequality we used $\delta_2 \le c_\star / 25 \le \zinf^2 / 25$ as $\bq \in R_2'(C_\star)$. By inequality (\ref{ieq:R2_hessian_nega2}), we obtain 
    \begin{align}
        \ba_i^{T}\Hess f\paren{\mb q}\mb a_i\leq-4\norm{\mb \zeta}\infty^2\biggl\{\norm{\mb \zeta}\infty^4-\frac{3\paren{1-\theta}}{2}\norm{\mb \zeta}\infty^2+{1-3\theta \over 2}\biggl\},
    \end{align}
    hence
    \begin{align}
         \mb a_i^{T} \Hess {F}(\bq) \mb a_i & \leq-4\norm{\mb \zeta}\infty^2\biggl\{\norm{\mb \zeta}\infty^4-\frac{3\paren{1-\theta}}{2}\norm{\mb \zeta}\infty^2+{1-3\theta \over 2}-\frac{1}{100}\biggl\}
    \end{align}
    Define
\begin{align}
    g\paren{ x}=x^2-\phi x+\omega,\qquad \text{with }\quad
    \phi=\frac{3\paren{1-\theta}}{2},\quad\omega={1-3\theta \over 2}-\frac{1}{100}.
\end{align}
It remains to prove $\mb a_{i}^{T}\Hess F(\bq)\mb a_i\leq-4\norm{\mb \zeta}\infty^2g\paren{\zinf^2 } < 0.$ To this end, note that $\omega > 0$ under $\theta < 1/4$. 
Since
\begin{align}
    \phi^2-4\omega=\frac{9\paren{1-\theta}^2}{4}-2 + 6\theta+\frac{1}{25} =\paren{\frac{3\theta + 1}{2}}^2+\frac{1}{25}>0,
\end{align}
we know that, for all 
\begin{align}\label{ieq:R2_Linf_upper}
    \norm{\mb \zeta}\infty^2&\leq\frac{\phi-\sqrt{\phi^2-4\omega}}{2} = \frac{\frac{3-3\theta}{2}-\paren{\frac{3\theta+1}{2}}\sqrt{1+\frac{1}{25\paren{\frac{3\theta+1}{2}}^2}}}{2} \doteq r_-
\end{align}
$g(\zinf^2) \ge 0$ and $g(\zinf^2)$ increases as $\zinf^2$ gets smaller.
Recall that $\bq \in R_2(C_\star)$ implies 
$\zinf^2 < C_\star$. We then have 
\[
    g(\zinf^2) > g(C_\star). 
\]
We proceed to show $C_\star \le r_-$ by noticing that
\begin{align}\nonumber
    r_- &\geq\frac{\frac{3-3\theta}{2}-\paren{\frac{3\theta+1}{2}}\brac{1+\frac{1}{50\paren{\frac{3\theta+1}{2}}^2}}}{2}\\\nonumber
    &\geq\frac{\frac{3-3\theta}{2}-\paren{\frac{3\theta+1}{2}}\brac{1+\frac{4}{50}}}{2}\\\nonumber
    &=\frac{1-3\theta-\frac{12}{100}\theta-\frac{1}{25}}{2}\\
    &>\frac{12}{25}-2\theta.
\end{align}
Thus, provided that
\[
        C_\star \le \frac{12}{25}-2\theta,
\]
we conclude 
$g(\zinf^2) > g(C_\star) \ge  0$ hence 
\begin{align}
    \mb a_{i}^{T}\Hess F(\bq)\mb a_i\leq-4\norm{\mb \zeta}\infty^2g\paren{\zinf^2 } < 0.
\end{align}
In particular, taking $\theta\leq 1/9$ and $C_{\star}\leq1/4$ yields
\begin{align}
     \mb a_{i}^{T}\Hess F(\bq)\mb a_i\leq-\frac{21}{100}\norm{\mb \zeta}\infty^2.
\end{align}
This completes the proof.
  \end{proof}

  \subsubsection{Proof of Lemma \ref{lem:R1_three_case_short}}\label{sample_proof_R1}
 
 \begin{proof}
 The proof of this lemma is similar in spirit to that of Lemma \ref{app_proof_R1}. Follows the notations there, any critical point $\bq \in R_1'(C_\star)$ satisfies
 \begin{align}
     \grad {f}(\bq) +\grad {F}(\bq) -\grad {f}(\bq) =0.
 \end{align}
 Following the same procedure of proving Lemma \ref{app_proof_R1}, analogous to (\ref{eq_zeta_cubic_equation}), we obtain
\begin{align}\label{eq:sample_zeta_cubic_equation}
    \zeta_i^3-\alpha\zeta_i+\beta=0
\end{align}
for any $i\in [r]$, where $\mb \zeta = \bA^{T}\bq$,
\begin{align}\label{eq:sample_alpha_beta}
    \alpha=\norm{\mb \zeta}4^4+\frac{\theta}{1-\theta}\paren{\norm{\mb \zeta}2^4-\norm{\mb \zeta}2^2},\qquad \beta\doteq \beta_i=\innerprod{\grad f\paren{\mb q}-\grad F\paren{\mb q}}{\mb a_i}.
\end{align}
To further characterize $\mb \zeta$ satisfying (\ref{eq:sample_zeta_cubic_equation}), note that $\alpha > 0$ from Lemma \ref{lem:alpha_and_beta_bound} and we also prove in Lemma \ref{lem:alpha/beta<1/4}, stated and proved in Section \ref{app_proof_R1_sample}, that 
$
    4|\beta| < \alpha^{3/2}.
$
In conjunction with Lemma \ref{lem:cubic_function} in Section \ref{app_proof_R1_sample}, we conclude that $\mb \zeta$ belongs to one of the following three cases:
\begin{itemize}
    \item \textbf{Case 1}: %$\mb\zeta$ is entry-wise small in magnitude such that 
    $$
        \abs{\zeta_i} \leq {2|\beta| \over  \alpha},\quad \forall 1\le i\le r;
    $$
    \item \textbf{Case 2}: There exists $i\in [r]$ such that 
    \[
           |\zeta_i| \ge \sqrt{\alpha} - {2|\beta|\over \alpha},\qquad |\zeta_j| 
           < \sqrt{\alpha} - {2|\beta|\over \alpha},\quad \forall j\in [r]\setminus \{i\};
    \]
    \item \textbf{Case 3}: There exists at least $i,j\in [r]$ with $i\ne j$ such that 
    \[
        |\zeta_i| \ge \sqrt{\alpha} - {2|\beta|\over \alpha}, \qquad |\zeta_j| \ge \sqrt{\alpha} - {2|\beta|\over \alpha}.
    \]
\end{itemize}

We provide analysis case by case.  {\bf Case 1} is ruled out by Lemma \ref{lem:Case1_sample} below. 
%Specifically, for any $\mb \zeta$ belonging to {\bf Case 1}, we prove in Lemma \ref{lem:Case1_sample} below that the existence of such $\mb \zeta$ contradicts with the fact that $\bq \in R_1'$. Thus, {\bf Case 1} can not happen. 
For any $\mb \zeta$ belonging to {\bf Case 2}, Lemma \ref{lem:Case2_sample} below proves that $\mb \zeta$ satisfies the second order optimality condition, hence is a local solution. Furthermore, $\mb q$
is close to one column of $\mb A$.
Finally, Lemma \ref{lem:Case3_sample} shows that any $\mb \zeta$ belonging to {\bf Case 3} is a saddle point, hence is not a local solution.
Summarizing the Lemmas \ref{lem:Case1_sample} -- \ref{lem:Case3_sample} concludes that all local solutions in $R_1$ lie in {\bf Case 2}, hence concludes the proof of lemma \ref{lem:R1_three_case_short}. Lemmas \ref{lem:Case1_sample} -- \ref{lem:Case3_sample} are proved in Sections \ref{sec:sample_Case1_proof}, \ref{sec:sample_Case2_proof} and \ref{sec:sample_Case3_proof}, respectively. 
\end{proof}

\begin{lemma}\label{lem:Case1_sample}
    Assume 
    $$
    \theta\leq 1/9,\qquad  \delta_1\leq 10^{-4}.
    $$
    For any critical point $\mb q \in R_{1}^{'}(C_\star)$ with 
    {$C_{\star}\geq1/5 $}, there exists at least one $i \in [r]$ such that 
    $$ 
        |\zeta_i| ~ > ~  \frac{2|\beta|}{\alpha}
    $$
    where $\mb \zeta=\mb A^{T}\mb q$ and $\alpha$ and $\beta$ are defined in (\ref{eq:sample_alpha_beta}).
  \end{lemma}

\begin{lemma}\label{lem:Case2_sample}
    Assume 
          \begin{align}
       \theta\leq 1/9,\quad \delta_1\leq 5\times 10^{-5}, \quad \delta_2\leq 10^{-3}.
    \end{align}
   Let $\mb q$ be any critical point in $R_{1}^{'}(C_\star)$ with {$C_{\star}\geq 1/5$}. If there exists $i \in [r]$ such that
   \[
      |\zeta_i| \ge  \sqrt{\alpha} - {2|\beta|\over \alpha},\qquad |\zeta_j| \le {2|\beta|\over \alpha},\quad \forall j\in [r]\setminus \{i\},
   \]
   with $\mb \zeta = \mb A^T \mb q$ and $\alpha$ and $\beta$ defined in (\ref{eq:sample_alpha_beta}),
   then
   \begin{align}\label{ieq:lowbound_for_zeta_l}
       \norm{\mb q - \bA \bP_{\cdot 1}}2^2 ~ \leq~ C\delta_1
   \end{align}
   for some signed permutation matrix $\bP$ and some constant $C>0$.
   Furthermore,
   \begin{align}
       \mb v^{T}\Hess f(\mb q)\mb v ~ > ~  0,\qquad \forall\mb v \textrm{ such that } P_{\mb q}^{\perp}\mb v \ne 0.
   \end{align}
   %which shows that the given critical point $\mb q$ satisfy the second order optimality.
\end{lemma}

\begin{lemma}\label{lem:Case3_sample}
    Assume 
    \begin{align}
       \theta\leq 1/9,\qquad \delta_1\leq 10^{-4}, \qquad \delta_2\leq 10^{-3}.
    \end{align}
    For any critical point $\mb q\in R_{1}^{'}(C_\star)$ with {$C_{\star}\geq 1/5$}, if there exists $i, j\in [r]$ with $i\ne j$ such that 
    \[
        |\zeta_i| \ge \sqrt{\alpha} - {2|\beta|\over \alpha}, \qquad |\zeta_j| \ge \sqrt{\alpha} - {2|\beta|\over \alpha},
    \]
    where $\mb\zeta= \mb A^{T}\mb q$ and $\alpha$ and $\beta$ are defined in (\ref{eq:sample_alpha_beta}),
    then there exists $\mb v$ with $P_{\bq}^{\perp}\mb v \ne 0$ such that 
  \begin{align}
      \mb v^{T}\Hess f(\mb q) \mb v ~ \leq   -0.00315\|P_{\mb q}^\perp \mb v\|_2^2 ~ <~ 0.
  \end{align}
\end{lemma}

\subsubsection{Lemmas used in Section \ref{sample_proof_R1}}\label{app_proof_R1_sample}

%{\color{gray}
%\begin{lemma}\label{lem:prop1}
% For any $\mb q\in R_1'(C_\star)$ in (\ref{def:obj}), $\mb \zeta=\mb A^{T}\mb q$ satisfies
%\begin{align}
%    \frac{\norm{\mb\zeta}4^6}{\norm{\mb \zeta}3^3}\geq C_{\star}^3
%\end{align}
%\begin{proof}
%For any $\mb\zeta = \bA^T \bq$ with $\bq \in R_1'(C_\star)$, we have
%  \begin{align}\label{ieq:zeta_4_6_over_zeta_3_3}
%   \frac{\norm{\mb \zeta}4^2}{\norm{\mb \zeta}3^1}\geq\frac{\norm{\mb \zeta}\infty^2}{\norm{\mb \zeta}3^1}\geq \frac{C_{\star}}{\norm{\mb \zeta}3^1}\geq C_{\star}.
%  \end{align}
%  The first inequality uses $\norm{\mb \zeta}4^4\geq\norm{\mb \zeta}\infty^4$, the second inequality the definition of $R_1'(C_\star)$ in (\ref{def:obj}), and the last inequality comes from $\norm{\mb \zeta}3^3\leq\norm{\mb \zeta}2^2\norm{\mb\zeta}\infty\leq1$ . The result follows by taking the cubic of (\ref{ieq:zeta_4_6_over_zeta_3_3}).
%\end{proof}
%\end{lemma}

%\begin{lemma}
%$C_{\star}^2\geq\frac{8\theta}{7\paren{1-\theta}}$ obtains when
%          \begin{align}
%       \theta\leq 1/9,\quad C_{\star}\geq 1/5.
%    \end{align}
%Further we have:
%\begin{align}
%    1-\frac{\theta}{4\paren{1-\theta}C_{\star}^2}\geq\frac{7}{32}>1/5
%\end{align}
% \end{lemma}
%  \begin{proof}
%   Direct calculation.
%  \end{proof}
%}

\begin{lemma}
\label{lem:alpha/beta<1/4}
    Assume $\theta \le 1/9$. 
    For any critical point $\bq \in R_1'(C_\star)$ with $C_\star \ge 1/5$, on the event $\mathcal{E}$ in (\ref{def_event_E}), we have 
    $$
        4|\beta| ~ <~  \alpha^{3/2}
    $$ 
    where $\beta$ and $\alpha$ are defined in (\ref{eq:sample_alpha_beta}).
    
    %denote $\mb \zeta=\mb A^{T}\mb q$, when               Assume 
    %      \begin{align}
    %   \delta_1\leq 10^{-5}\norm{\mb \zeta}3^3 \quad \delta_2\leq 10^{-5}\norm{\mb \zeta}4^4,\quad\theta\leq 1/9.
    %\end{align} 
   \end{lemma}
   \begin{proof}
     By definition and the event $\mathcal E$,
     \begin{equation}\label{disp_beta_upper_bound}
        |\beta| = |\innerprod{\grad f\paren{\mb q}-\grad F\paren{\mb q}}{\mb a_i}| \le \delta_1 \|\ba_i\|_2 = \delta_1.
     \end{equation}
     Then by Lemma \ref{lem:alpha_and_beta_bound}, 
     \[
         \frac{|\beta|}{\alpha^{3/2}}  \leq\frac{\delta_1}{\alpha^{3/2}} \leq \frac{\delta_1}{\norm{\mb \zeta}4^6\brac{1-\frac{\theta}{4\paren{1-\theta}C^2_{\star}}}^{\frac{3}{2}}}.
     \]
     Since $\bq \in R_1(C_\star)$ implies 
     $\zinf^2\ge C_\star$, using $\|\mb \zeta\|_4^6 \ge \zinf^6 \ge C_\star^3$ together with  $C_\star \ge 1/5$ and $\theta \le 1/9$ gives 
     \begin{equation}\label{lem:C_low_bound}
         \norm{\mb \zeta}4^6\brac{1-\frac{\theta}{4\paren{1-\theta}C^2_{\star}}}^{\frac{3}{2}} \ge \brac{C_\star^2-\frac{\theta}{4\paren{1-\theta}}}^{\frac{3}{2}} \ge \paren{7\over 800}^{3/2}
     \end{equation}
     The result follows from $\delta_1 < 2\times 10^{-4}$.
   \end{proof}

\begin{lemma}[Lemma $B.3$, \cite{qu2019analysis}]
\label{lem:cubic_function}
Considering the cubic function
\begin{align}
    f(x)=x^3-\alpha x+\beta
\end{align}
When $\alpha \geq 0$ and $4|\beta| \leq \alpha^{3/2}$, the roots of the function $f(\cdot)$ are contained in the following union of the intervals.
\begin{align}
    \Brac{ |x|\leq \frac{2|\beta|}{\alpha}}\bigcup\Brac{ |x-\sqrt{\alpha}|\leq \frac{2|\beta|}{\alpha}}\bigcup\Brac{ |x+\sqrt{\alpha}|\leq \frac{2|\beta|}{\alpha}}.
\end{align}
\end{lemma}

  \subsubsection{Proof of Lemma \ref{lem:Case1_sample}}\label{sec:sample_Case1_proof}
    
    \begin{proof} We prove that for any critical point $\mb q\in R_{1}^{'}$, there exists at least one $i\in [r]$ such that 
    $|\zeta_i| > 2|\beta| /\alpha$ 
    with $\alpha$ and $\beta$ being defined in (\ref{eq:sample_alpha_beta}). Suppose 
    \[
        |\zeta_i| \le {2|\beta| \over \alpha},\quad \forall i\in [r].
    \]
    Assume $|\zeta_k| = \zinf$ for some $k\in [r]$. We obtain 
    $$
    \norm{\mb \zeta}\infty \leq\frac{2|\beta|}{\alpha}
    $$
    hence, by also using $\|\mb\zeta\|_2\le 1$, 
       \begin{align}\label{Case1 main}
            \norm{\mb \zeta}4^4&\leq \norm{\mb \zeta}\infty^2\norm{\mb \zeta}2^2\leq\frac{4\beta^2}{\alpha^2} \leq \frac{4\delta_1^2}{\norm{\mb \zeta}4^{12}\brac{1-\frac{\theta}{4\paren{1-\theta}C^2_{\star}}}^2}\norm{\mb \zeta}4^4
            \overset{(\ref{lem:C_low_bound})}{\le} 4\delta_1^2 \paren{800\over 7}^3\norm{\mb \zeta}4^4.
        \end{align}
        This is a contradiction whenever
        \[
            \delta_1 \le {1\over 2}\paren{7\over 800}^{-3/2},
        \]
        which is the case if $\delta_1 \le 10^{-4}$.
\end{proof}

       \subsubsection{Proof of Lemma \ref{lem:Case3_sample}}\label{sec:sample_Case3_proof}

  \begin{proof}
    Let $\bq$ be any critical point $\mb q \in R_{1}^{'}(C_\star)$ with $C_\star \ge 1/5$ and write $\mb \zeta=\mb A^{T}\mb q$. Suppose there exists $l, m\in [r]$ with $l\ne m$ such that
  \[
        |\zeta_l| > \sqrt{\alpha} - {2|\beta|\over \alpha}, \qquad |\zeta_m| > \sqrt{\alpha} - {2|\beta|\over \alpha}.
    \]
        We prove there exist $\mb v$ such that.
      \begin{align}
           \mb v^{T}\Hess F(\bq) \mb v \leq \mb v^{T}\Hess f(\bq) \mb v+ \mb v^T\brac{\Hess F(\bq) -\Hess f(\bq)}\mb v <0.
      \end{align}
      %Note that $c_1:=\frac{\delta_1}{\norm{\mb \zeta}3^3}$ and $c_2:=\frac{\delta_2}{\norm{\mb \zeta}4^4}$ are used in derivation during section \ref{sec:sample_Case3_proof}.
      Pick any vector $\mb v\in \Sp^{p-1}$ such that $\mb v\perp \mb q$ and $\mb v$ lies in the span of $\{\mb a_l, \mb a_m\}$, that is, $\mb v=c_l\mb a_l+c_m\mb a_m$ for some $c_l^2 + c_m^2 = 1$. Recall  
      from (\ref{eq_hess}) that 
      \[
        \mb v^{T}\Hess f(\bq) \mb v = \mb v^{T}\Hess_{\ell_4}f(\bq) \mb v+\mb v^{T}\Hess_{\ell_2} f(\bq) \mb v.
      \]
      By (\ref{eqn:Hess1And2}), we first have
    \begin{align}\label{ieq:case3_hess_L_4}\nonumber
        &\mb v^{T}\Hess_{\ell_4}f(\bq) \mb v\\\nonumber
        &=\paren{1-\theta}\brac{-3\paren{\mb a_l^T\mb v}^2\zeta_l^2-3\paren{\mb a_m^T\mb v}^2\zeta_m^2-3\sum_{k\neq l,k\neq m}\paren{\mb A^T\mb v}_k^2\zeta_k^2+\norm{\mb \zeta}4^4}\\\nonumber
        &\leq\paren{1-\theta}\Brac{-3\brac{\paren{\mb a_l^T\mb v}^2+\paren{\mb a_m^T\mb v}^2}\min\{\zeta_l^2,\zeta_m^2\}+\norm{\mb \zeta}4^4}\\\nonumber
        &\leq \paren{1-\theta}\Brac{-3\brac{c_l^2+c_m^2+\paren{c_1^2+c_2^2}\paren{\mb a_l^T\mb a_m}^2+4c_lc_m\paren{\mb a_l^T\mb a_m}}\min\{\zeta_l^2,\zeta_m^2\}+\norm{\mb \zeta}4^4}\\
        &=\paren{1-\theta}\Brac{-3\min\Brac{\zeta_l^2,\zeta_m^2}+\norm{\mb \zeta}4^4}.
    \end{align}
    Here $\norm{\mb v}2^2=c_l^2+c_m^2=1$ and $\mb a_l^T\mb a_m=0$ are used in last step derivation.
    Note that 
    \begin{align}
        \zeta_l^2&\geq\paren{\sqrt{\alpha}-\frac{2|\beta|}{\alpha}}^2\geq\alpha-\frac{4|\beta|}{\sqrt{\alpha}}.
    \end{align}
    Lemma \ref{lem:alpha_and_beta_bound} gives 
    \[
        \alpha \ge \norm{\mb \zeta}4^4-\frac{\theta}{1-\theta}\brac{\norm{\mb \zeta}2^2-\norm{\mb \zeta}2^4}.
    \]
    Also by (\ref{disp_beta_upper_bound}), we obtain
    \begin{align}\label{disp_beta_sqrt_alpha_bound}
        \frac{4|\beta|}{\sqrt{\alpha}} & \le \frac{4
        \delta_1}{\norm{\mb \zeta}4^{2}\brac{1-\frac{\theta}{4\paren{1-\theta}C^2_{\star}}}^{\frac{1}{2}}}
        \le 4\delta_1\paren{800 \over 7}^{1/2}
        \le  {4\delta_1\over C_\star^2}\paren{800 \over 7}^{1/2}\norm{\mb \zeta}4^4
    \end{align}
    where the second inequality is due to (\ref{lem:C_low_bound}) and the last one uses $\norm{\mb \zeta}4^4\ge \zinf^4 \ge C_\star^2$. By writing 
    \begin{equation}\label{def_eta}
        \eta  ~ \doteq ~ {4\delta_1\over C_\star^2}\paren{800 \over 7}^{1/2},
    \end{equation}
    it follows that  
    \begin{align}
        \zeta_l^2 \ge \norm{\mb \zeta}4^4\paren{1-\eta}-\frac{\theta}{\paren{1-\theta}}\brac{\norm{\mb \zeta}2^2-\norm{\mb \zeta}2^4}.
    \end{align}
    This lower bound also holds for $\min\{\zeta_l^2,\zeta_m^2\}$. Plugging it in (\ref{ieq:case3_hess_L_4}) yields
    \begin{align}\label{ieq:case3_hess1_upper}\nonumber
        \mb v^{T}\Hess_{\ell_4}f(\bq) \mb v &\leq \paren{1-\theta}\Brac{-3\Brac{\norm{\mb \zeta}4^4\paren{1-\eta }-\frac{\theta}{\paren{1-\theta}}\brac{\norm{\mb \zeta}2^2-\norm{\mb \zeta}2^4}}+\norm{\mb \zeta}4^4}\\
        &=\paren{1-\theta}\Brac{\paren{-2+3 \eta }\norm{\mb \zeta}4^4+\frac{3\theta}{1-\theta}\brac{\norm{\mb \zeta}2^2-\norm{\mb \zeta}2^4}}.
    \end{align}
    On the other hand, from (\ref{eqn:Hess1And2}), we have 
    \begin{align}\label{ieq:case3_hess2_upper}\nonumber
        \mb v^{T}\Hess_{\ell_2}f(\bq) \mb v &=\theta\brac{-2\paren{\mb v^{T}\mb A\mb \zeta}^2-\norm{\mb \zeta}2^2\norm{\mb A^{T}\mb v}2^2+\norm{\mb \zeta}2^4}\\\nonumber
        &\leq \theta\brac{-\norm{\mb \zeta}2^2\norm{\mb A^{T}\mb v}2^2+\norm{\mb \zeta}2^4}\\
        &=\theta\brac{-\norm{\mb \zeta}2^2+\norm{\mb \zeta}2^4}.
    \end{align}
    Third inequality uses $\norm{\mb A^{T}\mb v}2^2=\sum_{j=1}^{r}(\mb v^{T}\mb a_{j})^2=c_l^2+c_m^2=\norm{\mb v}2^2=1$. 
    Combine (\ref{ieq:case3_hess2_upper}) and (\ref{ieq:case3_hess1_upper}) to obtain
    \begin{align}
        \mb v^{T}\Hess f(\bq) \mb v&\leq\paren{1-\theta}\Brac{\paren{-2+ 3\eta}\norm{\mb \zeta}4^4+\frac{3\theta}{1-\theta}\brac{\norm{\mb \zeta}2^2-\norm{\mb \zeta}2^4}}+\theta\brac{-\norm{\mb \zeta}2^2+\norm{\mb \zeta}2^4}\nonumber\\
        &\leq\paren{1-\theta}\paren{-2+ 3\eta}\norm{\mb \zeta}4^4+2\theta\brac{\norm{\mb \zeta}2^2-\norm{\mb \zeta}2^4}\nonumber\\
        &\leq\paren{1-\theta}\paren{-2+ 3\eta}\norm{\mb \zeta}4^4+\frac{\theta}{2}.
    \end{align}
    Here $\norm{\mb \zeta}2^2-\norm{\mb \zeta}2^4\leq 1/4 $ is used in last step. We thus conclude that 
    \begin{align}
        \mb v^{T}\Hess F(\mb q)\mb v
        &\leq\mb v^{T}\Hess_f(\mb q)\mb v+\left\|\Hess F(\bq) -\Hess f(\bq)\right\|_{\op}\nonumber\\
        &\leq\paren{1-\theta}\paren{-2+ 3\eta}\norm{\mb \zeta}4^4+\frac{\theta}{2}+\delta_2
    \end{align}
    on the event $\mathcal{E}$. 
    Note that $-2+3\eta \leq 0$ from  
    $C_{\star}\geq 1/5$ and $\delta_1 < 10^{-4}$.
    By using $\norm{\mb \zeta}4^4\geq \zinf^4 \ge C_{\star}^2$, we obtain
    \begin{align}
        \mb v^{T}\Hess F(\mb q)\mb v& ~ \leq~ \paren{1-\theta}\paren{-2+ 3\eta}C_{\star}^2+\frac{\theta}{2}+\delta_2\nonumber\\
        & ~ = ~ \paren{1-\theta}\paren{-2C_{\star}^2+ {12\delta_1}\paren{800 \over 7}^{1/2}}+\frac{\theta}{2}+\delta_2 &\text{by }(\ref{def_eta}).
    \end{align}
    Recalling that 
    $$\theta\leq 1/9,\quad C_{\star}\geq 1/5,\quad \delta_1\leq 10^{-4},\quad \delta_2\leq 10^{-3},$$ 
    we further have 
    \begin{align}
        \mb v^{T}\Hess F(\mb q)\mb v&\leq -0.00315\|P_{\mb q}^\perp \mb v\|_2^2 = -0.00315 <0.
    \end{align}
    This completes the proof. 
  \end{proof}

  \subsubsection{Proof of Lemma \ref{lem:Case2_sample}}\label{sec:sample_Case2_proof}

   \begin{proof}
   This proof contains two parts: the first part shows that any critical $\mb q \in R_1'$ is close to the ground truth vector $\mb a_l$ for some $l\in[m]$, and the second part proves the second order optimality for this $\mb q$.
   
   %Note that $c_1:=\frac{\delta_1}{\norm{\mb \zeta}3^3}$ and $c_2:=\frac{\delta_2}{\norm{\mb \zeta}4^4}$ are used in derivation during section \ref{sec:sample_Case2_proof}
     \paragraph{Closeness to the target ground-truth vector:}
      Pick any critical point $\bq\in R_{1}^{'}$ and suppose that, for some $l\in [m]$,  
      $$
      \zeta_l\geq\sqrt{\alpha}-\frac{2|\beta|}{\alpha},\quad \zeta_j\leq \frac{2|\beta|}{\alpha},\quad \forall j\ne \ell. 
      $$ 
      On the one hand, we bound $\zeta_l^4$ from below as 
      \begin{align}\label{ieq:zeta_inf_lower}
        \zeta_l^4= \norm{\mb \zeta}4^4-\sum_{k\neq l}\zeta_k^4
        &\geq\norm{\mb \zeta}4^4-\sum_{k\neq l}\zeta_k^2 \max_{k\neq l}\zeta_k^2\nonumber\\
        &\geq\norm{\mb \zeta}4^4- \norm{\mb \zeta}2^2 {4\beta^2 \over \alpha^2}\nonumber\\
        &\geq\norm{\mb \zeta}4^4- 4\delta_1^2 \paren{800\over 7}^3\norm{\mb \zeta}4^4  & \text{by (\ref{Case1 main})}.
      \end{align}
      The last inequality also uses $\sum_{k\neq l}\zeta_k^2\leq \norm{\mb \zeta}2^2\leq1$. On the other hand, 
      since Lemma \ref{lem_beta_i}, stated and proved below, ensures that
      \[
            |\beta_l| \le \|\ba_l-\bq\|_2 \delta_1 := \Delta_l \delta_1,
      \]
      the upper bound of $\zeta_l^2$ follows from 
      \begin{align}\label{ieq:zeta_inf_upper}\nonumber
          \zeta_l^2\leq \paren{\sqrt{\alpha}+\frac{2|\beta_l|}{\alpha}}^2
          &= \alpha +\frac{4\beta_l^2}{\alpha^2} + {4|\beta_l| \over \sqrt \alpha}\\
          &\le \norm{\mb \zeta}4^4+ 4\delta_1^2 \paren{800\over 7}^3  \norm{\mb \zeta}4^4+ \Delta_l \eta \norm{\mb \zeta}4^4
      \end{align}
      where we also use Lemma \ref{lem:alpha_and_beta_bound}, (\ref{Case1 main}) and (\ref{disp_beta_sqrt_alpha_bound}) and recall from (\ref{def_eta}) that 
      \[
        \eta ~ \doteq ~ {4\delta_1\over C_\star^2}\paren{800 \over 7}^{1/2}.
      \]
      Define 
      \begin{equation}\label{def_xi}
        \xi ~ \doteq ~4\delta_1^2 \paren{800\over 7}^3.
      \end{equation}
        Combine (\ref{ieq:zeta_inf_lower}) and (\ref{ieq:zeta_inf_upper}) to obtain
      \begin{align}\label{ieq:zeta_inf_lower_med_1}
          \zeta_l^2=\frac{\zeta_l^4}{\zeta_l^2} &\geq {1-\xi \over 1 + \xi + \Delta_l\eta} =  1 - {2\xi + \Delta_l\eta \over 1+\xi+\Delta_l\eta}
          % &=\frac{1-\frac{100c_1^2\norm{\mb \zeta}3^6}{\norm{\mb \zeta}4^{12}}}{1+\frac{2c_1\norm{\mb \zeta}3^3}{\norm{\mb \zeta}4^{6}\brac{1-\frac{\theta}{4\paren{1-\theta}C^2_{\star}}}}+\frac{4c_1\norm{\mb \zeta}3^3}{\norm{\mb \zeta}4^{6}\brac{1-\frac{\theta}{4\paren{1-\theta}C^2_{\star}}}}}\geq\frac{1-\frac{100c_1^2\norm{\mb \zeta}3^6}{\norm{\mb \zeta}4^{12}}-\frac{800c_1^2\norm{\mb \zeta}3^6}{\norm{\mb \zeta}4^{12}}}{1+\frac{6c_1\norm{\mb \zeta}3^3}{\norm{\mb \zeta}4^{6}\brac{1-\frac{\theta}{4\paren{1-\theta}C^2_{\star}}}}}\geq\frac{1-\frac{900c_1^2\norm{\mb \zeta}3^6}{\norm{\mb \zeta}4^{12}}}{1+\frac{30c_1\norm{\mb \zeta}3^3}{\norm{\mb \zeta}4^{6}}}\\
          %&=1-\frac{30c_1\norm{\mb \zeta}3^3}{\norm{\mb \zeta}4^{6}}\geq1-\frac{30c_1}{C^3_{\star}}
      \end{align}
      which implies 
      \[
        1 -  |\zeta_l| \le {2\xi + \Delta_l\eta }.
      \]
      Consequently, assuming $\zeta_l = \ba_l^T \bq > 0$ without loss of generality, we have
     \begin{align}\label{upper_bd_error}
       \norm{\mb a_l-\mb q}2^2&= \norm{\mb a_l}2^2+\norm{\mb q}2^2-2\zeta_l=2(1-|\zeta_l|) \leq  4\xi + 2 \|\ba_l - \bq\|_2\eta,
     \end{align}
     implying the desired result. 
   %Considering the case that $\theta\leq 1/9,\quad C_{\star}\geq 1/5,\quad c_1=\frac{\delta_1}{\norm{\mb \zeta}3^3}\leq 10^{-5},\quad c_2=\frac{\delta_2}{\norm{\mb \zeta}4^4}\leq 10^{-5}$, we could easily obtain
  % \begin{align}
   %   \norm{\mb a_l-\mb q}2^2&\leq0.038.
   %\end{align}
   
   \paragraph{Second order optimality:}
      We show that
      \begin{align}
          \mb v^{T}\Hess F\paren{\mb q}\mb v\geq0. 
      \end{align}
      for any $\mb v\in \bb S^{p-1}$ that $P_{\mb q}^{\perp}\mb v \ne 0$. Pick any $\mb v\in \Sp^{p-1}$ such that $\mb v\perp \mb q$. We have
      \begin{align}
         \mb v^{T} \Hess F\paren{\mb q}\mb v&=\mb v^{T} \Hess f\paren{\mb q}\mb v+\mb v^{T} \brac{\Hess F\paren{\mb q}-\Hess f\paren{\mb q}}\mb v\nonumber\\
         &\geq\mb v^{T} \Hess f\paren{\mb q}\mb v-\norm{\Hess F\paren{\mb q}-\Hess f\paren{\mb q}}\op\nonumber\\
         &=\mb v^{T} \brac{\Hess_{\ell_4} f\paren{\mb q}+\Hess_{\ell_2} f\paren{\mb q}}\mb v-\norm{\Hess F\paren{\mb q}-\Hess f\paren{\mb q}}\op.
      \end{align}
      We bound from below $\mb v^{T} \Hess_{\ell_4}f\paren{\mb q}\mb v$ and $\mb v^{T} \Hess_{\ell_2}f\paren{\mb q}\mb v$ respectively. 
      
      Recall from (\ref{eqn:Hess1And2}) that 
      $$
        \Hess_{\ell_4}f(\bq) = -(1-\theta)P_{\mb q^\perp}\brac{
            3 \sum_{j=1}^r \mb a_j\mb a_j^T (\mb q^T \mb a_j)^2  -  \|\mb q^T\mb A\|_4^4 \mb I }P_{\mb q^\perp}.
      $$
      Also recall that $\zinf = |\zeta_l|$. We have
        \begin{align}\label{eqn_hessian_ell_4}
            \mb v^{T}\Hess_{\ell_4}f(\bq)  \mb v&=(1-\theta)\brac{ -3\mb v^{T}\mb A \mb D_{\mb \zeta}^{\circ 2}\mb A\mb v+\norm{\mb \zeta}4^4}\nonumber\\
            &=(1-\theta)\brac{-3\sum_{k=1}^r \paren{\mb A^{T}\mb v}_{k}^{2}\zeta_k^{2}+\norm{\mb \zeta}4^4}\nonumber\\
            &=(1-\theta)\Brac{-3\brac{(\mb A^{T}\mb v)_{l}^{2}\zeta_l^2+\sum_{k\neq l}\paren{\mb A^{T}\mb v}_{k}^{2}\zeta_k^{2}}+\norm{\mb \zeta}4^4}.
        \end{align}
        Note from (\ref{upper_bd_error}) that 
        \begin{align}\label{disp_upper_bd_v_a_square}
            (\mb A^{T}\mb v)_{l}^{2}=\left|\innerprod{\mb a_l}{\mb v}\right|^2&=\left|\innerprod{\mb a_l-\mb q}{\mb v}\right|^2 \leq \norm{\mb a_l-\mb q }2^2 =2(1-|\zeta_l|)
           \leq 2\paren{1-\zeta_l^2}.
        \end{align}
        Also note that 
        \[
            \sum_{k\neq l}\paren{\mb A^{T}\mb v}_{k}^{2}\zeta_k^{2} \le {4\beta^2 \over \alpha^2}\|\bA^T \mb v\|_2^2 \le {4\beta^2 \over \alpha^2} \overset{(\ref{Case1 main})}{\le} \xi \cdot \|\mb \zeta\|_4^4
        \]
        with $\xi$ defined in (\ref{def_xi}).
        We thus have 
        \begin{align}\label{ieq:case2_hess1_1}
            \mb v^{T}\Hess_{\ell_4}(\mb q)\mb v&\geq \paren{1-\theta}\brac{-6\paren{1-\zeta_l^2}\zeta_l^2-3\xi\norm{\mb \zeta}4^4+\norm{\mb \zeta}4^4}.
        \end{align}
        Since (\ref{ieq:zeta_inf_upper}) ensures
        \begin{align}\label{ieq:zeta_inf_upper_2}
             \zeta_l^2\leq  (1 +\xi + \eta)\norm{\mb \zeta}4^4,
        \end{align}
        we further obtain 
        \begin{align}\label{ieq:case2_hess1_3}\nonumber
             \mb v^{T}\Hess_{\ell_4}(\mb q)\mb v
             &\geq \paren{1-\theta}\norm{\mb \zeta}4^4\brac{-6(1 +\xi + \eta)\paren{1-\zeta_l^2}-3\xi + 1}\\\nonumber
             &\geq \paren{1-\theta}\norm{\mb \zeta}4^4\brac{-6(2\xi + \eta)-3\xi + 1}\\
             &= \paren{1-\theta}\paren{1-15\xi - 6\eta} \norm{\mb \zeta}4^4.
        \end{align} 
        The last step uses (\ref{ieq:zeta_inf_lower_med_1}) again.
        
        To bound from below $\mb v^{T} \Hess_{\ell_2}f\paren{\mb q}\mb v$, by (\ref{eqn:Hess1And2}), we have
        \begin{align}\label{eq_case2_hess2_low1}
            \mb v^{T}\Hess_{\ell_2}(\mb q)\mb v ~ =~ \theta\brac{-2\paren{\mb v^T\mb A\mb \zeta}^2-\norm{\mb \zeta}2^2\norm{\mb A\mb v}2^2+\norm{\mb \zeta}2^4}.
        \end{align}
        To upper bound $\paren{\mb v^T\mb A\mb \zeta}^2$, observe that
        \begin{align}
            \paren{\mb v^T \mb A\mb \zeta}^2=\brac{\mb a_l^{T}\mb v\zeta_l+\sum_{k\neq l}\ba_k^T\mb v \zeta_k}^2
            &\leq2\brac{\paren{\mb v^T\mb a_l}^2 \zeta_l^2 +\paren{\sum_{k\neq l}\ba_k^T\mb v \zeta_k}^2}\nonumber\\
            &\leq2\brac{2(1-\zeta_l^2)\zeta_l^2+\sum_{k\neq l}(\ba_k^T\mb v)^2 \sum_{k\neq l}\zeta^2_k}\nonumber\\
            &\leq2\brac{2(1-\zeta_l^2)\zeta_l^2+ \norm{\mb \zeta}2^2-\zeta^2_l }
        \end{align}
        where we used (\ref{disp_upper_bd_v_a_square}) and Cauchy-Schwarz inequality in the second line, and the fact that $\sum_{k\ne \ell}(\bA^T \mb v)_k^2 \le \|\bA^T \mb v\|_2^2 \le 1$ in the third line. By observing that 
         \begin{align}\label{bd_upper_zeta_2}
            \norm{\mb \zeta}2^2 = \frac{\norm{\mb \zeta}2^2}{\norm{\mb \zeta}4^4}\norm{\mb \zeta}4^4\leq\frac{\norm{\mb \zeta}2^2}{\zeta_l^4}\norm{\mb \zeta}4^4 \le {\norm{\mb \zeta}4^4  \over \zeta_l^4}
            %\overset{(\ref{ieq:zeta_inf_lower_med_1})}{\le} {1\over (1- \Delta)^2} \norm{\mb \zeta}4^4
        \end{align}
        and 
        $$
        \zeta_{l}^2 \le   {\norm{\mb \zeta}4^4 \over \zeta_l^4}\zeta_l^2,
        $$
        %with 
        %\[
        %    \Delta = {2\xi + \eta \over 1+\xi + \eta}.
        %\]
        we have 
        \[
             \norm{\mb \zeta}2^2-\zeta^2_l \le {\norm{\mb \zeta}4^4 \over \zeta_l^4}(1-\zeta_l^2)
        \]
        which further yields
        \begin{align}
            \paren{\mb v^T \mb A\mb \zeta}^2
            &\leq2\brac{2(1-\zeta_l^2)\zeta_l^2+ {\norm{\mb \zeta}4^4 \over \zeta_l^4}(1-\zeta_l^2) }\nonumber\\
            &\leq {2(2\xi + \eta) \over 1+\xi+\eta} \brac{2\zeta_l^2+ {\norm{\mb \zeta}4^4 \over \zeta_l^4}} & \text{by (\ref{ieq:zeta_inf_lower_med_1})}\nonumber\\
            &\le {2(2\xi + \eta) \over 1+\xi+\eta} \brac{2(1+\xi+\eta)+ {1 \over \paren{1 - {2\xi + \eta \over 1+\xi + \eta}}^2}}\norm{\mb \zeta}4^4  & \text{by (\ref{ieq:zeta_inf_lower_med_1}) and (\ref{ieq:zeta_inf_upper_2})}\nonumber\\
            &\le  2(2\xi + \eta)\brac{2 +   {1\over 1 - 3\xi - \eta}}\norm{\mb \zeta}4^4. 
        \end{align}
        On the other hand, we have  
        \[
            \norm{\mb \zeta}2^2\norm{\mb A\mb v}2^2\le \norm{\mb \zeta}2^2 \le {1 \over \paren{1 - {2\xi + \eta \over 1+\xi + \eta}}^2}\norm{\mb \zeta}4^4
        \]
        by (\ref{bd_upper_zeta_2}) and (\ref{ieq:zeta_inf_lower_med_1}), 
        and $\|\mb\zeta\|_2^4 \ge \|\mb\zeta\|_4^4$. It then follows that 
        \begin{align}\label{ieq:case2_hess2_low2}\nonumber
            \mb v^{T}\Hess_{\ell_2}(\mb q)\mb v &\ge  \theta \norm{\mb \zeta}%%% ====== Math Texts/Sets/Brackets 
4^4 \brac{ - 4(2\xi + \eta)\brac{2 +   {1\over 1 - 3\xi - \eta}}- {1 \over \paren{1 - {2\xi + \eta \over 1+\xi + \eta}}^2} + 1
            }\\
            & \ge  - 16.5\theta \norm{\mb \zeta}4^4\paren{
             2\xi +\eta
             }
        \end{align}
        where the last line follows from $\delta_1 \le 5\times 10^{-5}$ and $C_\star \ge 1/5$ together with some simple algebra.
        Combine (\ref{ieq:case2_hess1_3}) and (\ref{ieq:case2_hess2_low2}) to obtain
        \begin{align}
             \mb v^{T} \Hess_{\ell_4}f(\bq) \mb v&\geq\norm{\mb \zeta}4^4  \brac{\paren{1-\theta}\paren{1-15\xi - 6\eta} - 16.5\theta\paren{2\xi +\eta} }
             %biggl[1-\theta-\paren{390+570\theta}\frac{c_1}{C^3_{\star}}\biggr].
        \end{align}
        whence, on the event $\mathcal E$,
        \begin{align}
            \mb v^{T}\Hess F(\mb q)\mb v&\geq \norm{\mb \zeta}4^4  \brac{\paren{1-\theta}\paren{1-15\xi - 6\eta} -16.5\theta\paren{2\xi +\eta}} - \delta_2\nonumber\\
            &\geq  C_\star^2\brac{\paren{1-\theta}\paren{1-15\xi - 6\eta} - 16.5\theta\paren{2\xi +\eta} } - \delta_2\nonumber\\
            &> 0
        \end{align}
        by using $\delta_1 \le 5\times 10^{-5}$ and $\delta_2<10^{-3}$. 
        %$\theta\leq 1/9,\quad C_{\star}\geq 1/5,\quad c_1\leq 10^{-5},\quad c_2\leq 10^{-5}$ are used here.
        The proof is complete. 
   \end{proof}
   \begin{lemma}\label{lem_beta_i}
        Let $\beta_i$ and $\delta_1$ be defined in (\ref{eq:sample_alpha_beta}) and (\ref{def_event_E}), respectively. Then, 
        \[
            |\beta_i| \le \|\ba_i - \bq\|_2 \delta_1.
        \]
    \end{lemma}
    \begin{proof}
        By the gradients in (\ref{grad_F}) and (\ref{grad_f}), both $\grad F\paren{\mb q}$ and $\grad f\paren{\mb q}$ lie in the space of $P_{\mb q^\perp}$. We immediately have 
        \[
            |\beta_i| = |\innerprod{\grad f\paren{\mb q}-\grad F\paren{\mb q}}{\mb a_i}| = |\innerprod{\grad f\paren{\mb q}-\grad F\paren{\mb q}}{\mb a_i - \bq}| \le \|\ba_i - \bq\|_2 \delta_1,
        \]
        as desired.
    \end{proof}
   \subsection{Proof of Proposition \ref{prop_precond}}\label{app_proof_prop_precond}

    \begin{proof}
    Write the eigenvalue decomposition of $\mb Y\mb Y^T = \bU \bLbd \bU^T $ with $\bU = [\bu_1, \ldots, \bu_r]$ and $\bLbd$ contains the first $r$ eigenvalues (in non-increasing order). By the definition of the Moore-Penrose inverse, we have 
    \[
        \bD = \bU \bLbd^{-1/2}\bU^T
    \]
    such that 
    \[
        \bar \bY = \bD\bY = \bU \bLbd^{-1/2}\bU^T\bA\bX.
    \]
    Here $\bLbd^{-1/2}$ is the diagonal matrix with diagonal elements equal to the reciprocals of the square root of those of $\bLbd$. Further write the SVD of $\bA$ as $\bA = \bU_A \bD_A \bV_A^T$ with $\bU_A^T\bU_A = \bI_r$ and $\bD_A$ being diagonal and containing non-increasing singular values. Since $\bU_A = \bU \mb Q$ for some orthogonal matrix $\mb Q\in \R^{r\times r}$,  we obtain
    \begin{align}\label{eqn_D}
        \bD  = \bU_A\mb Q^T \bLbd^{-1/2}\mb Q \bU_A =  \bU_A {\bD_A^{-1}\over \sqrt{n\sigma^2\theta}} \bU_A^T + \bU_A\left(\mb Q^T \bLbd^{-1/2}\mb Q - {\bD_A^{-1}\over \sqrt{n\sigma^2\theta}}\right) \bU_A^T.
    \end{align}
    It then follows that 
    \begin{align}
        \bar \bY &= \bU_A {\bD_A^{-1}\over \sqrt{n\sigma^2\theta}} \bU_A^T\bA \bX  + \bU_A\left(\mb Q^T \bLbd^{-1/2}\mb Q - {\bD_A^{-1}\over \sqrt{n\sigma^2\theta}}\right) \bU_A^T\bA \bX \nonumber\\
        &= \bU_A \bV_A^T \oX  + \bU_A\left(\mb Q^T \bLbd^{-1/2}\mb Q - {\bD_A^{-1}\over \sqrt{n\sigma^2\theta}}\right) \bD_A\bV_A^T \bX \nonumber\\
        &= \bU_A \bV_A^T \oX  + \bU_A\bV_A^T \bV_A\left(\sqrt{n\sigma^2\theta}\mb Q^T \bLbd^{-1/2}\mb Q \bD_A - \bI_r\right) \bV_A^T \oX \nonumber\\
        &= \oA \oX  + \oA \bV_A\left(\sqrt{n\sigma^2\theta}\mb Q^T \bLbd^{-1/2}\mb Q \bD_A - \bI_r\right) \bV_A^T \oX\nonumber\\
        &= \oA\oX + \oA\mb\Delta \oX
        % &= \sqrt{n\theta\sigma^2}\bU \bLbd^{-1/2}\bU^T \bU_A \bD_A \bV_A^T \bX\\
        %  & = \sqrt{n\theta\sigma^2}\bU \bLbd^{-1/2} \mb Q \bD_A \bV_A^T \bX\\
        %  & = \bU \bV_A^T \bX + \bU\brac{
        %   \sqrt{\theta n\sigma^2} \bLbd^{-1/2} \mb Q \bD_A  - \bI_r
        %  }\bV_A^T \bX\\
        %  & := \bA' \bX + \bA' \bV_A \brac{
        %   \sqrt{\theta n\sigma^2} \bLbd^{-1/2} \mb Q \bD_A  - \bI_r
        %  }\bV_A^T \bX.
    \end{align}
    where we used $\oX = \bX / \sqrt{n\sigma^2\theta}$ and $\oA = \bU_A\bV_A^T$ and write 
        $$
            \mb\Delta = \bV_A \brac{
          \sqrt{\theta n\sigma^2} \mb Q^T \bLbd^{-1/2} \mb Q \bD_A  - \bI_r
         }\bV_A^T
         $$
         and it remains to bound from above $\|\mb\Delta\|_{\op}$. Note that 
    \begin{align}
        \bU \bLbd \bU^T  = \bY\bY^T &= \bA\paren{n\sigma^2 \theta \bI_r + \bX\bX^T - n\sigma^2 \theta\bI_r} \bA^T\nonumber\\
        & = \bU_A\bD_A\bV_A^T \paren{n\sigma^2 \theta\bI_r + \bX\bX^T - n\sigma^2 \theta\bI_r}\bV_A \bD_A \bU_A^T.
    \end{align}
    It then follows by using $\bU_A = \bU \mb Q$ that 
    \[
        \mb Q^T \bLbd \mb Q  - n\sigma^2 \theta\bD_A^2  = \bD_A \bV_A^T \paren{\bX\bX^T - n\sigma^2 \theta\bI_r}\bV_A \bD_A,
    \]
    hence
    \[
        {1\over \theta n\sigma^2}\bD_A^{-1}\mb Q^T \bLbd \mb Q \bD_A^{-1}  -  \bI_r  = \bV_A^T \paren{{1\over \theta n\sigma^2}\bX\bX^T -  \bI_r}\bV_A.
    \]
    Let $\lambda_k$ denote the largest $k$th eigenvalue of the left hand side of the above equation, for $1\le k\le r$. Then Weyl's inequality guarantees
    \[
        \max_k|\lambda_k - 1|  ~ \le
        ~ \left\|
            {1\over \theta n\sigma^2}\bX\bX^T -  \bI_r
        \right\|_{\op}.
    \]
    Clearly, 
    \begin{align}
        \|\mb\Delta\|_{\op} &= \left\|
            \sqrt{\theta n\sigma^2}  \mb Q^T \bLbd^{-1/2} \mb Q \bD_A  - \bI_r
        \right\|_{\op}\nonumber\\
        &= \max_k\left|
            {1\over \sqrt{\lambda_k}} -1 
        \right|\nonumber\\
        &= \max_k{|1-\lambda_k| \over \sqrt{\lambda_k}(1 + \sqrt{\lambda_k})}\nonumber\\
        &\le  \max_k{|1-\lambda_k| \over \sqrt \lambda_k}.
    \end{align}
    It remains to bound from above the operator norm of $(\theta n\sigma^2)^{-1}\bX\bX^T -  \bI_r$. It is easy to see that 
    \[
        \E\brac{
        {1\over \theta n\sigma^2}\bX\bX^T
        } = \bI_r.
    \]
    Since $\bX_{it}$ for $1\le i\le r$ and $1\le t\le n$ are i.i.d. sub-Gaussian random variables with sub-Gaussian constant no greater than $1$, classical deviation inequality of the operator norm of the sample covariance matrices for i.i.d. sub-Gaussian entries \cite[Remark 5.40]{vershynin_2012} gives 
    \begin{equation}\label{bd_cov_X}
        \left\|
            {1\over n\sigma^2}\bX\bX^T -  \theta\bI_r
        \right\|_{\op} \le c\paren{
            \sqrt{r\over n} + {r\over n}
        }
    \end{equation} 
    with probability $1 - 2e^{-c'r}$ for some constants $c,c'>0$. 
    Using 
    \[
        {1\over \theta}\sqrt{r \over n}\le c''
    \]
    for some small constant $c''>0$ 
    concludes 
    \begin{equation}\label{bd_QLbdQDA}
            \norm{
            \sqrt{\theta n\sigma^2}  \mb Q^T \bLbd^{-1/2} \mb Q \bD_A  -\bI_r
        }{op}  \le 
           c''' {1\over \theta}\sqrt{r\over n} 
    \end{equation}
    with probability $1-2e^{-c'r}$. This completes the proof. 
    \end{proof}

    \subsection{Proof of Theorem \ref{thm:obj_gene_sample}}\label{proof:thm:obj_gene_sample}
        In this section we provide the proof of Theorem \ref{thm:obj_gene_sample}. Our proof is similar to Section \ref{proof:thm:obj_sample}. Recall that $\bar A = \bU_A\bV_A^T$. We define a new partition of $\bb S^{p-1}$ as
    \begin{align}\label{def:obj_prime_gene}
    R_0^{''} &\doteq R_0^{''}(c_\star) =\Brac{
        \mb q \in \bb S^{p-1}:\norm{\oA^{T}\mb q}\infty^2 \leq c_\star
    },\\\nonumber
    R_{1}^{''} &\doteq R_1(C_{\star})=\Brac{\mb q \in \bb S^{p-1}:\norm{\oA^{T}\mb q}\infty^{2}\geq  C_{\star} },\\\nonumber
    R_{2}^{''} & = \bb S^{p-1} \setminus \paren{R_0^{''} \cup R_1^{''}}.
    \end{align}
    Here $c_\star$ and $C_{\star}$ are positive constants satisfying $0\leq c_\star \leq C_{\star}<1$. Further define
    \begin{align}\label{def:gene_f_A_bar}
        \bar{f}_{g}\left(\mb q\right):=\mathbb{E}\left[-\frac{1}{12\theta\sigma^4n}\norm{\mb q^{T}\oA\mb X}4^4\right]=-{1\over 4}\brac{(1-\theta)\norm{\oA^T \mb q}4^4 +\theta \norm{\oA^T \mb q}2^4}.
    \end{align}
    The equality uses Lemma \ref{lem:obj}.
    Let $\delta_1$ and $\delta_2$ be some positive sequences and define the random event 
    \begin{align}\label{def_event_E_gene}
        \mathcal E = \Brac{
        \sup_{\bq \in \Sp^{p-1}}\norm{\grad \bar{f}_{g}\paren{\mb q}-\grad F_{g}\paren{\mb q}}2\leq \delta_{1},~ \sup_{\bq \in \Sp^{p-1}}\|\Hess \bar{f}_{g}\paren{\mb q}-\Hess F_{g}\paren{\mb q}\|_{\op} \leq \delta_{2}}.
    \end{align}
    Here $\grad \bar{f}_{g}(\bq)$ and $\grad F_{g}(\bq)$ are the gradients of (\ref{def:gene_f_A_bar}) and (\ref{obj_gene}), respectively, at any point $\bq\in \Sp^{p-1}$. Similarly, $\Hess \bar{f}_{g}(\bq)$ and $\Hess F_{g}(\bq)$ are the corresponding Hessian matrices.
    
    We observe that Lemmas \ref{lem:sample_R2_nega_short} and \ref{lem:R1_three_case_short} continue to hold by replacing $F(\bq)$, $f(\bq)$ and $\bA$ by $F_g(\bq)$, $\bar f_g(\bq)$ and $\bar \bA$, respectively, and by using $R_1''$ and $R_2''$ in lieu of $R_1'$ and $R_2'$. The proof is then completed by verifying that  $\delta_2 \le c_\star / 25$, $\delta_2\le 10^{-3}$ and $\delta_1 \le 5\times 10^{-5}$. These are guaranteed by invoking Lemmas \ref{lem:gene_gradient_concentration}
   and \ref{lem:gene_hessian_concentration} and using condition (\ref{cond_n_general}).

    \subsection{Proof of Lemma \ref{lem:gene_init}}\label{app_proof_lemma_init}
    
    \begin{proof}
    It suffices to prove 
    \[
        \left\|\bar \bA^T \bq^{(0)}\right \|_\infty^2 ~ \ge ~  c_\star 
    \]
    for some $c_\star$ such that 
    (\ref{cond_n_general}) holds. To this end, we work on the event where Proposition \ref{prop_precond} holds such that 
    \[
        \bar\bY = \bar\bA (\bI_r + \mb \Delta)\bar\bX.
    \]
    It then follows that 
     \begin{align}
         \norm{\oA^{T}\mb q^{(0)}}\infty^2\geq\frac{1}{r}\norm{\oA^{T}\mb q^{(0)}}2^2=
         \frac{1}{r}\norm{\oA^{T}\bar\bY \mb 1_n \over \|\bar\bY \mb 1_n\|_2}2^2 = 
         \frac{1}{r}\norm{(\bI_r+\mb\Delta)\bar\bX \mb 1_n \over \|\bar\bA (\bI_r+\mb\Delta)\bar\bX\mb 1_n\|_2}2^2 = {1\over r}
        %  \frac{1}{r}\norm{\frac{\oA^{T}\left(\mb Y\mb Y^{T}\right)^{-\frac{1}{2}}\mb A\bar{\mb x}}{\norm{\left(\mb Y\mb Y^{T}\right)^{-\frac{1}{2}}\mb A\bar{\mb x}}2}}2^2=\frac{1}{r}\norm{\frac{\oA^{T}\left(\frac{1}{\theta n\sigma^{2}}\mb Y\mb Y^{T}\right)^{-\frac{1}{2}}\mb A\bar{\mb x}}{\norm{\left(\frac{1}{\theta n\sigma^{2}}\mb Y\mb Y^{T}\right)^{-\frac{1}{2}}\mb A\bar{\mb x}}2}}2^2
     \end{align}
     by using $\bar\bA^T\bar\bA= \bI_r$, provided that $\|\bar\bY\mb 1_n\|_2 \ne 0$ which holds only one a set with zero measure. 
     The proof is completed by invoking condition (\ref{cond_n_init}) to ensure (\ref{cond_n_general}) holds for $c_\star = 1/(2r)$.
    %  Here $\bar{\mb x}=\frac{1}{n}\sum_{i=1}^{i=n}\left(\mb x_i\right)$.
    %  On the other hand, leveraging lemma \ref{lem:bound_Y_pre_A_pre}, we could obtain:
    %  \begin{align}
    %      \|\oA^{T}\left(\frac{1}{\theta n\sigma^{2}}\mb Y\mb Y^{T}\right)^{-\frac{1}{2}}\mb A\bar{\mb x}-\bar{\mb x}\|_{2}\leq\|\left(\frac{1}{\theta n\sigma^{2}}\mb Y\mb Y^{T}\right)^{-\frac{1}{2}}\mb A\bar{\mb x}-\oA\bar{\mb x}\|_{2}\leq \delta\norm{\bar{\mb x}}2
    %  \end{align}
    %  with probability at least $1-2e^{-c_1 n\delta^2}$. Here $\|\oA\|_{\op}=1$ is used in derivation. Define a random event $\mathcal{E}$ such that
    %  \begin{align}
    %      \mathcal{E}':=\Brac{\norm{\bar{{\mb x}}}2\neq0}
    %  \end{align}On this event, we could then conclude that 
    %  \begin{align}
    %      \norm{\oA^{T}\mb q^{(0)}'}\infty^2\geq\frac{1}{r}\frac{\left(1-\delta\right)^{2}\norm{\bar{\mb x}}2^2}{\left(1+\delta\right)^{2}\norm{\bar{\mb x}}2^2}=\frac{1}{r}\frac{(1-\delta)^2}{(1+\delta)^2}\geq\frac{25c_{1}}{\log^{4}(n)}=c_{\star}
    %  \end{align}
    %   Till here this proof completes.
    \end{proof}

     \subsection{Proof of Lemma \ref{lem:gene_iter}}\label{app_proof_lemma_iter}
     \begin{proof}
        Recall that  $F_{g}(\bq)$ and $\bar{f}_{g}(\bq)$ are defined in (\ref{obj_gene}) and (\ref{def:gene_f_A_bar}), respectively. 
         We work on the event  
        \begin{align}
            \mathcal{E}_g:=\Brac{\sup_{\mb q \in\bb S^{p-1}}\left|F_{g}\paren{\mb q}-\bar{f}_{g}\paren{\mb q}\right|\lesssim 
            \delta_n
            },
        \end{align} 
        with 
        \[
            \delta_n = \left(\sqrt{r\theta} + \sqrt{\log n}\right)\sqrt{r\over \theta^2\sqrt{\theta} n} + \left(\theta r^2 + {\log^2 n\over \theta}\right) {r \log n \over n}.
        \]  
        According to Lemma \ref{lem:function_value_concentration_gene}, $\cE_g$ holds with probability at least $1-cn^{-c'}-2e^{-c''r}$. We aim to prove 
        \[
            \left\|\bar\bA^T \bq^{(k)}\right\|_\infty^2  ~ \ge~  c_\star = {1\over 2r},\qquad \forall k\ge 1.
        \]
        
        Pick any $k\ge 1$. On the event $\cE_g$, 
        we have 
        \[
             F_{g}\paren{\bq^{(k)}} \ge \bar{f}_{g}\paren{\bq^{(k)}} - \delta_n.
        \]
        For any $\bq\in \Sp^{p-1}$, we write $\bar{\mb \zeta} = \oA^{T} \bq$. Since 
        \begin{align}\label{ieq:init_lowb_zeta_inft2}\nonumber
            \bar{f_{g}}\paren{\mb q}&=-\frac{1}{4}\brac{\paren{1-\theta}\norm{\bar{\mb \zeta}}4^4+\theta\norm{\bar{\mb \zeta}}2^4}\\\nonumber &\geq-\frac{1}{4}\brac{\paren{1-\theta}\norm{\bar{\mb \zeta}}\infty^2+\theta r\norm{\bar{\mb \zeta}}\infty^2}\\
            &=-\frac{1}{4}\paren{1-\theta+\theta r}\norm{\bar{\mb \zeta}}\infty^2,
        \end{align}
        where we used $\norm{\bar{\mb \zeta}}4^4\leq\norm{\bar{\mb \zeta}}\infty^2\norm{\bar{\mb \zeta}}2^2\leq\norm{\bar{\mb \zeta}}\infty^2$, $\norm{\bar{\mb \zeta}}2^2\leq r\norm{\bar{\mb \zeta}}\infty^2$ and $\norm{\bar{\mb \zeta}}2^2\leq 1$. It then follows that 
        \begin{align}
            F_g\paren{\bq^{(k)}}  \geq-\frac{1}{4}\paren{1-\theta+\theta r}\norm{\oA^T\bq^{(k)}}\infty^2-\delta_n
        \end{align}
        on the event $\cE_g$. On the other hand, any gradient descent algorithm ensures 
        $$
            F_g(\bq^{(k)}) \le F_g(\bq^{(0)}). 
        $$
        We thus have 
        \[
            \frac{1}{4}\paren{1-\theta+\theta r}\norm{\oA^T\bq^{(k)}}\infty^2 \ge -F_g\left(\bq^{(0)}\right) - \delta_n \ge -\bar{f_{g}}\left(\bq^{(0)}\right) - 2\delta_n
        \]
        by using $\cE_g$ again in the last inequality. 
        To bound from below $-\bar{f_{g}}(\bq^{(0)})$, recalling the definition of $\bar{f}_{g}$ from (\ref{def:gene_f_A_bar}), we have 
        \begin{align}
            -\bar{f_{g}}(\bq^{(0)}) &= {1\over 4}\brac{ 
            (1-\theta)\norm{\oA^T \mb q^{(0)}}4^4 + \theta \norm{\oA^T \mb q^{(0)}}2^4
            }
        \end{align}
       Since, on the event where Proposition \ref{prop_precond} holds such that $\bar\bY = \bar\bA (\bI_r + \mb \Delta)\bar\bX$, 
       \begin{align}
           \norm{\oA^T \mb q^{(0)}}4^4 & = \norm{\oA^T \bar\bY\mb 1_n \over \|\bar\bY\mb 1_n\|_2}4^4 = { \norm{(\bI_r + \mb \Delta)\bar\bX \mb 1_n}4^4  \over \| (\bI_r + \mb \Delta)\bar\bX\mb 1_n\|_2^4}
       \end{align}
       and, similarly, 
       \[
         \norm{\oA^T \mb q^{(0)}}2^4  = 1,
       \]
       we conclude 
       \[
             -\bar{f_{g}}(\bq^{(0)})  \ge {1\over 4}\left[
             (1-\theta)  { \norm{(\bI_r + \mb \Delta)\bar\bX \mb 1_n}4^4  \over \| (\bI_r + \mb \Delta)\bar\bX\mb 1_n\|_2^4} + \theta \right] \ge {1\over 4}\left[
             {1-\theta \over r} + \theta \right]
       \]
       with probability $1-2e^{-cr}$. Here we used the basic inequality $\|\bv\|_2^4 \le r\|\bv\|_4^4$ for any $\bv\in \R^r$. With the same probability, we further have 
       \[
         \norm{\oA^T\bq^{(k)}}\infty^2 \ge {1\over r} -  {8\delta_n\over 1-\theta+\theta r} \ge {1\over 2r}
       \]
       provided that 
       \[
            {8\delta_n\over 1-\theta+\theta r}  \le {1\over 2r}.
       \]
       This is guaranteed by condition (\ref{cond_n_iter}). The proof is complete.
 \end{proof}

  \subsection{Concentration inequalities when $A$ is semi-orthonormal}\label{app_concentration_orth}
  
  In this section, we provide deviation inequalities for different quantities between the  population-level problem $f(\bq)$ in (\ref{obj_pop}) and its sample counterpart $F(\bq)$ in (\ref{obj}), including the objective function, the Riemannian gradient and the Riemannian Hessian matrix. Our analysis adapts some technical results in \cite{zhang2018structured} and  \cite{qu2019analysis} to our setting.

  %In this part we shall show the concentration for gradient and hessian for function $F_{\textrm{UDF}}\paren{\mb q}$ defined in definition \ref{def:F_UDL and f_UDL} by demonstrating these two points:\\
  %\begin{itemize}
  %  \item \textbf{Close critical points}: Riemannian gradient and hessian for function $F_{\textrm{UDF}}\paren{\mb q}$ and $f_{\textrm{UDF}}\paren{\mb q}$ are close enough. Which means for any small $\kappa$ there exist a finite large $n$ such that
  %  \begin{align}
  %      \sup \norm{\grad F_{\textrm{UDF}}\paren{\mb q}-\grad f_{\textrm{UDF}}\paren{\mb q}}2 \leq \kappa \textrm{ , } \sup \norm{\Hess F_{\textrm{UDF}}\paren{\mb q}-\Hess f_{\textrm{UDF}}\paren{\mb q}}2 \leq \kappa \nonumber
  %  \end{align}
  %  \item \textbf{Similar curvature}: The spectrum of hessian matrix of function $F_{\textrm{UDF}}\paren{\mb q}$ is close to $\Hess F_{\textrm{UDF}}\paren{\mb q}$ so that we could say when there exist a negative eigenvalue of $\Hess F_{\textrm{UDF}}\paren{\mb q}$, so does $\Hess f_{\textrm{UDF}}\paren{\mb q}$ 
  %\end{itemize}
    
    %and considering the similarity of our function $F_{\textrm{UDF}}$ with $\phi_{\textrm{DL}}$ defined in paper \cite{qu2019analysis} we directly leverage some lemmas to get the upper bound of some parameters in lemma \ref{lem:gradient_concentration} and \ref{lem:hessian_concentration}.

  \subsubsection{Deviation inequalities of the objective value}\label{app_deviation_obj}
  Recall that 
    \begin{align}\label{def_F_f_proof}
        &F(\bq) = -\frac1{12 \theta \sigma^4 n} \norm{\mb q^T\mb A\mb X}4^4,\\\nonumber
        &f(\bq) = -{1\over 4}\brac{(1-\theta)\norm{\mb A^T \mb q}4^4 + \theta \norm{\mb A^T \mb q}2^4}.
    \end{align}
  Define  
  \begin{equation}\label{def_Mn}
        M_n =  C (n+r) \left(\theta r^2 + {\log^2 n \over \theta}\right)
  \end{equation}
  for some constant $C>0$.
   
  \begin{lemma}\label{lem:function_value_concentration}
 Under Assumptions \ref{ass_X} and \ref{ass_A_orth}, with probability greater than $1-cn^{-c'}$ for some constants $c,c'>0$, one has
     \begin{align}\nonumber
        \sup_{\mb q\in \Sp^{p-1}}\left| F\paren{\mb q}-f\paren{\mb q} \right| ~ \lesssim ~
       \sqrt{r \log (M_n) \over \theta n} + {M_n\over n} {r \log (M_n) \over n}.
     \end{align}
% %   Then for any $\delta\in\paren{0,
%  for any $\delta\in \paren{0, c~ \theta^{-1}\omega_n^{-1}}$, if 
%  \begin{align}
%      n\geq C{r\over\theta \delta^2}\log\paren{\omega_n \over  \delta}, 
%  \end{align} 
%  then, with probability greater than $1-\paren{rn}^{-1}-e^{-c'r\log\paren{\delta^{-1}\omega_n}}-2n^{-c''}$,
%  we have 
%  \begin{align}
%     \sup_{\mb q\in \Sp^{p-1}}\norm{ F\paren{\mb q}-f\paren{\mb q}}2\leq \delta,
%  \end{align}
%  with $R_0$ defined in (\ref{def_R0}).
%  Here $c$, $c'$, $c''$ and $C$ are some positive constants.
 \end{lemma}
 \begin{proof}
    Pick any $\bq \in \Sp^{p-1}$. Note that the result holds trivially if $\bA^T \bq = 0$. For $\bA^T \bq \ne 0$, we define 
    \[
        \wt \bq = {\bA^T\bq \over \|\bA^T\bq\|_2},\qquad \textrm{with }\quad \wt\bq \in \Sp^{r-1}.
    \]
    Note that 
    \begin{align}\label{def:definition_of_f_value}
         {F(\mb q) \over\|\bA^T\bq\|_2^4}
         = {1\over n}\sum_{k=1}^n F_{\wt \bq}\paren{\mb x_k},\qquad \textrm{with}\qquad 
         F_{\wt \bq}\paren{\mb x_k}=-\frac{1}{12\theta\sigma^4}\paren{\wt \bq^{T} \mb x_k}^4.
     \end{align}
     The proof of Lemma \ref{lem:obj} shows that 
     \begin{align} 
        \E[F_{\wt  \bq}\paren{\mb x_k}] = {f(\bq) \over \|\bA^T\bq\|_2^4}
        &= -{1\over 4}\brac{(1-\theta)\norm{\wt \bq}4^4 + \theta \norm{\wt \bq}2^4}\doteq g(\wt \bq),\quad \forall 1\le k\le n.
     \end{align}
     We thus aim to invoke Lemma \ref{lem:vector concentration} with $n_1 = r$, $d_1 = 1$, $n = r$ and $p = n$ to bound from above 
     \[
        \sup_{\wt \bq\in \Sp^{r-1}}\left|{1\over n}\sum_{k=1}^n F_{\wt \bq}\paren{\mb x_k} -g\paren{\wt \bq}\right|.
     \] 
     Consequently, the result follows by noting that 
     \[
        \sup_{\mb q\in \Sp^{p-1}}\left| F\paren{\mb q}-f\paren{\mb q}\right| = \sup_{\bq\in \Sp^{p-1}\setminus R_0}\norm{\bA^T\bq}2^4\cdot \left|{1\over n}\sum_{k=1}^n F_{\wt \bq}\paren{\mb x_k} -g\paren{\wt \bq}\right|
     \]
     and using $\|\bA^T\bq\|_2\le 1$ which holds uniformly over $\bq\in\Sp^{p-1}$. 
     
     Since the entries of $\mb x_k$ are i.i.d. Bernoulli-Gaussian random variables with parameter $(\theta, \sigma^2)$, each $\mb x_{ki}$, for $1\le i\le r$, is sub-Gaussian with the sub-Gaussian parameter equal to $\sigma^2$. It thus suffices to verify Conditions 1 -- 2 in Lemma \ref{lem:vector concentration}. For simplicity, we write $\mb x = \mb x_k$. \\ 
    
    \noindent{\bf Verification of Condition 1:} 
    Since 
    $
        \E[F_{\wt \bq}(\mb x)] = g(\wt \bq),
    $
    we observe
    \begin{align}
             \left|\E[F_{\wt \bq}(\mb x)]\right| =\frac{1}{4}\left[\left(1-\theta\right)\norm{\wt \bq}4^4+\theta\norm{\wt \bq}2^4\right]\leq\frac{1}{4}\norm{\wt \bq}2^4 = \frac{1}{4}
          \end{align}
    where we used $\norm{\wt \bq}4^4\leq\norm{\wt \bq}2^4$. Thus $B_f = 1/4$. For any $\bq_1 \ne  \bq_2 \in \Sp^{r-1}$, we have
  \begin{align}
      &\left|\E\brac{F_{\mb q_1}\paren{\mb x}}-\E\brac{F_{\mb q_2}\paren{\mb x}}\right|\nonumber\\
      &\le \frac{1-\theta}{4}\left|\norm{\mb q_1}4^4-\norm{\mb q_2}4^4\right|+\frac{\theta}{4}\left|\norm{\mb q_1}2^4-\norm{\mb q_2}2^4\right|\nonumber\\
      &=\frac{1-\theta}{4}\Bigl|\norm{\mb q_1}4-\norm{\mb q_2}4\Bigr|\left(\norm{\mb q_1}4+\norm{\mb q_2}4\right)\left(\norm{\mb q_1}4^2+\norm{\mb q_2}4^2\right)\nonumber\\
      &\quad +\frac{\theta}{4}\left|\norm{\mb q_1}2-\norm{\mb q_2}2\right|\left(\norm{\mb q_1}2+\norm{\mb q_2}2\right)\left(\norm{\mb q_1}2^2+\norm{\mb q_2}2^2\right)\nonumber\\
      &\leq\left(1-\theta\right)\norm{\mb q_1 - \mb q_2}4 + \theta\norm{\mb q_1-\mb q_2}2 & \textrm{by }\|\bq_1\|_4\le \|\bq_1\|_2 = 1\nonumber\\
      &\leq \norm{\mb q_1-\mb q_2}2.
  \end{align}
   This gives $L_f=1$.\\

    \noindent{\bf Verification of Condition 2:} We define 
    \[
        \mb x = \bar {\mb x} + \wt {\mb x}
    \]
    as (\ref{def_Z_trunc}) with $B = 2\sigma\sqrt{\log(nr)}$. For the similar fashion, we define $\mb x_k = \bar {\mb x}_k + \wt {\mb x}_k$ for $1\le k\le n$. We verify Condition 2 on the event 
    \begin{equation}\label{def_event_E_prime}
        \cE' := \bigcap_{k=1}^n\Bigl\{
            \norm{\bx_k}2 \lesssim \sigma (\sqrt{r\theta} + \sqrt{\log n})
            \Bigr\}.
    \end{equation}
    Lemma \ref{lemma:bound_X_inf_column} ensures that $\P(\cE') \ge 1-2n^{-c}$ for some $c>0$. Note that on the event $\cE'$,
   \begin{equation}\label{bd_bar_x}
       \|\bar \bx\|_2 ~ \le ~ \sigma (\sqrt{r\theta} + \sqrt{\log n}).
   \end{equation}
    Pick any $\bq \in \Sp^{r-1}$, we have 
        \begin{align}
            \norm{F_{\bq}\paren{\bar{\mb x}}}2=\norm{\frac{1}{12\theta\sigma^4}\paren{\mb q^{T}\bar \bx}^4}2&\leq\frac{\norm{\bq}2^4\norm{\bar{\mb x}}2^4}{12\theta\sigma^4}
            \le  C\left(\theta r^2 + {\log^2 n \over \theta}\right)
        \end{align}
        for some constant $C>0$.
        Thus,
        \[
           R_1 =  C\left(\theta r^2 + {\log^2 n \over \theta}\right).
        \]
        On the other hand, we have 
        \begin{align} 
            \sup_{\bq\in \Sp^{r-1}}\mathbb{E}\brac{\norm{F_{\bq}\paren{\bar{\mb x}}}2^2}\leq \sup_{\bq\in \Sp^{r-1}}\mathbb{E}\brac{\norm{F_{\bq}\paren{\mb x}}2^2}\leq  c\theta^{-1} =R_2 
        \end{align}
        for some constant $c>0$. Here Lemma \ref{lem:upperbound_for_exp_f_value} is used in the last inequality.
        
        % \xb{Lemma C.5 is for $F_{\bq}(\bx)$ instead of $F_{\bq}(\bar\bx)$.}
        
        On the other hand, pick $\bq_1 \ne \bq_2 \in \Sp^{r-1}$. We obtain
        \begin{align}
            &\norm{F_{\mb q_1}\paren{\bar{\mb x}}-F_{\mb q_2}\paren{\bar{\mb x}}}2\nonumber\\
            &=\frac{1}{12\theta\sigma^4}\norm{\paren{\mb q_1^{T}\bar{\mb x}}^4-\paren{\mb q_2^{T}\bar{\mb x}}^4}2\nonumber\\
            &=\frac{1}{12\theta\sigma^4}\left|\paren{\mb q_1^{T}\bar{\mb x}}-\paren{\mb q_2^{T}\bar{\mb x}}\right|\cdot \left|\paren{\mb q_1^{T}\bar{\mb x}}+\paren{\mb q_2^{T}\bar{\mb x}}\right|\left(\paren{\mb q_1^{T}\bar{\mb x}}^2+\paren{\mb q_2^{T}\bar{\mb x}}^2\right)\nonumber\\
            &\leq\frac{4}{12\theta\sigma^4}\norm{\bar{\mb x}}2^4\norm{\mb q_1-\mb q_2}2.
        \end{align}
        Combine with (\ref{bd_bar_x}) to conclude
        \begin{align}
            \norm{F_{\mb q_1}\paren{\bar{\mb x}}-F_{\mb q_2}\paren{\bar{\mb x}}}2\leq R_1\norm{\mb q_1-\mb q_2}2,
        \end{align}
        hence $\bar L_f = R_1$.\\
        
     Finally, invoke Lemma \ref{lem:vector concentration} with 
     $M = c' R_1 n = M_n$
     to obtain the  desired result and complete the proof.
 \end{proof}

\subsubsection{Deviation inequalities of the Riemannian gradient}

In this part, we derive the deviation inequalities between the Riemannian gradient of $F(\bq)$ and that of function $f\paren{\mb q}$. From (\ref{eq_grad}), note that, for any $\bq\in \Sp^{p-1}$, 
\begin{align}\label{grad_F}
    &\grad F\paren{\mb q} \doteq \grad_{\norm{\mb q}2=1}F\paren{\mb q}=-\frac{1}{3\theta\sigma^4 n}\mb P_{\mb q\perp}\sum_{k=1}^{n}\paren{\mb q^{T}\mb A\mb x_{k}}^3\mb A\mb x_{k},\\
    &\grad f\paren{\mb q}\doteq \grad_{\norm{\mb q}2=1}f\paren{\mb q}= -P_{\mb q^\perp}\brac{
    (1-\theta)\sum_{j=1}^r \mb a_j (\mb q^T \mb a_j)^3 + \theta \|\mb q^T \mb A\|_2^2 \mb A\mb A^T \mb q}.\label{grad_f}
\end{align}
Direct calculation shows that 
 \begin{align}
     \E\brac{\grad F\paren{\mb q}}=\grad f(\mb q).
 \end{align}
The following lemma provides deviation inequalities between $F(\bq)$ and $f(\bq)$ by invoking Lemma \ref{lem:vector concentration}, stated in Appendix \ref{app_auxiliary}. Recall that $M_n$ is defined in (\ref{def_Mn}).

 \begin{lemma}\label{lem:gradient_concentration}
 Under Assumptions \ref{ass_X} and \ref{ass_A_orth}, with probability greater than $1-cn^{-c'}$ for some constants $c,c'>0$, one has
     \begin{align}\nonumber
       \sup_{\mb q\in \Sp^{p-1}}\norm{\grad F\paren{\mb q}-\grad f\paren{\mb q}}2 \lesssim ~
       \sqrt{r^2 \log (M_n) \over \theta n} + {M_n\over n} {r \log (M_n) \over n}.
     \end{align}
%  for any $\delta\in \paren{0,c r/(\theta \omega_n)}$, if 
%  \begin{align}
%      n\geq C{ r^2\over \theta\delta^2}\log\paren{\omega_n \over  \delta},
%  \end{align} 
%  then, with probability greater than $1-\paren{rn}^{-1}-e^{-c'r\log\paren{\omega_n / \delta}}-2n^{-c''}$,
%  we have 
%  \begin{align}
%     \sup_{\mb q\in \Sp^{p-1}}\norm{\grad F\paren{\mb q}-\grad f\paren{\mb q}}2\leq \delta,
%  \end{align}
%  with $R_0$ defined in (\ref{def_R0}).
%  Here $c$, $c'$, $c''$ and $C$ are some positive constants.
 \end{lemma}
 
 \begin{proof}
 Pick any $\bq \in \Sp^{p-1}$. As the result trivially holds for $\bA^T\bq = 0$, we only focus on when $\bA^T \bq \ne 0$. Define 
    \[
        \wt \bq = {\bA^T\bq \over \|\bA^T\bq\|_2},\qquad \textrm{with }\quad \wt\bq \in \Sp^{r-1}.
    \]
    By writing $\mb \zeta = \bA^T\bq$, observe that 
    \begin{align}
        &\norm{\grad F(\mb q)-\grad f(\mb q)}2\nonumber\\
        &=\norm{\mb  P_{\mb q\perp}\left[\frac{1}{3\theta\sigma^4 n}\sum_{k=1}^{n}\paren{\mb q^{T}\mb A\mb x_{k}}^3\mb A\mb x_{k}-\brac{\paren{1-\theta}\mb A\paren{\mb \zeta}^{\circ3}+\theta\norm{\mb\zeta}2^2\mb A\mb\zeta}\right]}2\nonumber\\
        &\leq \norm{\frac{1}{3\theta\sigma^4 n}\sum_{k=1}^{n}\paren{\mb q^{T}\mb A\mb x_{k}}^3\mb A\mb x_{k}-\brac{\paren{1-\theta}\mb A\paren{\mb \zeta}^{\circ3}+\theta\norm{\mb\zeta}2^2\mb A\mb\zeta}}2\nonumber\\
        &\leq\norm{\frac{1}{3\theta\sigma^4 n}\sum_{k=1}^{n}\paren{\mb q^{T}\bA \mb x_{k}}^3\mb x_{k}-\brac{\paren{1-\theta}\paren{\mb \zeta}^{\circ3}+\theta\norm{\mb\zeta}2^2\mb\zeta}}2\nonumber\\
        &\leq\norm{\mb A^{T}\mb q}2^3 \norm{\frac{1}{3\theta\sigma^4 n}\sum_{k=1}^{n}\paren{\wt \bq^{T}\mb x_{k}}^3\mb x_{k}-\brac{\paren{1-\theta}\paren{\wt \bq}^{\circ3}+\theta\norm{\wt \bq}2^2\wt \bq}}2\nonumber\\
        &\leq\norm{\frac{1}{3\theta\sigma^4 n}\sum_{k=1}^{n}\paren{\wt \bq^{T}\mb x_{k}}^3\mb x_{k}-\brac{\paren{1-\theta}\paren{\wt \bq}^{\circ3}+\theta\norm{\wt \bq}2^2\wt \bq}}2,
    \end{align}
     where we have used $\|\mb A\|_{\op}\le 1$ and  $\norm{\mb A^{T}\mb q}2^3\leq 1$ in the last two steps.  Define 
     \begin{align}\label{def_F_tilde_q}
         F_{\wt \bq}\paren{\mb x}:=\frac{1}{3\theta\sigma^4}\paren{\wt \bq^{T}\mb x}^3\mb x.
     \end{align}
     It is easy to verify that
     \begin{align}\label{def_f_tilde_q}
         \bb E\left[F_{\wt \bq}\paren{\mb x}\right]=\brac{\paren{1-\theta}\paren{\wt \bq}^{\circ3}+\theta\norm{\wt \bq}2^2\wt \bq} \doteq g\paren{\wt \bq}.
     \end{align}
     We thus aim to invoke Lemma \ref{lem:vector concentration}  with $n_1 = r$, $d_1 = r$, $n = r$ and $p = n$, to bound from above 
     \[
        \sup_{\wt \bq\in \Sp^{r-1}}\norm{{1\over n}\sum_{k=1}^n F_{\wt \bq}\paren{\mb x_k} -g\paren{\wt \bq}}2.
     \]
     Recall that $\bx_{ij}$ is sub-Gaussian with parameter $\sigma^2$, for $1\le i\le n$ and $1\le j\le r$. \\

     \noindent{\bf Verification of Condition 1:} By $\|\wt\bq\|_2=1$, notice that
         \begin{align}
             \norm{g(\wt \bq)}2&=\norm{\paren{1-\theta}\paren{\wt \bq}^{\circ3}+\theta \wt \bq}2\leq \left(1-\theta\right)\norm{\wt \bq^{\circ 3}}2+\theta\norm{\wt \bq}2 \le \|\wt \bq\|_2 =1.
         \end{align}
         Further note that, for any $\wt \bq_1, \wt\bq_2 \in \Sp^{r-1}$,
         \begin{align}
            \norm{g\paren{\wt \bq_1}-g\paren{\wt \bq_2}}2
            &\leq\left(1-\theta\right)\norm{\paren{\wt \bq_1}^{\circ3}-\paren{\wt \bq_2}^{\circ3}}2+\theta\norm{\wt \bq_1-\wt \bq_2}2\nonumber\\
            &\leq3(1-\theta)\norm{\wt \bq_1-\wt \bq_2}2+\theta\norm{ \wt \bq_1- \wt \bq_2}2\nonumber\\
            &\leq3\norm{\wt \bq_1-\wt \bq_2}2.
         \end{align}
         Here $\|\paren{\wt \bq_1}^{\circ3}-\paren{\wt \bq_2}^{\circ3}\|_2\leq3\norm{\wt \bq_1-\wt \bq_2}2$ is used in the second step. As a result, $B_f=1$ and $L_{f}=6$.\\

        \noindent{\bf Verification of Condition 2:} We still work on the event $\cE'$ in (\ref{def_event_E_prime}) such that (\ref{bd_bar_x}) holds for all $1\le k\le n$. In this case,
        \begin{align}
            \norm{F_{\wt \bq}\paren{\bar{\mb x}_i}}2=\norm{\frac{1}{3\theta\sigma^4}\paren{\wt \bq^{T}\bar{\mb x}_i}^3\bar{\mb x}_i}2\leq\frac{\norm{\bar{\mb x}_i}2^4}{3\theta \sigma^4} \le C \left(\theta r^2 + {\log^2 n \over \theta}\right).
            % \le B^2\norm{\bar{\mb x}_i}0\leq4B^2\theta r\log\paren{n}.
        \end{align}
        Hence 
        \begin{align}
            R_1 = C \left(\theta r^2 + {\log^2 n \over \theta}\right).
        \end{align}
        Also from Lemma \ref{lem:upperbound_for_exp_f_grad} with some straightforward modifications,  we know 
        \begin{align}
            \sup_{\wt\bq\in\Sp^{r-1}}\mathbb{E}\brac{\norm{F_{\wt \bq}\paren{\bar{\mb x}_i}}2^2}\leq\sup_{\wt\bq\in\Sp^{r-1}}\mathbb{E}\brac{\norm{F_{\wt \bq}\paren{\mb x_i}}2^2}\leq c\theta^{-1} r,
        \end{align}
        for some constant $c>0$. We thus have $R_2 = c \theta^{-1} r$. 
        
        To calculate $\bar L_f$, we have
        \begin{align}
            \norm{F_{\wt \bq_1}\paren{\bar{\mb x}_i}-F_{\wt \bq_2}\paren{\bar{\mb x}_i}}2&=\frac{1}{3\theta\sigma^4}\norm{\paren{\wt \bq_1^{T}\bar{\mb x}_i}^3\bar{\mb x}_i-\paren{\wt \bq_2^{T}\bar{\mb x}_i}^3\bar{\mb x}_i}2\nonumber\\
            &\leq\frac{1}{3\theta\sigma^4}\norm{\left(\wt \bq_1\right)^{\circ 3}-\left(\wt \bq_2\right)^{\circ 3}}2\norm{\mb \bar{\mb x}}2^4\nonumber\\
            &\leq\frac{1}{\theta\sigma^4}\norm{\wt \bq_1-\wt \bq_2}2\norm{\mb \bar{\mb x}}2^4.
        \end{align}
        Here $\|\wt \bq_1^{\circ 3}-\wt \bq_2^{\circ 3}\|_2\leq3\norm{\wt \bq_1-\wt \bq_2}2$ is used in  the last step. We thus conclude
        \begin{align}
            \norm{F_{\mb q_1}\paren{\bar{\mb x}}-F_{\mb q_2}\paren{\bar{\mb x}}}2\leq R_1 \norm{\mb q_1-\mb q_2}2
        \end{align}
        hence $\bar L_f =R_1$. 
        
        Finally invoke Lemma \ref{lem:vector concentration} with $M = C'(n+r)R_1$ to complete the proof. 
 \end{proof}

\subsubsection{Deviation inequalities of the Riemannian Hessian}
 In this part we will show that the Hessian of $F(\bq)$ concentrates around that of $f(\bq)$. Notice that, for any $\bq\in\Sp^{p-1}$ with $\mb \zeta = \bA^T\bq$, 
 \begin{align}
     &\Hess F\left(\mb q\right)=-\frac{1}{3\theta\sigma^4n}\sum_{k=1}^nP_{\mb q\perp }\brac{3\paren{\mb \zeta^T\mb x_{k}}^2\mb A\mb x_k\paren{\mb A\mb x_k}^{T}-\paren{\mb \zeta^T\mb x_k}^4\mb I_p}P_{\mb q\perp },\\ \nonumber
     & \Hess f\left(\mb q\right)=-\left\{\paren{1-\theta}P_{\mb q\perp }\brac{3\mb A\diag(\mb \zeta^{\circ2})\mb A^{T}-\norm{\mb \zeta}4^4\mb I}P_{\mb q\perp }\right.\\
     &\hspace{2.5cm}\left.+\theta P_{\mb q\perp}\brac{\norm{\mb \zeta}2^2\mb A\mb A^{T}+2\mb A\mb \zeta\mb \zeta^T\mb A^{T}-\norm{\mb \zeta}2^4\mb I }P_{\mb q\perp }\right\}.
 \end{align}
 Straightforward calculation shows that 
 $$
    \E[\Hess F(\bq)] = \Hess f(\bq).
 $$
 The following lemma provides the deviation inequalities between  $\Hess F\paren{\mb q}$ and $\Hess f(\bq)$ via an application of Lemma \ref{lem:matrix concentration}, stated in Appendix \ref{app_auxiliary}. Recall that $M_n$ is defined in (\ref{def_Mn}).
 
 \begin{lemma} \label{lem:hessian_concentration}
 Under Assumptions \ref{ass_X} and \ref{ass_A_orth}, 
 with probability greater than $1-cn^{-c'}$ for some constants $c,c'>0$, one has
     \begin{align}\nonumber
       \sup_{\mb q\in \Sp^{p-1}}\norm{\Hess F\paren{\mb q}-\Hess f\paren{\mb q}}{\rm op} ~\lesssim ~
       \sqrt{r^3 \log (M_n) \over \theta n} + {M_n\over n} {r \log (M_n) \over n}.
     \end{align}
%  for any $\delta\in \paren{0, cr^2/(\theta\omega_n)}$, if 
%   \begin{align}
%      n\geq \frac{C'r^3}{\theta\delta^{2}}\log\paren{\omega_n \over  \delta},
%  \end{align}
%  then with probability greater than $1-\paren{rn}^{-1}-e^{-c'r\log\paren{\omega_n/\delta}}-2n^{-c''}$, we have 
%  \begin{align}
%     \sup_{\mb q\in \Sp^{p-1}\setminus R_0}\norm{\Hess F\paren{\mb q}-\Hess f\paren{\mb q}}{\rm op}\leq \delta,
%  \end{align}
%  with $R_0$ defined in (\ref{def_R0}). Here $c, c', c''$ and $C'$ are some positive constants.
 \end{lemma}
 \begin{proof}
    Pick any $\bq \in \Sp^{p-1}$ and consider $\bA^T \bq \ne 0$. Recall that
    \[
        \wt \bq = {\bA^T\bq \over \|\bA^T\bq\|_2},\qquad \textrm{with }\quad \wt\bq \in \Sp^{r-1}.
    \]
    Observe that
     \begin{align}
         &\norm{\Hess F\paren{\mb q}-\Hess f\paren{\mb q}}{\rm op}\nonumber\\
         &=\left\|\frac{1}{3\theta\sigma^4n}\sum_{k=1}^nP_{\mb q\perp }\brac{3\paren{\mb \zeta^T\mb x_{k}}^2\mb A\mb x_k\paren{\mb A\mb x_k}^{T}-\paren{\mb \zeta^T\mb x_k}^4\mb I_p}P_{\mb q\perp }\right.\nonumber\\
         &\qquad\left.-\paren{1-\theta}P_{\mb q\perp }\brac{3\mb A\diag(\mb \zeta^{\circ2})\mb A^{T}-\norm{\mb \zeta}4^4\mb I}P_{\mb q\perp}-\theta P_{\mb q\perp }\brac{\norm{\mb\zeta}2^2\mb A\mb A^{T}+2\mb A\mb \zeta\mb \zeta^T\mb A^{T}-\norm{\mb \zeta}2^4\mb I_p }P_{\mb q\perp}\right\|_{\op}\nonumber\\
         &\leq \left\|\frac{1}{3\theta\sigma^4n}\sum_{k=1}^n\brac{3\paren{\mb \zeta^T\mb x_{k}}^2\mb A\mb x_k\paren{\mb A\mb x_k}^{T}-\paren{\mb \zeta^T\mb x_k}^4\mb I_p}\right.\nonumber\\
         &\quad \left.-\paren{1-\theta}\brac{3\mb A\diag(\mb \zeta^{\circ2})\mb A^{T}-\norm{\mb \zeta}4^4\mb I_p}-\theta \brac{\norm{\mb\zeta}2^2\mb A\mb A^{T}+2\mb A\mb \zeta\mb \zeta^T\mb A^{T}-\norm{\mb \zeta}2^4\mb I_p }\right\|_{\op}\nonumber\\
         &\leq\norm{\frac{1}{3\theta\sigma^4n}\sum_{k=1}^n\brac{3\paren{\mb \zeta^T\mb x_{k}}^2\mb A\mb x_k \mb x_k^{T}\mb A^T }-\left[3\left(1-\theta\right)\mb A\diag(\mb \zeta^{\circ2})\mb A^{T}+\theta\left(\norm{\mb\zeta}2^2\mb A\mb A^{T}+2\mb A\mb \zeta\mb \zeta^T\mb A^{T}\right)\right]}{\rm op}\nonumber\\
         &\qquad +\norm{\frac{1}{3\theta\sigma^4n}\sum_{k=1}^n\paren{\mb \zeta^T\mb x_k}^4\mb I_p-\left[\theta\norm{\mb \zeta}2^4+\left(1-\theta\right)\norm{\mb \zeta}4^4\right]\mb I_p}{\rm op}\nonumber\\
         &\leq \norm{\frac{1}{3\theta\sigma^4n}\sum_{k=1}^n\brac{3\paren{\mb \zeta^T\mb x_{k}}^2\mb x_k \mb x_k^{T}}-\left[3\left(1-\theta\right)\diag(\mb \zeta^{\circ2})+\theta\left(\norm{\mb\zeta}2^2\bI_p+2\mb \zeta\mb \zeta^T\right)\right]}{\rm op}\nonumber\\
         &\qquad +\abs{\frac{1}{3\theta\sigma^4n}\sum_{k=1}^n\paren{\mb \zeta^T\mb x_k}^4\mb -\theta\norm{\mb \zeta}2^4-\left(1-\theta\right)\norm{\mb \zeta}4^4}\nonumber\\
         &=\norm{\frac{1}{\theta\sigma^4n}\norm{\mb A^{T}\mb q}2^2\sum_{k=1}^n\brac{\paren{\wt \bq^T \mb x_{k}}^2\mb x_k \mb x_k^{T}}-\norm{\mb A^{T}\mb q}2^2\left[3\left(1-\theta\right)\diag(\wt \bq^{\circ2})+\theta\left(\norm{\wt \bq}2^2\bI_p+2 \wt \bq\wt \bq^T\right)\right]}{\rm op}\nonumber\\
         &\qquad +\abs{\norm{\mb A^{T}\mb q}2^4\frac{1}{3\theta\sigma^4n}\sum_{k=1}^n\paren{\wt \bq^T \mb x_k}^4\mb -\theta\norm{\mb A^{T}\mb q}2^4\norm{\wt \bq}2^4-\left(1-\theta\right)\norm{\mb A^{T}\mb q}2^4\norm{\wt \bq}4^4}.
     \end{align}
     Define 
     \begin{align}\label{def_F_tilde_L2}
        F_{\wt \bq}^{L_2}\paren{\mb x}\doteq\frac{1}{\theta\sigma^3}\paren{\wt \bq^{T}\mb x}^2\mb x\mb x^{T},\qquad F_{\wt \bq}^{L_4}\paren{\mb x}\doteq\frac{1}{3\theta\sigma^3}\paren{\wt \bq^{T}\mb x}^4
     \end{align}
     and 
     \begin{align}\label{def_g_tilde_L2}
         g^{L_2}\paren{\wt \bq} \doteq 3\left(1-\theta\right)\diag(\wt \bq^{\circ2})+\theta\left(\norm{\wt \bq}2^2\mb I_r+2\wt \bq\wt \bq^{T}\right),\qquad g^{L_4}\paren{\wt \bq}\doteq \theta\norm{\wt \bq}2^4+\left(1-\theta\right)\norm{\wt \bq}4^4
     \end{align}
     such that 
     \begin{align}
         \bb E\left[ F_{\wt \bq}^{L_2}(\bx)\right]=g^{L_2}\paren{\wt \bq},\qquad \bb E\left[ F_{\wt \bq}^{L_4}(\bx)\right]=g^{L_4}\paren{\wt \bq}
     \end{align}
     from Lemma \ref{lem:obj}. 
     Using $\mb \zeta =\norm{\mb \zeta}2 \wt \bq$ and $\|\mb\zeta\|_2\le 1$ further yields
     \begin{align}
         \norm{\Hess F\paren{\mb q}-\Hess f\paren{\mb q}}{\rm op}
         &\leq \norm{\frac{1}{n}\sum_{k=1}^nF_{\wt \bq}^{L_2}\paren{\mb x_k}-g^{L_2}\paren{\wt \bq}}{\rm op}+\abs{\frac{1}{3\theta\sigma^4n}\sum_{k=1}^nF_{\wt \bq}^{L_4}\paren{\mb x_k}-g^{L_4}\paren{\wt \bq}}.
     \end{align}
     Notice that the second term has been studied in Appendix \ref{app_deviation_obj}. It suffices to invoke Lemma \ref{lem:matrix concentration}  with $n_1 = d_1 =d_2= r$, $n_2=1$ and $p = n$ to bound from above 
     \[
        \sup_{\wt \bq\in \Sp^{r-1}}\norm{{1\over n}\sum_{k=1}^n F_{\wt \bq}^{L_2}\paren{\mb x_k} -g^{L_2}\paren{\wt \bq}}{\rm op}.
     \]
     We note that $\bx_{ki}$ is sub-Gaussian for all $1\le k\le n$ and $1\le i\le r$. W.L.O.G., we assume $\sigma^2=1$.\\

     \noindent{\bf Verification of Condition 1: } By $\|\wt \bq\|_2=1$, notice that 
         \begin{align}
             \norm{\E\brac{F_{\wt \bq}^{L_2}\left(\mb x\right)}}{\rm op} &=\norm{3\left(1-\theta\right)\diag(\wt \bq^{\circ2})+\theta\left(\norm{\wt \bq}2^2\mb I_r+2\wt \bq\wt \bq^{T}\right)}{\rm op}\nonumber\\
            &\leq3\left(1-\theta\right)\|\wt \bq\|_2^2+3\theta\nonumber\\
            &= 3.
         \end{align}
         For any $\wt\bq_1, \wt\bq_2\in\Sp^{r-1}$, 
         \begin{align}
             &\norm{\E\brac{F_{\wt \bq_1}^{L_2}\left(\mb x\right)}-\E\brac{F_{\wt \bq_2}^{L_2}\left(\mb x\right)}}{\rm op}\nonumber\\
             &=\norm{3\left(1-\theta\right)\diag(\wt \bq_1^{\circ2})+2\theta\wt \bq_1 \wt \bq_1^{T}-3\left(1-\theta\right)\diag(\wt \bq_2^{\circ2})-2\theta\wt \bq_2\wt \bq_2^{T}+\theta\norm{\wt \bq_1}2^2\mb I-\theta\norm{\wt \bq_2}2^2\mb I}{\rm op}\\
             &\leq3\left(1-\theta\right)\norm{\diag(\wt \bq_1^{\circ2})-\diag(\wt \bq_2^{\circ2})}{\rm op}+2\theta\norm{\wt \bq_1 \wt \bq_1^{T}-\wt \bq_2 \wt \bq_2^{T}}2+\theta|\norm{\wt \bq_1}2^2-\norm{\wt \bq_2}2^2|\nonumber\\
             &\leq6\left(1-\theta\right)\norm{\wt \bq_2-\wt \bq_1}\infty+4\theta\norm{\wt \bq_2-\wt \bq_1}{\rm op}+2\theta\norm{\wt \bq_2-\wt \bq_1}{\rm op}\nonumber\\
             &\leq6\norm{\wt \bq_2-\wt \bq_1}2.
         \end{align}
         We thus have $L_f=6$ and $B_f=3$.\\

       \noindent{\bf Verification of Condition 2: } We again work on the event $\cE'$ in (\ref{def_event_E_prime}) such that,
         for each $i\in [n]$,
        \begin{align}
            \norm{F_{\wt \bq}^{L_2}\left(\bar{\mb x}_i\right)}{\rm op}=\frac{1}{\theta\sigma^3}\norm{\paren{\wt \bq^{T}\Bar{\mb x}}^2\bar{\mb x}_i\bar{\mb x}_i^{T}}{\rm op}&\leq\theta^{-1}\norm{\wt \bq}2^2\norm{\bar{\mb x}_i}2^4\\
            &\le C \left(\theta r^2 + {\log^2 n \over \theta}\right).
        \end{align}
        Lemma \ref{lem:upperbound_for_exp_f_hess} in Appendix \ref{app_auxiliary} with some straightforward modifications ensures
        \begin{align}
            \sup_{\bq\in\Sp^{r-1}}\norm{\mathbb{E}\brac{F_{\wt \bq}^{L_2}\left(\bar{\mb x}_i\right)\left(F_{\wt \bq}^{L_2}\left(\bar{\mb x}_i\right)\right)^{T}}}{\rm op}\leq\sup_{\bq\in\Sp^{r-1}}\norm{\mathbb{E}\brac{F_{\wt \bq}^{L_2}\left(\mb x_i\right)\left(F_{\wt \bq}^{L_2}\left(\mb x_i\right)\right)^{T}}}{\rm op}\leq c\theta^{-1} r^2
        \end{align}
        for some constant $c>0$. Therefore, we have
        \begin{align}
         R_1 =C \left(\theta r^2 + {\log^2 n \over \theta}\right),\qquad R_2=c\theta^{-1} r^2.
        \end{align}
        On the other hand, for any $\wt\bq_1, \wt \bq_2\in\Sp^{r-1}$, 
        \begin{align}
            \norm{F_{\wt \bq_1}^{L_2}\left(\bar{\mb x}_i\right)-F_{\wt \bq_2}^{L_2}\left(\bar{\mb x}_i\right)}{\rm op} &\leq \frac{1}{\theta\sigma^3}\abs{\paren{\wt \bq_1^{T}\bar{\mb x}_i}^2 - \paren{\wt \bq_2^{T}\bar{\mb x}_i}^2}\norm{\bar{\mb x}_i \bar{\mb x}_i^{T}}{\rm op}\nonumber\\
            &\leq\frac{2}{\theta}\norm{\bar{\mb x}_i}2^4\norm{\wt \bq_1-\wt \bq_2}2\nonumber\\
            &\leq  2R_1 \norm{\wt \bq_1-\wt \bq_2}2
        \end{align}
        on the event $\cE'$, which implies 
        $\bar{L}_f= 2R_1.$\\
        
        Finally invoke Lemma \ref{lem:matrix concentration} with $M = C'R_1(n+r)$ to conclude the proof.
 \end{proof}
 
 \subsection{Concentration inequalities when $A$ is full column rank}
 In this section we provide deviation bounds for the objective values, Riemannian gradients and Hessian matrices between $F_{g}(\mb q)$ and $\bar{f}_{g}(\mb q)$ defined as 
 \begin{align}
     F_{g}(\mb q)&:=-\frac{\theta n}{12}\norm{\bar{\mb Y}^{T}\mb q}4^4,\\
     \bar{f_{g}}(\mb q)&:=-{1\over 4}\brac{(1-\theta)\norm{\oA^T \mb q}4^4 +\theta \norm{\oA^T \mb q}2^4}
 \end{align}
 where
 \begin{align}\label{def:gene_Y_and_A}
     &\bar{\mb Y}=\left(\left(\mb Y\mb Y^{T}\right)^{+}\right)^{\frac{1}{2}}\mb Y = \bD\bY,\nonumber\\
     &\oA=\left(\left(\mb A\mb A^{T}\right)^{+}\right)^{\frac{1}{2}}\mb A = \bU_A\bV_A^T.
 \end{align}

 \subsubsection{Deviation inequalities between the function values}
Recall that $M_n$ is defined in (\ref{def_Mn}).
 \begin{lemma}\label{lem:function_value_concentration_gene}
 Under Assumptions \ref{ass_X} and \ref{ass_A_rank}, assume 
 \begin{equation}\label{cond_n}
      n \ge C \max\left\{
        {r\log^3 n\over \theta}, ~\theta r^3\log n, ~ {r^2 \over \theta\sqrt \theta}, ~ {r\log n\over \theta^2\sqrt \theta}
      \right\}
  \end{equation}
  for some constant $C>0$. 
 With probability greater than $1- cn^{-c'} - 2e^{-c''r}$, 
 \[
     \sup_{\mb q\in \Sp^{p-1}}| F_{g}\paren{\mb q}-\bar{f_{g}}\paren{\mb q}|
     ~\lesssim ~ \left(\sqrt{r\theta} + \sqrt{\log n}\right)\sqrt{r\over \theta^2 n\sqrt{\theta n}} + \left(\theta r^2 + {\log^2 n\over \theta}\right) {r \log n \over n}.
 \]
%  for any $\delta\in \paren{0, c~ \theta^{-1}\omega_n^{-1}}$, if 
%  \begin{align}
%      n\geq\max\left\{ C{r\over\theta \delta^2}\log\paren{\omega_n \over  \delta}, \frac{C'\theta}{\delta^{2}} r^{5} \log^{4}(n)\right\}
%  \end{align} 
%  then, with probability greater than $1-\paren{rn}^{-1}-e^{-c'r\log\paren{\delta^{-1}\omega_n}}-2\exp^{-c_{1}'r\theta^{-1}}-2\exp^{-c_{2}n}-n^{-c_2'\theta r}$,
%  we have 
%  \begin{align}
%     \sup_{\mb q\in \Sp^{p-1}\setminus R_0}| F_{g}\paren{\mb q}-\bar{f_{g}}\paren{\mb q}|\leq \delta,
%  \end{align}
%  with $R_0$ defined in (\ref{def_R0}).
%  Here $c$, $c'$, $c''$,$c_1'$,$c_2$,$c_2'$,$C'$ and $C$ are some positive constants.
 \end{lemma}

 \begin{proof}
 First we introduce
 \begin{align}\label{def:gene_f_g}
     f_{g}(\mb q):=-\frac{1}{12\sigma^4 \theta n}\norm{\mb q^{T}\oA\mb X}4^4.
 \end{align}
 The proof of Lemma \ref{lem:obj} yields
 \begin{align}
     \mathbb{E}\left[f_{g}(\mb q)\right]=\bar{f_{g}}(\mb q).
 \end{align}
 Triangle inequality gives
 \begin{align}
     \sup_{\mb q\in \Sp^{p-1}}| F_{g}\paren{\mb q}-\bar{f_{g}}\paren{\mb q}|\leq \underbrace{\sup_{\mb q\in \Sp^{p-1}}| F_{g}\paren{\mb q}-f_{g}\paren{\mb q}|}_{\Gamma_1}+\underbrace{\sup_{\mb q\in \Sp^{p-1}}| f_{g}\paren{\mb q}-\bar{f_{g}}\paren{\mb q}|}_{\Gamma_2}.
 \end{align}

 \paragraph{Controlling $\Gamma_1$:}
 Define 
 \begin{align}\label{def:v0_v1}
     \mb v_{0}:=\sqrt{\theta n\sigma^2}\bar{\mb Y}^{T}\mb q\qquad \textrm{and}\qquad\mb v_{1}:=\left(\oA\mb X\right)^{T}\mb q
 \end{align}
 We have
 \begin{align}
     \Gamma_1&=\frac{1}{12\sigma^4 \theta n}\sup_{\mb q\in \Sp^{p-1}}\left|\theta^2 n^2\sigma^4\norm{\mb q^{T}\bar{\mb Y}}4^4-\norm{\mb q^{T}\oA\mb X}4^4\right|\nonumber\\
     &=\sup_{\mb q\in \Sp^{p-1}}\frac{1}{12\sigma^4 \theta n}\left|\left(\norm{\mb v_0}4-\norm{\mb v_1}4\right)\left(\norm{\mb v_0}4+\norm{\mb v_1}4\right)\left(\norm{\mb v_0}4^2+\norm{\mb v_1}4^2\right)\right|\nonumber\\
     &\lesssim \sup_{\mb q\in \Sp^{p-1}}\frac{1}{\sigma^4 \theta n}\norm{\mb v_0-\mb v_1}4 \left(\norm{\mb v_0}4^3 +\norm{\mb v_1}4^3\right).
    %  \label{ieq:gene_func_value_gamma1_upperbound1}
 \end{align}
Invoking Lemma \ref{lemma:bound_v1_v0_inf} gives 
\begin{equation}\label{bd_Gamma_1_obj}
    \P\left\{
        \Gamma_1 \lesssim \left(\sqrt{r\theta} + \sqrt{\log n}\right){1 \over \theta}\sqrt{r\over n\sqrt{\theta n} }
    \right\} \ge 1- 2e^{-cr}-c'n^{-c''r}.
\end{equation}

  \paragraph{Controlling $\Gamma_2$:} Notice that that $\oA^{T}\oA=\mb I_r$. We can thus apply 
 Lemma \ref{lem:function_value_concentration} by replacing $F(\bq)$ and $f(\bq)$ with $f_{g}\paren{\mb q}$ and $\bar{f_{g}}\paren{\mb q}$, respectively, to obtain
 \begin{align}\nonumber
    \P\left\{ \Gamma_2 ~ \lesssim ~
   \sqrt{r \log (M_n) \over \theta n} + {M_n\over n} {r \log (M_n) \over n}
   \right\}
   \ge  1-cn^{-c'}.
 \end{align}
 Combining the bounds of $\Gamma_{1}$ and $\Gamma_{2}$ and using (\ref{cond_n}) to simplify the expressions complete the proof.
 \end{proof}

 \bigskip
 
 Recall that, for any $\bq \in \Sp^{p-1}$,
 \[
    \bv_0 := \sqrt{\theta n\sigma^2}\bar{\mb Y}^{T}\mb q\qquad \textrm{and}\qquad\mb v_{1}:=\left(\oA\mb X\right)^{T}\mb q
 \]
 with $\bar \bA = \bU_A\bV_A^T$ and $\bar\bY = \bD\bY$.
 
 \begin{lemma}\label{lemma:bound_v1_v0_inf}
  Assume $n \ge Cr/\theta^2$ for some constant $C>0$. With probability $1-2e^{-cr}-2n^{-c'}$ for some constant $c,c'>0$, one has 
  \begin{align}
        \sup_{\bq \in \Sp^{p-1}} \left\|\bv_0 - \bv_1\right\|_\infty  ~ \lesssim  ~ \sigma\left(\sqrt{r\theta} + \sqrt{\log n}\right){1 \over \theta}\sqrt{r\over n}.
  \end{align}
  Furthermore, if  additionally 
  (\ref{cond_n})
  holds, then with probability $1- 2e^{-cr}-c'n^{-c''r}$, 
  \begin{align}
    &\sup_{\bq \in \Sp^{p-1}} \left\|\bv_1\right\|_4 \lesssim  (\theta n\sigma^4)^{1/4},\qquad \sup_{\bq \in \Sp^{p-1}} \left\|\bv_0\right\|_4 \lesssim  (\theta n\sigma^4)^{1/4}.
  \end{align}
  
 \end{lemma}
 \begin{proof}
 We work on the event $\cE'$, defined in (\ref{def_event_E_prime}), intersecting with 
 \begin{equation}\label{def_event_E_double_prime}
    \cE'' := \left\{
        \left\|\sqrt{\theta n \sigma^2} \bD \bA - \bar \bA\right\|_{\op} \lesssim 
        {1\over \theta}\sqrt{r\over n}
    \right\},
 \end{equation}
 which, according to Lemmas \ref{lemma:bound_X_inf_column} and \ref{lem:bound_Y_pre_A_pre}, holds with probability $1 - 2e^{-cr} - 2n^{-c'}$.  
 Recall $\bar \bY = \bD \bY = \bD\bA \bX$. 
 By definition, 
    \begin{align}
        \norm{\mb v_0-\mb v_1}\infty &= \max_{t\in [n]}\left|
        \bq^T \left(\bar \bA - \sqrt{\theta n \sigma^2}\bD \bA \right)\bx_t\right|\nonumber\\
        &\le \max_{t\in [n]}\|\bx_t\|_2  \|\bq\|_2  \left\|\sqrt{\theta n \sigma^2} \bD \bA - \bar \bA \right\|_{\op}\nonumber\\
        &\lesssim \sigma\left(\sqrt{r\theta} + \sqrt{\log n}\right){1 \over \theta}\sqrt{r\over n} &\textrm{(by $\cE'\cap \cE''$)}.
    \end{align}
    
    To bound from above 
    $\|\bv_1\|_4$, by recalling $F(\bq)$ and $f(\bq)$ from (\ref{def_F_f_proof}) with $\bar \bA$ in lieu of $\bA$, we observe that 
    \begin{align}
        \|\bv_1\|_4^4 &= \left\|
        \bq^T \bar\bA \bX \right\|_4^4 = 12\theta n \sigma^4 |F(\bq)| \le 12\theta n \sigma^4 \left( |F(\bq) - f(\bq)| + |f(\bq)| \right).
    \end{align}
    By Lemma \ref{lem:function_value_concentration} and $|f(\bq)|\le 1$ from its proof, we obtain 
    \begin{align}
        \sup_{\bq \in \Sp^{p-1}}\|\bv_1\|_4^4  ~ \le ~   12\theta n \sigma^4\left( 1 +  \sqrt{r \log (M_n) \over \theta n} + {M_n\over n} {r \log (M_n) \over n}\right)
    \end{align}
    with probability at least 
    $1-\paren{nr}^{-2}-cM_n^{-c'r}$ for some constants $c,c'>0$. Here $M_n$ is defined in (\ref{def_Mn}). The result then follows by invoking condition (\ref{cond_n}) and noting that $\log M_n \lesssim \log n$.
    
    Finally, since
    \[
        \|\bv_0\|_4 \le \|\bv_1\|_4 + \|\bv_0 - \bv_1\|_4 \le  \|\bv_1\|_4 + n^{1/4}\|\bv_0 - \bv_1\|_\infty,
    \]
    the last result follows by combining the previous two results.
 \end{proof}

 \subsubsection{Deviation inequalities between the Riemannian gradients}
 In this section, we derive the deviation inequalities between the Riemannian gradient of $F_g(\bq)$ and that of $\bar f_g\paren{\mb q}$. Note that, for any $\bq\in \Sp^{p-1}$, 
\begin{align}
    &\grad F_{g}\paren{\mb q} \doteq \grad_{\norm{\mb q}2=1}F_{g}\paren{\mb q}=-\frac{\theta n}{3}\mb P_{\mb q\perp}\sum_{k=1}^{n}\paren{\mb q^{T} \bar{\mb Y}_{k}}^3\bar{\mb Y}_{k},\\
    &\grad \bar{f}_{g}\paren{\mb q}\doteq \grad_{\norm{\mb q}2=1}\bar{f}_{g}\paren{\mb q}= -\mb P_{\mb q^\perp}\brac{
    (1-\theta)\sum_{j=1}^r {\oA_{j}} (\mb q^T {\oA_{j}})^3 + \theta \|\mb q^T {\oA}\|_2^2  {\oA}{\oA}^T \mb q}.
\end{align}
Here $\bar{\mb Y}$ and $\oA$ are defined in (\ref{def:gene_Y_and_A}). 
\begin{lemma}\label{lem:gene_gradient_concentration}
 Under Assumptions \ref{ass_X} and \ref{ass_A_rank}, 
 assume 
 \begin{equation}\label{cond_n_prime}
      n \ge C \max\left\{
        {r\log^3 n\over \theta}, ~\theta r^3\log n, ~ {r^2 \over \theta\sqrt \theta},  ~{r^2\log n \over \theta}, ~ {r\log n\over \theta^2\sqrt \theta}
      \right\}
  \end{equation}
  for some constant $C>0$. 
 With probability greater than $1- cn^{-c'} - 2e^{-c''r}$, 
 \[
     \sup_{\mb q\in \Sp^{p-1}}\norm{\grad F_{g}\paren{\mb q}-\grad \bar{f}_{g}\paren{\mb q}}2
     ~\lesssim ~ 
     \sqrt{r\log n\over \theta^2 n} +\sqrt{r^2\log n \over \theta n}+ \left(\theta r^2 + {\log^2 n\over \theta}\right) {r \log n \over n}\\
    %  \left(\sqrt{r\theta} + \sqrt{\log n}\right)\sqrt{r\over \theta^2\sqrt{\theta} n} + \left(\theta r^2 + {\log^2 n\over \theta}\right) {r \log n \over n}.
 \]
  \end{lemma}
 \begin{proof}
    Recall $f_{g}\left(\mb q\right)$ from (\ref{def:gene_f_g}). Its Riemannian gradient is \begin{align}
     \grad f_{g}\left(\mb q\right):=\grad_{\norm{\mb q}2=1} f_{g}\left(\mb q\right)=-\frac{1}{3\theta\sigma^4 n}\mb P_{\mb q\perp}\sum_{k=1}^{n}\paren{\mb q^{T} \oA\mb x_{k}}^3\oA\mb x_{k}.
 \end{align}
 We have 
 \begin{align}
     &\sup_{\mb q\in \Sp^{p-1}}\|\grad F_{g}\paren{\mb q}-\grad \bar{f}_{g}\paren{\mb q}\|_{2}\\
      & \qquad \leq\underbrace{\sup_{\mb q\in \Sp^{p-1}}\|\grad F_{g}\paren{\mb q}-\grad f_{g}\paren{\mb q}\|_{2}}_{\Gamma_1}+\underbrace{\sup_{\mb q\in \Sp^{p-1}}\|\grad f_{g}\paren{\mb q}-\grad \bar{f}_{g}\paren{\mb q}\|_{2}}_{\Gamma_2}.
 \end{align}
 \paragraph{\bf{Controlling $\Gamma_1$}} Recall that $\mb v_0$ and $\mb v_1$ are defined in (\ref{def:v0_v1}). We have:
 \begin{flalign}
     \Gamma_1&=\sup_{\mb q\in \Sp^{p-1}}\norm{\grad F_{g}\paren{\mb q}-\grad f_{g}\paren{\mb q}}2\nonumber\\
     &=\sup_{\mb q\in \Sp^{p-1}}\frac{1}{3\theta\sigma^{4}n}\norm{\theta^2 n^2\sigma^4\mb P_{\mb q\perp}\sum_{k=1}^{n}\paren{\mb q^{T} \bar{\mb Y}_{k}}^3\bar{\mb Y}_{k}-\mb P_{\mb q\perp}\sum_{k=1}^{n}\paren{\mb q^{T} \oA\mb x_{k}}^3\oA\mb x_{k}}2\nonumber\\
      &\leq\sup_{\mb q\in \Sp^{p-1}}\frac{1}{3\theta\sigma^{4}n}\left\|\sum_{k=1}^n \left(\sqrt{\theta n\sigma^2}\bar\bY_k \bv_{0k}^{3} - \bar\bA\bx_k \bv_{1k}^{3}\right)\right\|_2\nonumber\\
      &\leq \underbrace{\sup_{\mb q\in \Sp^{p-1}}\frac{1}{3\theta\sigma^{4}n}\left\|\sum_{k=1}^n \left(\sqrt{\theta n\sigma^2}\bar\bY_k - \bar\bA \bx_k\right) \bv_{0k}^{3}\right\|_2}_{\Gamma_{11}}\nonumber\\
      & \qquad + \underbrace{\sup_{\mb q\in \Sp^{p-1}}\frac{1}{3\theta\sigma^{4}n}\left\|\sum_{k=1}^n  \bar\bA \bx_k(\bv_{0k}^{3}-\bv_{1k}^3)\right\|_2}_{\Gamma_{12}}.\nonumber
 \end{flalign}
 For $\Gamma_{11}$, we obtain 
 \begin{align}
     \Gamma_{11} &= \sup_{\mb q\in \Sp^{p-1}}\frac{1}{3\theta\sigma^{4}n}\left\|\left(\sqrt{\theta n\sigma^2}\bD\bA - \bar\bA\right) \sum_{k=1}^n \bx_k \bv_{0k}^{3}\right\|_2\nonumber\\
     &\le  \sup_{\mb q\in \Sp^{p-1}}\frac{1}{3\theta\sigma^{4}n}\left\| \sum_{k=1}^n \bx_k \bv_{0k}^{3}\right\|_2\left\|\sqrt{\theta n\sigma^2}\bD\bA - \bar\bA\right\|_{\op} \nonumber\\
     &\le  \sup_{\mb q\in \Sp^{p-1}}\frac{1}{3\theta\sigma^{4}n}\left(\left\| \sum_{k=1}^n \bx_k \bv_{1k}^{3}\right\|_2 + \left\| \sum_{k=1}^n \bx_k (\bv_{0k}^{3}-\bv_{1k}^3)\right\|_2\right)\left\|\sqrt{\theta n\sigma^2}\bD\bA - \bar\bA\right\|_{\op}.
 \end{align}
 Observing 
 \[
     \sum_{k=1}^n \bx_k \bv_{1k}^{3} = 3\theta n\sigma^4  {1\over n}\norm{\bar{\mb A}^{T}\mb q}2^3\sum_{k=1}^nF_{\wt \bq}(\bx_k)
 \]
 with $F_{\wt \bq}(\bx_k)$ defined in (\ref{def_F_tilde_q}) and $\wt \bq = \frac{\bar\bA^T \bq}{\norm{\bar\bA^T \bq}2}\in \mathbb{S}^{r-1}$. We also have 
 \[
    \E[F_{\wt \bq}(\bx_k)] = g(\wt\bq)\overset{(\ref{def_f_tilde_q})}{=} \brac{\paren{1-\theta}\paren{\wt \bq}^{\circ3}+\theta\norm{\wt \bq}2^2\wt \bq}
 \]
 with 
 \[
    \sup_{\wt\bq\in \Sp^{r-1}}\|g(\wt\bq)\|_2  \le 1.
 \]
 Lemma \ref{lem:gradient_concentration} and its proof guarantee that 
 \begin{align}
    \sup_{\bq\in\Sp^{p-1}}{1\over  3\theta n\sigma^4 }\left\| \sum_{k=1}^n \bx_k \bv_{1t}^{3}\right\|_{2} &\le 
    \norm{\bar{\mb A}^{T}\mb q}2^3\sup_{\wt\bq\in \Sp^{r-1}}\left\|  {1\over n}\sum_{k=1}^nF_{\wt \bq}(\bx_k) - g(\wt\bq)\right\|_2 + \norm{\bar{\mb A}^{T}\mb q}2^3\sup_{\wt\bq\in \Sp^{r-1}}\left\|g(\wt\bq)\right\|_2\\
    &\lesssim 1 + \sqrt{r^2 \log (M_n) \over \theta n} + {M_n\over n} {r \log (M_n) \over n}\\
    &\lesssim 1
 \end{align}
 with probability at least $1 - cn^{-c'}$, where $M_n$ is defined in (\ref{def_Mn}). We used condition (\ref{cond_n_prime}) to simplify the expressions in the last step and $\norm{\bar{\mb A}^{T}\mb q}2^3\leq 1$ in the second step. Invoke $\cE''$ in (\ref{def_event_E_double_prime}) to conclude 
 \begin{equation}\label{bd_Gamma_11}
    \Gamma_{11} \lesssim {1\over \theta}\sqrt{r\over n}\left( 1+ \Gamma_{12}\right)
 \end{equation}
 with probability at least $1 - cn^{-c'}-2e^{-c''r}$. 
 
 To control $\Gamma_{12}$, we have 
 \begin{align}
     \Gamma_{12} & = \sup_{\mb q\in \Sp^{p-1}}\frac{1}{3\theta\sigma^{4}n}\left\|\sum_{k=1}^n  \bx_k(\bv_{0k}^{3}-\bv_{1k}^3)\right\|_2\nonumber\\
     & \le \sup_{\mb q\in \Sp^{p-1}}\frac{1}{3\theta\sigma^{4}n}\left(\left\|\sum_{k=1}^n  \bx_k(\bv_{0k}-\bv_{1k})\bv_{1k}^2\right\|_2 + \left\|\sum_{k=1}^n  \bx_k\bv_{0k}(\bv_{0k}-\bv_{1k})(\bv_{0k}+\bv_{1k})\right\|_2\right)\nonumber\\
     &\lesssim 
     {1\over \theta n\sigma^4}\|\bX\|_{\op}\sup_{\mb q\in \Sp^{p-1}} \left(
        \left\|\bv_1^{\circ 2}\circ(\bv_0 - \bv_1)\right\|_2 + \left\|\bv_0^{\circ 2}\circ(\bv_0 - \bv_1)\right\|_2 
     \right)\nonumber\\
     &\le {1\over \theta n\sigma^4}\|\bX\|_{\op}\sup_{\mb q\in \Sp^{p-1}} \left(
        \left\|\bv_0\right\|_4^2 +  \left\|\bv_1\right\|_4^2
     \right)\|\bv_0-\bv_1\|_{\infty}
 \end{align}
 where in the penultimate step we used 
 \[
    \|\bv^{\circ 2} \circ\bv'\|_2^2 = \sum_{i}\bv_i^4 (\bv_i')^2 \le \|\bv'\|_{\infty}^2 \|\bv\|_4^4.
 \]
 Invoking Lemma \ref{lemma:bound_X_norm_op} and Lemma \ref{lemma:bound_v1_v0_inf} yields
 \begin{equation}\label{bd_Gamma_12}
    \Gamma_{12} \lesssim \left(\sqrt{r\theta} + \sqrt{\log n}\right){1 \over \theta}\sqrt{r\over n}
 \end{equation}
 with probability at least $1-cn^{-c'} - 2e^{-c''r}$.

 \paragraph{Controlling $\Gamma_2$:} Since $ \oA^{T}\oA=\mb I_r$ and direct calculation gives 
 \begin{align}
     \mathbb{E}\left[\grad f_{g}\left(\mb q\right)\right]=\grad \bar{f}_{g}\left(\mb q\right).
 \end{align}
 Applying lemma \ref{lem:gradient_concentration} with $F(\bq)$ and $f(\bq)$ replaced by $f_{g}\left(\mb q\right)$ and $\bar{f}_{g}\left(\mb q\right)$, respectively, gives
 \begin{align}\label{bd_Gamma_2}
     \P\left\{\Gamma_{2} \lesssim \sqrt{r^2 \log (M_n) \over \theta n} + {M_n\over n} {r \log (M_n) \over n}\right\} \ge 1-cn^{-c'}.
 \end{align}

 Finally collecting (\ref{bd_Gamma_11}), (\ref{bd_Gamma_12}) and (\ref{bd_Gamma_2}) and using (\ref{cond_n_prime}) to simplify the expression finish the proof.
 \end{proof}

\subsubsection{Deviation inequalities of the Riemannian Hessian}
 In this part we will show that the Hessian of $F_{g}(\bq)$ concentrates around that of $\bar{f}_{g}(\bq)$. Notice that, for any $\bq\in\Sp^{-1}$, 
 \begin{align}
     &\Hess F_{g}\left(\mb q\right)=-\frac{\theta n}{3}\sum_{k=1}^nP_{\mb q\perp }\brac{3\paren{\mb q^{T}\bar{\mb Y}_{k}}^2\bar{\mb Y}_{k}\paren{\bar{\mb Y}_{k}}^{T}-\paren{\mb q^{T}\bar{\mb Y}_{k}}^4\mb I_p}P_{\mb q\perp },\\ \nonumber
     & \Hess \bar{f}_{g}\left(\mb q\right)=-\biggl\{\paren{1-\theta}P_{\mb q\perp }\brac{3\oA\diag(\left(\oA^{T}\mb q\right)^{\circ2})\oA^{T}-\norm{\oA^{T}\mb q}4^4\mb I_p}P_{\mb q\perp }.\\
     &\hspace{2.5cm}+\theta P_{\mb q\perp }\brac{\norm{\mb q^{T}\oA}2^2\oA\oA^{T}+2\oA\oA^{T}\mb q\mb q^{T}\oA\oA^{T}-\norm{\mb q^{T}\oA}2^4\mb I_p }P_{\mb q\perp }\biggr\}.
 \end{align}
 Here $\bar{\mb Y}$ and $\oA$ are defined in (\ref{def:gene_Y_and_A}). 
 
 \begin{lemma} \label{lem:gene_hessian_concentration}
     Under Assumptions \ref{ass_X} and \ref{ass_A_rank}, assume 
     \begin{equation}\label{cond_n_double_prime}
          n \ge C \max\left\{
            {r\log^3 n\over \theta}, ~ {r\log n\over \theta^2\sqrt \theta}, ~ {r^2 \over \theta\sqrt \theta},  ~{r^3\log n \over \theta}
          \right\}
      \end{equation}
      for some constant $C>0$. 
     With probability greater than $1- cn^{-c'} - 4e^{-c''r}$, 
     \begin{align}
         &\sup_{\mb q\in \Sp^{p-1}}\left\|\Hess F_{g}\paren{\mb q}-\Hess \bar{f}_{g}\paren{\mb q}\right\|_{\op}\nonumber\\
         &~\lesssim ~   \left(\sqrt{r\sqrt \theta} + \sqrt{\log n \over \sqrt \theta} + {\log n}\right)\sqrt{r\over \theta^2 n} +\sqrt{r^3\log n \over \theta n}+ \left(\theta r^2 + {\log^2 n\over \theta}\right) {r \log n \over n}.
        %  \left(\sqrt{r\theta} + \sqrt{\log n}\right)\sqrt{r\over \theta^2\sqrt{\theta} n} + \left(\theta r^2 + {\log^2 n\over \theta}\right) {r \log n \over n}.
     \end{align}
 \end{lemma}
\begin{proof}
    Recall $f_{g}\paren{\mb q}$ from (\ref{def:gene_f_g}). 
  Notice that 
   \begin{align}
       \sup_{\mb q\in \Sp^{p-1}}\norm{\Hess F_{g}\paren{\mb q}-\Hess \bar{f}_{g}\paren{\mb q}}{\rm op}\leq&\underbrace{ \sup_{\mb q\in \Sp^{p-1}}\norm{\Hess F_{g}\paren{\mb q}-\Hess f_{g}\paren{\mb q}}{\rm op}}_{\Gamma_1}\nonumber\\
       &+\underbrace{\sup_{\mb q\in \Sp^{p-1}}\norm{\Hess f_{g}\paren{\mb q}-\Hess \bar{f}_{g}\paren{\mb q}}{\rm op}}_{\Gamma_2}.
   \end{align}
   Straightforward calculation gives 
   \begin{align}
       \Hess f_{g}\left(\mb q\right)=-\frac{1}{3\theta\sigma^4n}\sum_{k=1}^nP_{\mb q\perp }\brac{3\paren{\mb q^{T}\oA\mb X_{k}}^2\oA\mb X_{k}\paren{\oA\mb X_{k}}^{T}-\paren{\mb q^{T}\oA\mb X_{k}}^4\mb I_p}P_{\mb q\perp }
   \end{align}
   It remains to bound from above $\Gamma_1$ and $\Gamma_2$ respectively.
   \paragraph{Controlling $\Gamma_1$:} Using the definition of $\mb v_0$ and $\mb v_1$ in (\ref{def:v0_v1}), we have:
   \begin{align}
       \Gamma_{1} &= \sup_{\mb q\in \Sp^{p-1}}\norm{\Hess F_{g}\paren{\mb q}-\Hess f_{g}\paren{\mb q}}{\rm op}\nonumber\\
       &\leq\frac{1}{3\theta\sigma^4n} \sup_{\mb q\in \Sp^{p-1}}\left\|\theta^2n^2\sigma^4\left[3\bar{\mb Y}\diag((\bar{\mb Y}^{T}\mb q)^{\circ 2})\bar{\mb Y}^{T}-\norm{\bar{\mb Y}^{T}\mb q}4^4\mb I\right]\right.\nonumber\\
       &\qquad \left.-\left[3\oA\mb X\diag((\left(\oA\mb X\right)^{T}\mb q)^{\circ 2})\left(\oA\mb X\right)^{T}-\norm{\left(\oA\mb X\right)^{T}\mb q}4^4\mb I\right]\right\|_{\op}\nonumber\\
       &\leq\underbrace{\frac{1}{\theta\sigma^4n}\sup_{\mb q\in \Sp^{p-1}}\left\|\theta n\sigma^2\bar{\mb Y}\diag\left(\mb v_0^{\circ 2}\right)\bar{\mb Y}^{T}-\oA\mb X\diag\left(\mb v_1^{\circ 2}\right)\left(\oA\mb X\right)^{T}\right\|_{\op}}_{\beta_1}\\
       &\qquad +\underbrace{\frac{1}{3\theta\sigma^4n}\sup_{\mb q\in \Sp^{p-1}}\left|\norm{\mb v_0}4^4-\norm{\mb v_1}4^4\right|}_{\beta_2}.
   \end{align}
   \paragraph{Upper bound for $\beta_{1}$:} By adding and subtracting terms, we have
       \begin{align}
           \beta_{1}&=\underbrace{\frac{1}{\theta\sigma^4n}\sup_{\mb q\in \Sp^{p-1}}
            \left\| (\sqrt{\theta n\sigma^4}\bar \bY - \bar\bA\bX) \diag(\bv_1^{\circ 2}) \bX^T\right\|_{\op}}_{\beta_{11}}\nonumber\\
           &  + \underbrace{\frac{1}{\theta\sigma^4n}\sup_{\mb q\in \Sp^{p-1}}\left\| \sqrt{\theta n\sigma^4}\bar\bY\diag(\bv_1^{\circ 2})(\sqrt{\theta n\sigma^4}\bar \bY - \bar\bA\bX)^T\right\|_{\op}}_{\beta_{12}}\nonumber\\
           &  + 
            \underbrace{\frac{1}{\theta\sigma^4n}\sup_{\mb q\in \Sp^{p-1}}\left\|\sqrt{\theta n\sigma^4}\bar\bY\diag(\bv_1^{\circ 2}-\bv_0^{\circ 2})\sqrt{\theta n\sigma^4}\bar\bY^T\right\|_{\op}}_{\beta_{13}}.
       \end{align}
       For $\beta_{11}$, by recalling that (\ref{def:gene_Y_and_A}), we have 
       \begin{align}
           \beta_{11} &= \frac{1}{\theta\sigma^4n}\sup_{\mb q\in \Sp^{p-1}}
            \left\| (\sqrt{\theta n\sigma^4}\bD\bA - \bar\bA)\bX \diag(\bv_1^{\circ 2}) \bX^T\right\|_{\op}\nonumber\\
            &\le  \frac{1}{\theta\sigma^4n}\sup_{\mb q\in \Sp^{p-1}}
            \left\|\sum_{t=1}^n \bv_{1t}^2\bx_t \bx_t^T\right\|_{\op} \left\|\sqrt{\theta n\sigma^4}\bD\bA - \bar\bA\right\|_{\op}.
       \end{align}
       Recalling (\ref{def_F_tilde_L2}) and (\ref{def_g_tilde_L2}), we have  
       \begin{align}
            \frac{1}{\theta\sigma^4n}
            \left\|\sum_{t=1}^n \bv_{1t}^2\bx_t \bx_t^T\right\|_{\op} \le &\norm{\bar\bA^T\bq}2^2\left\|g^{L_2}(\wt\bq)\right\|_{\op} + \norm{\bar\bA^T\bq}2^2\left\|{1\over n}\sum_{t=1}^n F_{\wt\bq}^{L_2}(\bx_t) - g^{L_2}(\wt\bq)\right\|_{\op}\nonumber\\
           \leq&\left\|g^{L_2}(\wt\bq)\right\|_{\op} + \left\|{\frac{1}{n}}\sum_{t=1}^n F_{\wt\bq}^{L_2}(\bx_t) - g^{L_2}(\wt\bq)\right\|_{\op}
       \end{align}
       with $\wt\bq = \bar\bA^T\bq/\norm{\bar\bA^T\bq}2\in\mathbb{S}^{r-1}$ and $\|g^{L_2}(\wt\bq)\|_{\op}\le 3$. Hence, according to the proof of Lemma \ref{lem:hessian_concentration}, invoke Lemma \ref{lem:hessian_concentration} and $\cE''$ in (\ref{def_event_E_double_prime}) together with (\ref{cond_n_double_prime}) to conclude that 
       \begin{equation}\label{bd_beta_11}
        \P\left\{
            \beta_{11} \lesssim {1\over \theta}\sqrt{r\over n}
        \right\} \ge 1 - cn^{-c'} - 2e^{-c''r}.
       \end{equation}
       By similar arguments and $\bar\bY = \bD\bA\bX$, we have 
       \begin{align}
           \beta_{12} \le \frac{1}{\theta\sigma^4n}\sup_{\mb q\in \Sp^{p-1}}
            \left\| \sum_{t=1}^n \bv_{1t}^2\bx_t \bx_t^T\right\|_{\op} \left\|\sqrt{\theta n\sigma^4}\bD\bA - \bar\bA\right\|_{\op}\left\|\sqrt{\theta n\sigma^4}\bD\bA\right\|_{\op}.
       \end{align}
       Since on the event $\cE''$, condition (\ref{cond_n_double_prime}) ensures 
       \begin{equation}\label{bd_DA}
        \left\|\sqrt{\theta n\sigma^4}\bD\bA\right\|_{\op} \le 1 + \left\|\sqrt{\theta n\sigma^4}\bD\bA-\bar\bA\right\|_{\op}\lesssim 1.
       \end{equation}
       We obtain 
       \begin{equation}\label{bd_beta_12}
        \P\left\{
            \beta_{12} \lesssim {1\over \theta}\sqrt{r\over n}
        \right\} \ge 1 - cn^{-c'} - 2e^{-c''r}.
       \end{equation}
       Finally, on the event $\cE''$, 
       \begin{align}
           \beta_{13} &\le \frac{1}{\theta\sigma^4n}\sup_{\mb q\in \Sp^{p-1}}\left\|\sum_{t=1}^n(\bv_{1t}-\bv_{0t})(\bv_{1t}+\bv_{0t})\bx_t\bx_t^T\right\|_{\op}\left\|\sqrt{\theta n\sigma^4}\bD\bA\right\|_{\op}^2\nonumber\\
           &\lesssim \frac{1}{\theta\sigma^4n}\sup_{\mb q\in \Sp^{p-1}}\left\|\sum_{t=1}^n(\bv_{1t}-\bv_{0t})(\bv_{1t}+\bv_{0t})\bx_t\bx_t^T\right\|_{\op}\nonumber\\
           &\lesssim \frac{1}{\theta\sigma^4n}\sup_{\mb q\in \Sp^{p-1}}\|
           \bX\|_{\op}^2\|\bv_0-\bv_1\|_{\infty}\left(\|\bv_0\|_{\infty} + \|\bv_1\|_{\infty}\right).
       \end{align}
       Since on the event $\cE'$ in (\ref{def_event_E_prime}), 
       \[
            \|\bv_1\|_{\infty} = \max_{t\in [n]}\|\bq^T\bar\bA\bx_t\|_{\infty} \lesssim \sigma(\sqrt{r\theta} + \sqrt{\log n}),
       \]
       and 
       \[
             \|\bv_0\|_{\infty} \le  \|\bv_1\|_{\infty} + \|\bv_0 - \bv_1\|_{\infty},
       \]
       invoke Lemma \ref{lemma:bound_v1_v0_inf} and Lemma \ref{lemma:bound_X_norm_op} to obtain 
       \begin{equation}\label{bd_beta_13}
          \beta_{13} \lesssim (r\theta + \log n){1\over \theta}\sqrt{r\over n} =  \sqrt{r^3\over n} + \sqrt{r\log^2 n \over \theta^2 n}
       \end{equation}
       with probability at least $1-4e^{-cr} - c'n^{-c''}$.

    \paragraph{Upper bound for $\beta_2$}
    Notice that 
    \begin{align} 
        \beta_2=\frac{1}{4}\sup_{\mb q\in \Sp^{p-1}}\left|F_{g}\left(\mb q\right)-f_{g}\left(\mb q\right)\right|.
    \end{align}
    Display (\ref{bd_Gamma_1_obj}) yields
    \begin{align}\label{bd_beta_2}
        \P\left\{
        \beta_2 \lesssim \left(\sqrt{r\theta} + \sqrt{\log n}\right){1 \over \theta}\sqrt{r\over \sqrt{\theta} n}
    \right\} \ge 1- 2e^{-cr}-c'n^{-c''r}.
    \end{align}
    
    \paragraph{Controlling $\Gamma_2$}
    Notice that $ \oA^{T}\oA=\mb I_r$ and simple calculation gives 
     \begin{align}
         \mathbb{E}\left[\Hess f_{g}\left(\mb q\right)\right]=\Hess \bar{f}_{g}\left(\mb q\right).
     \end{align}
     Apply Lemma \ref{lem:hessian_concentration} with $F(\bq)$ and $f(\bq)$ replaced by $f_{g}\left(\mb q\right)$ and $\bar{f}_{g}\left(\mb q\right)$, respectively, to obtain
     \begin{align}\label{bd_Gamma_2_hess}
         \P\left\{
         \Gamma_{2}\lesssim  \sqrt{r^3 \log (M_n) \over \theta n} + {M_n\over n} {r \log (M_n) \over n}
         \right\} \ge 1-cn^{-c'}.
     \end{align}

     Finally, collecting (\ref{bd_beta_11}), (\ref{bd_beta_12}), (\ref{bd_beta_13}), (\ref{bd_beta_2}) and (\ref{bd_Gamma_2_hess}) and using condition (\ref{cond_n_double_prime}) to simplify expressions complete the proof. 
\end{proof}

\newpage

 \section{Auxiliary lemmas}\label{app_auxiliary}
 
 Recall that the SVD of $\bA$ is $\bU_A \bD_A \bV_A^T$ and $\bD$ is defined in (\ref{def_D}). 
 \begin{lemma}\label{lem:bound_Y_pre_A_pre}
     Under Assumptions \ref{ass_X} and \ref{ass_A_rank}, assume $n\ge Cr/\theta^2$ for some constant $C>0$. With probability at least $1-2e^{-cr}$ for some constant $c>0$, we have:
     \begin{align}
         \left\|
         \sqrt{\theta n\sigma^2} \bD \bA - \bU_A \bV_A^T
         \right\|_{\op} \lesssim {1\over \theta} \sqrt{r\over n}.
    \end{align}
 \end{lemma}
 
 \begin{proof}
  From (\ref{eqn_D}), we have 
  \[
         \sqrt{\theta n\sigma^2}\bD  =  \bU_A \bD_A^{-1} \bU_A^T + \bU_A\left( \sqrt{\theta n\sigma^2}\mb Q^T \bLbd^{-1/2}\mb Q - \bD_A^{-1}\right) \bU_A^T.
  \]
  where $\bU$ contains the left $r$ singular vectors of $\bY$ and $\bU_A = \bU \mb Q$. It then follows 
  \begin{align}
      \left\|
      \sqrt{\theta n\sigma^2} \bD \bA - \bU_A \bV_A^T
      \right\|_{\op} &= \left\|\bU_A\left(\sqrt{\theta n\sigma^2} 
        \mb Q^T\bLbd^{-1/2}\mb Q \bD_A- \bI_r\right)\bV_A^T\right\|_{\op}\nonumber\\
      &= \left\|\sqrt{\theta n\sigma^2} 
        \mb Q^T\bLbd^{-1/2}\mb Q \bD_A- \bI_r\right\|_{\op}.
  \end{align}
  The result follows by invoking (\ref{bd_QLbdQDA}).
 \end{proof}

\begin{lemma}\label{lemma:bound_X_norm_op}
     Under Assumption \ref{ass_X}, assume $n \ge Cr/\theta^2$ for some constant $C>0$. 
     One has
     \begin{align}
         \left\|\mb X\right\|_{\op} \lesssim \sqrt{\theta n \sigma^2}
     \end{align}
     with probability at least $1-2e^{-cr}$. Here $c$ is some positive constant.
 \end{lemma}
  \begin{proof}
  Assume $\sigma^2=1$ without loss of generality.
  Recall from (\ref{bd_cov_X}) that 
 \begin{align}
     \left\|\frac{1}{n\sigma^2}\mb X\mb X^{T}-\theta \mb I\right\|_{\op}\leq c\left(\sqrt{\frac{r}{n}} +{r\over n}\right)
 \end{align}
 with probability at least $1-2e^{-c'r}$ for some constants $c,c'>0$. With the same probability, it follows immediately 
 \[
    {1\over n\sigma^2}\|\bX\bX^T\|_{\op} \le \theta + c\left(\sqrt{\frac{r}{n}} +{r\over n}\right) \le c''\theta 
 \]
 provided that
 \[
     n \ge Cr/\theta^2
 \]
 for some constant $C>0$.
 \end{proof}
 
 \medskip
 
 Recall that $\bX = (\bx_1, \ldots, \bx_n)\in \R^{r\times n}$. The following lemma provides the upper bound of $\max_{i\in[n]} \|\bx_i\|_2$.
 
 \begin{lemma}\label{lemma:bound_X_inf_column}
 Under Assumption \ref{ass_X},  we have
 \begin{align}
     \max_{i\in [n]}\|\mb x_i\|_{2} \lesssim  \sigma\left(\sqrt{r\theta} + \sqrt{\log(n)}\right)
 \end{align}
 with probability at least $1-2n^{-c}$. Here $c$ is some positive constant.
 \end{lemma}
 \begin{proof}
    Pick any $i\in [n]$ and assume $\sigma^2=1$ without loss of generality.
    By \cite[Theorem 7.30]{10.5555/2526243}, we have
    \begin{align}
        \P\left\{\left|\norm{\mb x_i}2^2-r\theta\right|\geq t\right\} \leq 2\exp\left(-\frac{t^2}{4r\theta+4t}\right).
    \end{align}
    Take $t = c\log n$ to obtain 
    \[
        \P\left\{
        \left|\norm{\mb x_i}2^2-r\theta\right| \le c'(\sqrt{r\theta \log n} + \log n)
        \right\} \ge 1 - 2n^{-c''}
    \]
    for some constants $c,c',c''>0$. Take the union bounds over $1\le i\le n$ to complete the proof. 
 \end{proof}

     \begin{lemma}\label{lem:upperbound_for_exp_f_value}
   Suppose $\mb x\in \R^r$ has i.i.d Bernoulli-Gaussian entries. Let $F_{\mb q}\paren{\mb x}$ be defined as equation (\ref{def:definition_of_f_value}). We have
   \begin{align}
       \sup_{\bq\in\Sp^{r-1}}\E\brac{\norm{F_{\mb q}\paren{\mb x}}2^2}\leq C\theta^{-1}
   \end{align}
   for some constant $C>0$. 
 \end{lemma}
 \begin{proof}
    Note that $\mb x=\mb b\circ \mb g$ where $\mb b \overset{i.i.d.}{\sim}\textrm{Ber}(\theta)$ and $\mb g\overset{i.i.d.}{\sim} \mathcal{N}(0,\sigma^2)$. We assume $\sigma=1$ for simplicity. Define $\mathcal{I} $ as the nonzero support of $\mb x$ such that we could write $\bx=\mathcal{P}_{\mathcal{I}}\left(\mb g\right)$. We have
 \begin{align}
     \bb E\left[\norm{F_{\mb q}\paren{\mb x}}2^2\right]=\left(\frac{1}{12\theta}\right)^2\bb E\left[\left(\mb q^{T}\mb x\right)^8\right]=\left(\frac{1}{12\theta}\right)^2\bb E\left[\innerprod{\mathcal{P}_{\mathcal{I}}\left(\mb q\right)}{\mb g}^8\right].
 \end{align}
 Since 
 \begin{align}\label{def:bound_exp_f1}
     \bb E\left[\innerprod{\mathcal{P}_{\mathcal{I}}\left(\mb q\right)}{\mb g}^8\right]=\left(7!!\right)\bb E_{\mathcal{I}}\left[\norm{\mathcal{P}_{\mathcal{I}}\left(\mb q\right)}2^8\right],
 \end{align}
 we further obtain
 \begin{align}\label{def:bound_exp_f2}
     \bb E_{\mathcal{I}}\left[\norm{\mathcal{P}_{\mathcal{I}}\left(\mb q\right)}2^8\right]=\sum_{k_1,k_2,k_3,k_4}q_{k_1}^{2}\mathbbm{1}_{k_1\in \mathcal{I}}q_{k_2}^{2}\mathbbm{1}_{k_2\in \mathcal{I}}q_{k_3}^{2}\mathbbm{1}_{k_3\in \mathcal{I}}q_{k_4}^{2}\mathbbm{1}_{k_4\in \mathcal{I}}.
 \end{align}
 We consider four scenarios.
 \begin{itemize}
     \item Only one index among $k_1,k_2,k_3,k_4$ in $\mathcal{I}$. 
     In this case we have 
     \begin{align}
         \bb E_{\mathcal{I}}\left[\norm{\mathcal{P}_{\mathcal{I}}\left(\mb q\right)}2^8\right]=\theta\sum_{k_1}q_{k_1}^8\leq\theta\norm{\mb q}2^8
     \end{align}
      \item Two index among $k_1,k_2,k_3,k_4$ in $\mathcal{I}$. 
     In this case we have 
     \begin{align}
         \bb E_{\mathcal{I}}\left[\norm{\mathcal{P}_{\mathcal{I}}\left(\mb q\right)}2^8\right]=\theta^2\sum_{k_1,k_2}\left[q_{k_1}^2q_{k_2}^6+q_{k_1}^4q_{k_2}^4+q_{k_1}^6q_{k_2}^2\right]\leq3\theta^2\norm{\mb q}2^8
     \end{align}
    \item Three index among $k_1,k_2,k_3,k_4$ in $\mathcal{I}$. 
     In this case we have 
     \begin{align}
         \bb E_{\mathcal{I}}\left[\norm{\mathcal{P}_{\mathcal{I}}\left(\mb q\right)}2^8\right]=\theta^3\sum_{k_1,k_2,k_3}\left[q_{k_1}^2q_{k_2}^2q_{k_3}^4\right]\leq\theta^3\norm{\mb q}2^8
     \end{align}
     \item All four index among $k_1,k_2,k_3,k_4$ in $\mathcal{I}$. 
     In this case we have 
     \begin{align}
         \bb E_{\mathcal{I}}\left[\norm{\mathcal{P}_{\mathcal{I}}\left(\mb q\right)}2^8\right]=\theta^4\sum_{k_1,k_2,k_3,k_4}\left[q_{k_1}^2q_{k_2}^2q_{k_3}^2q_{k_4}^2\right]\leq\theta^4\norm{\mb q}2^8
     \end{align}
 \end{itemize}
 Use $\|\bq\|_2 = 1$ and collect the above four results to obtain
 \begin{align}
     \bb E_{\mathcal{I}}\left[\norm{\mathcal{P}_{\mathcal{I}}\left(\mb q\right)}2^8\right]=\theta+3\theta^2+\theta^3+\theta^4\leq c_1\theta
 \end{align}
 Here $c_1>6$. Plugging back into (\ref{def:bound_exp_f1}) yields
 \begin{align}
      \bb E\left[\norm{f_{\mb q}\paren{\mb x}}2^2\right]\leq\left(7!!\right)\frac{\theta}{144\theta^2}\leq C\theta^{-1}
 \end{align}
 and finishes our proof.
 \end{proof}

  \bigskip
  
  The following results provide deviation inequalities of the average of i.i.d. functionals of a sub-Gaussian random vector / matrix. They are proved in \cite{zhang2018structured} and we offer a modified version here. 
  
  \begin{lemma}[Theorem  $F.1$, \cite{qu2019analysis}]
\label{lem:matrix concentration}
  Let $\mb Z_1$,$\mb Z_2$, \dots, $\mb Z_p$ be i.i.d  realizations of a random matrix $\mb Z \in \bb R^{n_1\times n_2}$ satisfying
  \begin{align}
      \E\brac{\mb Z}={\mb 0} \textrm{ , } \qquad \P\paren{|Z_{ij}|>t}\leq 2 \exp\paren{-\frac{t^2}{2\sigma^2}},\quad \forall 1\le i\le n_1, 1\le j\le n_2. 
  \end{align}
  For any fixed $\mb q\in \bb S^{n-1}$, define a function $f_{\mb q}:\bb R^{n_1\times n_2}\rightarrow \bb R^{d_1\times d_2}$ such that the following conditions hold.
  \begin{itemize}
      \item \textbf{Condition 1.}  There exists some positive numbers $B_f$ and $L_f$ such that 
      %\begin{align}\label{cond:2.1}
      %    \P\paren{\norm{f_{\mb q}\paren{\mb Z}}2>t}\leq 2 \exp\paren{-C\sqrt{t}},\quad \forall t>0.
      %\end{align}
      %\item \textbf{Condition 2.} 
      \begin{align}
          &\norm{\E\brac{f_{\mb q}\paren{\mb Z}}}{\rm op}\leq B_{f} \textrm{,}\\
          &\norm{\E\brac{f_{\mb q_1}\paren{\mb Z}}-\E\brac{f_{\mb q_2}\paren{\mb Z}}}{\rm op}\leq L_{f} \norm{\mb q_1-\mb q_2}2,\quad \forall\mb q_1, \mb q_2 \in \bb S^{n-1}.
      \end{align} 
      
      \item \textbf{Condition 2.} Define $\bar{\mb Z}$ as a truncated random matrix of $\mb Z$, such that
      \begin{align}
          \mb Z=\bar{\mb Z}+\Tilde{\mb Z} \textrm{,} \qquad \bar{Z}_{ij}=
          \begin{cases}
            Z_{ij}, \quad ~\textrm{if $|Z_{ij}|<B$ }\\
            0, \qquad \textrm{otherwise}
          \end{cases}
      \end{align}
      with $B=2\sigma\sqrt{\log\paren{pn_1n_2}}$. There exists some positive quantities $R_1, R_2$ and $\bar L_f$ such that
         \begin{align}
          &\norm{f_{\mb q}\paren{\bar{\mb Z}}}{\rm op}\leq R_1\textrm{,} \quad \max\left\{
          \norm{\E\left[f_{\mb q}\paren{\bar{\mb Z}}(f_{\mb q}\paren{\bar{\mb Z}})^T\right]}{\rm op}, \norm{\E\left[(f_{\mb q}\paren{\bar{\mb Z}})^Tf_{\mb q}\paren{\bar{\mb Z}}\right]}{\rm op}\right\}\leq R_2,\\
          &\norm{f_{\mb q_1}\paren{\bar{\mb Z}}-f_{\mb q_2}\paren{\bar{\mb Z}}}{\rm op}\leq\bar{L}_f\norm{\mb q_1-\mb q_2}2 \textrm{,} \qquad \forall \mb q_1, \mb q_2 \in \bb S^{n-1}.
         \end{align}
    \end{itemize}
    Then with probability greater than $1-\paren{n_1n_2p}^{-1}-cM^{-c'n}$ for some constants $c,c'>0$ and
     $M = (\bar L_f + L_f)(p+d_1+d_2)$, one has
     \begin{align}\nonumber
       \sup_{\mb q\in \Sp^{n-1} }\norm{\frac{1}{p}\sum_{i=1}^{p}f_{\mb q}\paren{\mb Z_i}-\E   \brac{f_{\mb q}\paren{\mb Z}}}{\rm op} \lesssim {(d_1 \wedge d_2) B_f \over \sqrt{n_1n_2} p} + 
       \sqrt{R_2 n \log (M) \over p} + {R_1n\log (M) \over p}.
     \end{align}
        % Then for any $\delta\in\paren{0,6R_2/R_1}$, if
        %  \begin{align}\label{ieq:low_bound_p_matrix}
        %      p\geq C^{'}\max\Brac{\frac{B_f}{n_1\delta}, ~ {R_2 \over \delta^2}\brac{n\log\paren{\frac{6\paren{L_f+\bar{L}_f}}{\delta}}+\log\paren{d_1}}},
        %  \end{align}
        %  with probability greater than $1-\paren{n_1n_2p}^{-1}-\exp^{-cn\log\paren{\paren{L_f+\bar{L_f}}/\delta}}$, there exists
        %  \begin{align}
        %   \sup_{\mb q\in \Sp^{n-1} }\norm{\frac{1}{p}\sum_{i=1}^{p}f_{\mb q}\paren{\mb Z_i}-\E   \brac{f_{\mb q}\paren{\mb Z}}}{\rm op}\leq \delta
        %  \end{align}
        %  Here $c$ and $C^{'}$ are some positive constants.
 \end{lemma}

 \bigskip

    \begin{lemma}[Corollary  $F.2$, \cite{qu2019analysis}]\label{lem:vector concentration}
  Let $\mb Z_1$,$\mb Z_2$, \dots, $\mb Z_p $ be i.i.d  realizations of a sub-Gaussian random vector $\mb Z\in \R^{n_1}$ satisfying
  \begin{align}
      \E\brac{\mb Z}=0 \textrm{ , } \qquad \P\paren{|Z_{j}|>t}\leq 2 \exp\paren{-\frac{t^2}{2\sigma^2}},\qquad \forall 1\le j\le n_1.
  \end{align}
  For any fixed $\mb q\in \Sp^{n-1}$, define a function $f_{\mb q}:\R^{n_1}\rightarrow \R^{d_1}$ satisfying the following conditions.
  \begin{itemize}
      %\item \textbf{Condition 1.} 
      %\begin{align}\label{cond:1}
      %    \P\paren{\|f_{\mb q}\paren{\mb Z}\|_2>t}\leq 2 \exp\paren{-C\sqrt{t}},\quad \forall t>0.
      %\end{align}
      \item \textbf{Condition 1.} For some positive numbers $B_f,L_f>0$, 
      \begin{align}
          &\norm{\E\brac{f_{\mb q}\paren{\mb Z}}}2\leq B_{f} \textrm{,}\\
          &\norm{\E\brac{f_{\mb q_1}\paren{\mb Z}}-\E\brac{f_{\mb q_2}\paren{\mb Z}}}2\leq L_{f} \norm{\mb q_1-\mb q_2}2,\qquad \forall \mb q_1,\mb q_2 \in \Sp^{n-1}.
      \end{align} 
      \item \textbf{Condition 2.} Let $\bar{\mb Z}$ be the truncated random vector of $\mb Z$, such that
      \begin{align}\label{def_Z_trunc}
          \mb Z=\bar{\mb Z}+\Tilde{\mb Z} \textrm{,} \qquad \bar{Z_j}=
          \begin{cases}
            Z_j, ~\quad\textrm{if $|Z_j|<B$}\\
            0, \qquad \textrm{otherwise}
          \end{cases}
      \end{align}
      with $B=2\sigma\sqrt{\log\paren{pn_1}}$. There exists some positive numbers $R_1, R_2$ and $\bar L_f$ such that 
         \begin{align}
          &\norm{f_{\mb q}\paren{\bar{\mb Z}}}2\leq R_1\textrm{,} \qquad \E\brac{\norm{f_{\mb q}\paren{\bar{\mb Z}}}2^2}\leq R_2,\\
        &\norm{f_{\mb q_1}\paren{\bar{\mb Z}}-f_{\mb q_2}\paren{\bar{\mb Z}}}2\leq \bar L_f\paren{\sigma}\norm{\mb q_1-\mb q_2}2 \textrm{,} \qquad \forall  \mb q_1,\mb q_2 \in \Sp^{n-1}.
         \end{align}
  \end{itemize}
  Then with probability greater than $1-\paren{n_1p}^{-1}-cM^{-c'n}$ for some constants $c,c'>0$ and
     $M = (\bar L_f + L_f)(p+d_1)$, one has
     \begin{align}\nonumber
       \sup_{\mb q\in \Sp^{n-1} }\norm{\frac{1}{p}\sum_{i=1}^{p}f_{\mb q}\paren{\mb Z_i}-\E   \brac{f_{\mb q}\paren{\mb Z}}}{\rm op} \lesssim {B_f \over \sqrt{n_1} p} + 
       \sqrt{R_2 n \log (M) \over p} + {R_1n\log (M) \over p}.
     \end{align}
%   Then for any $\delta\in\paren{0,6R_2/R_1}$, if
%          \begin{align}\label{ieq:low_bound_p_vector}
%              p\geq C\max\Brac{\frac{B_f}{n_1\delta}, ~ {R_2 \over \delta^2}\brac{n\log\paren{\frac{6\paren{L_f+\bar L_f}}{\delta}}+\log\paren{d_1}}},
%          \end{align}
%          then with probability greater than $1-\paren{n_1 p}^{-1}-\exp^{-cn\log\paren{(L_f+\bar{L}_f)/\delta}}$, one has
%          \begin{align}
%           \sup_{\mb q\in \Sp{n-1} }\norm{\frac{1}{p}\sum_{i=1}^{p}f_{\mb q}\paren{\mb z_i}- \E\brac{f_{\mb q}\paren{\mb Z}}}2\leq \delta.
%          \end{align}
%          Here $c$ and $C$ are some positive constants.
 \end{lemma}

 \bigskip

 Our analysis also uses the following lemmas that have already been established in the existing literature. 
 
 \begin{lemma}[Lemma 4, \cite{zhang2018structured}]\label{lem:Nozero}
        Let $\mb v \in \R^{n}$ contains i.i.d. Ber($\theta$) random variables. We have
        \begin{align}
            \P\brac{\|\mb v\|_0\geq\paren{1+t}n\theta}\leq2\exp\paren{-\frac{3t^2n\theta}{2t+6}}
        \end{align}
    \end{lemma}

     \begin{lemma}[Corollary  $F.5$, \cite{qu2019analysis}]\label{lem:upperbound_for_exp_f_grad}
   Suppose $\mb x$ is i.i.d Bernoulli-Gaussian random variables with parameter $(\theta, \sigma^2=1)$. For any $\bq \in \Sp^{r-1}$, let 
   \[
    f_{\mb q}(\bx) = \frac{1}{3\theta }\paren{\bq^{T}\mb x}^3\mb x.
   \]
   We have
   \begin{align}
       \sup_{\bq\in\Sp^{r-1}}\E\brac{\norm{f_{\mb q}\paren{\mb x}}2^2}\leq C\theta^{-1} r,
   \end{align}
   for some  constant $C>0$.
 \end{lemma}

        \begin{lemma}[Corollary  $F.7$, \cite{qu2019analysis}]\label{lem:upperbound_for_exp_f_hess}
   Suppose $\mb x \in \R^r$ contains i.i.d Bernoulli-Gaussian random variables with parameter $(\theta, \sigma^2=1)$. For any $\bq\in \Sp^{r-1}$, let 
   \[
    f_{\mb q}(\bx) = \frac{1}{\theta}\paren{\bq^{T}\mb x}^2\mb x\mb x^{T}.
   \]
   We have
   \begin{align}
       \sup_{\bq\in\Sp^{r-1}}\left\|\E\brac{f_{\mb q}\paren{\mb x}(f_{\mb q}\paren{\mb x})^T}\right\|_{\op}\leq C\theta^{-1} r^2,
   \end{align}
   for some  constant $C>0$.
 \end{lemma}

 \newpage

\end{appendices}

\end{document}